\documentclass[11pt, a4paper]{amsart}

\usepackage{xcolor}
\definecolor{MyBlue}{cmyk}{1,0.13,0,0.63}
\definecolor{MyGreen}{cmyk}{0.91,0,0.88,0.52}
\newcommand{\mylinkcolor}{MyBlue}
\newcommand{\mycitecolor}{MyGreen}
\newcommand{\myurlcolor}{black}

\usepackage{hyperref}
\hypersetup{%
  bookmarksnumbered=true,bookmarksopen=false,%
  plainpages=false,
  linktocpage=true,%
  colorlinks=true,breaklinks=true,%
  linkcolor=\mylinkcolor,citecolor=\mycitecolor,urlcolor=\myurlcolor,%
  pdfpagelayout=OneColumn,%
  pageanchor=true,%
}

\usepackage{graphicx}
\usepackage{amsmath, amsthm,  amsfonts, amscd}
\usepackage{amssymb}
\usepackage{url}
\usepackage{fullpage}
\usepackage{mathtools}
\usepackage[all]{xy}
\usepackage{slashed}
\usepackage{multirow}

\newtheorem{thm}{Theorem}[section]
\newtheorem*{thm*}{Theorem}
\newtheorem{cor}[thm]{Corollary}
\newtheorem{lemma}[thm]{Lemma}
\newtheorem{prop}[thm]{Proposition}

\theoremstyle{definition}
\newtheorem{defn}[thm]{Definition}

\theoremstyle{remark}
\newtheorem{remark}[thm]{Remark}

\newtheorem{example}[thm]{Example}
\newtheorem{remarks}[thm]{Remarks}

\newcommand{\End}{\ensuremath{\mathrm{End}}}

\newcommand{\R}{\ensuremath{\mathbb{R}}}
\newcommand{\N}{\ensuremath{\mathbb{N}}}
\newcommand{\Z}{\ensuremath{\mathbb{Z}}}

\newcommand{\C}{\ensuremath{\mathbb{C}}}
\newcommand{\T}{\ensuremath{\mathbb{T}}}

\def\calT{\mathcal{T}}
\def\calC{\mathcal{C}}
\def\calL{\mathcal{L}}
\def\calO{\mathcal{O}}
\def\calK{\mathcal{K}}
\def\calB{\mathcal{B}}
\def\calH{\mathcal{H}}

\def\calF{\mathcal{F}}
\def\calA{\mathcal{A}}
\def\calN{\mathcal{N}}

\def\calM{\mathcal{M}}

\def\calQ{\mathcal{Q}}
\def\calV{\mathcal{V}}
\def\calW{\mathcal{W}}
\def\calU{\mathcal{U}}

\def\calG{\mathcal{G}}

\def\bP{\mathbf{P}}

\newcommand{\ol}{\overline}
\DeclareMathOperator{\Aut}{Aut}

\theoremstyle{definition}

\DeclareMathOperator{\Dom}{Dom}

\DeclareMathOperator{\Index}{Index}

\DeclareMathOperator{\Ker}{Ker}

\DeclareMathOperator{\Ran}{Ran}
\DeclareMathOperator{\Tr}{Tr}
\DeclareMathOperator*{\res}{res}

\newcommand{\rst}[1]{\ensuremath{{\mathbin\upharpoonright}%
\raise-.5ex\hbox{$#1$}}}

\makeatletter

\newcommand{\Rmnum}[1]{\expandafter\@sl217--242owromancap\romannumeral #1@}
\makeatother

\author{C. Bourne}
\address{WPI-Advanced Institute for Materials Research (WPI-AIMR), Tohoku University,
2-1-1 Katahira, Aoba-ku, Sendai, 980-8577, Japan \emph{and} 
iTHEMS Research Group, RIKEN, 2-1 Hirosawa, Wako, Saitama 351-0198, Japan}
\email{chris.bourne@tohoku.ac.jp}

\author{B. Mesland}
\address{Max Planck Institut f\"{u}r Mathematik, Vivatsgasse 7 D-53111, Bonn, Germany}
\email{brammesland@gmail.com}

\date{\today}

\begin{document}

\begin{abstract}
We examine the noncommutative index theory associated to the dynamics of a Delone set 
and the corresponding transversal groupoid. Our main motivation comes from the application to 
topological phases of aperiodic lattices and materials, and applies to invariants from tilings as well. 
Our discussion concerns semifinite index pairings, factorisation properties of 
Kasparov modules and the construction of unbounded Fredholm modules for lattices with  
finite local complexity.
\end{abstract}

\title{Index theory and topological phases of aperiodic lattices}
\maketitle

\tableofcontents

\section*{Introduction}

Models of systems in solid state physics that do not make reference to a Bloch decomposition 
or Fourier transform are essential if one wishes to understand topological phases of 
disordered or aperiodic systems. A description of disordered media using crossed 
product $C^*$-algebras has
successfully adapted many important properties of periodic topological insulators 
to the disordered setting, see~\cite{Bellissard94, PSBbook, HMT} for example.

Recently, newer proposed models 
of topological materials and meta-materials have emerged whose underlying lattice has a quasicrystalline~\cite{quasiphase}
or amorphous configuration~\cite{GyroInsulator}. 
In the case of amorphous lattice configurations, because there is no canonical 
labelling of the lattice points, the Hamiltonians of interest can not 
be described by a crossed product $C^*$-algebra. Hence the techniques and 
results on the bulk-boundary correspondence in~\cite{BKR1} can not be applied to 
recent results on edge states and transport in 
topological amorphous (meta-)materials~\cite{GyroInsulator}. Such amorphous 
systems are instead modelled by the transversal  (\'{e}tale) groupoid $\calG$ associated to a Delone set developed 
in~\cite{Bel86,Kellendonk95,Kellendonk97,BHZ00}. One of the key 
results of the paper is the extension of the $K$-theoretic framework for topological phases 
to such algebras and aperiodic media.

Quasicrystalline materials often display finite local complexity, meaning that up to translation the lattice is determined by 
a finite number of patterns or polyhedra (cf. Definition \ref{def:FLC_et_al}).
If the lattice has finite local complexity, the aperiodic 
but ordered pattern configurations can be described using tiling spaces. By~\cite[Theorem 2]{SW03}, 
the tiling space of such  lattices  is homeomorphic to the $d$-fold 
suspension of a $\Z^d$-subshift space, though not necessarily topologically conjugate. 
This result implies that, adding some extra sites to quasicrystalline  lattice configurations if necessary, there 
is a $\Z^d$-labelling of points and the system can be described by a 
discrete crossed product $C(Z)\rtimes \Z^d$.

The advantage of the transversal groupoid approach is that it does not require finite local complexity nor a $\Z^d$-labelling. 
In particular, the framework can accommodate non-periodic $\R^d$-actions. Thus,  
the modifications needed to obtain  a $\Z^d$-labelling (which may not be physically reasonable) can be avoided. 
 Furthermore, the transversal groupoid covers  
a broader range of examples which are not covered by the results in~\cite{SW03} such as the pinwheel tiling.

 Given a Hamiltonian on an aperiodic system, computational techniques are 
currently in development to determine its spectrum~\cite{BBD18}. 
If the Hamiltonian contains a spectral gap, we can associate a topological phase 
to the system by modelling its dynamics via the transversal groupoid $\calG$. 
In particular, topological indices and $K$-theoretic properties of such Hamiltonians 
are determined using the groupoid $C^*$-algebra.

In previous work, this groupoid description was used to describe bulk topological 
phases~\cite{BP17}.
In this paper, we show that a (gapped) 
Hamiltonian acting on a Delone set $\calL \subset \R^d$
is enough to define strong and weak topological phases as well as show the bulk-edge correspondence of Hamiltonians acting 
on the lattice $\ell^2(\calL)$. Furthermore, if the unit space $\Omega_0$ of the transversal groupoid 
$\calG$ has an invariant measure, then Chern number formulas can also be defined for complex 
topological phases.

Because of the generality of Delone sets, they are able to model materials that go beyond what is 
normally considered when discussing topological phases of matter. These include quasicrystal 
structures but also other materials such as glasses and some liquids, see~\cite{BellissardDelone} 
for example. This tells us, at least from a mathematical perspective, that our constructions 
and results are potentially applicable to a broader range of materials and meta-materials in addition to 
the applications to (possibly disordered) crystals.

In the present paper, our central object of study is an unbounded $KK$-cycle for the transversal groupoid $C^*$-algebra 
which gives rise to a class in $KK^d(C^*_r(\calG,\sigma),C(\Omega_0))$ (real or complex) with $d$ 
the dimension of the underlying space, $\sigma$ a magnetic twisting and $\Omega_0$ 
the transversal space. When this 
$KK$-cycle is coupled with the $K$-theoretic phase of a gapped free-fermionic Hamiltonian 
(which gives a class in $K_n(C^*_r(\calG,\sigma))$), 
the corresponding index pairing gives analytic indices 
that encode the strong topological phase. When the transversal $\Omega_0$ has an invariant 
probability measure, we can construct a semifinite spectral triple and measure this (disorder-averaged) 
pairing using the semifinite local index formula (considered for ergodic measures in~\cite{BP17}). 

The factorisation properties of the unbounded $KK$-cycle also allow us to express the index pairing 
as a pairing over a closed subgroupoid $\Upsilon$ that encodes the dynamics of the transversal 
in $(d-1)$-dimensions and models an edge system. Namely, we can link these systems explicitly 
via a short exact sequence
$$
  0 \to C^*_r(\Upsilon,\sigma)\otimes \mathbb{K} \to \calT \to C^*_r(\calG,\sigma) \to 0
$$
with $\calT$ modelling a half-space system. When the lattices we consider have a canonical 
$\Z^d$-labelling, then this short exact sequence is the usual Toeplitz extension of a crossed product 
considered in~\cite{PSBbook}. Our result, analogous to the crossed-product case, \cite{PSBbook, BKR1}, 
is that our $d$-dimensional pairing with $C^*_r(\calG,\sigma)$ (as an element in the $K$-theory of 
the configuration space or a numerical phase label of this pairing) is equal to or in the same $K$-theory 
class  as the $(d-1)$-dimensional pairing with $C^*_r(\Upsilon,\sigma)$ up to a possible sign.

For aperiodic lattices with finite local complexity, the transversal $\Omega_0$ is totally disconnected. 
In this case a general construction due to Pearson and 
Bellissard ~\cite{PearsonBellissard} gives a family of spectral triples on $C(\Omega_0)$.
Coupling the unbounded $KK$-cycle for $(C^*_r(\calG),C(\Omega_0))$ to such a spectral triple for $C(\Omega_0)$ 
using the unbounded Kasparov product, gives us $K$-homology representatives for $C^{*}_{r}(\mathcal{G})$. The 
construction of the product operator employs the techniques developed in~\cite{MR}, but the commutators with
$C^{*}_{r}(\mathcal{G})$ turn out to be unbounded. Nonetheless, using arguments similar to~\cite{GM15} and recent results in~\cite{MesLesch},
we are able to show that the operator represents the Kasparov product of the given classes via the bounded transform.
The analytic difficulties with the commutators can be directly attributed to the disorder, that is, the nonperiodicity of the Delone set.

Let us remark that the unbounded Fredholm module constructed from quasicrystalline 
lattice configurations allows us to consider new topological phases that can not be defined 
in periodic systems or disordered systems with a contractible disorder space of configurations. 
Indeed, the totally disconnected structure of the transversal $\Omega_0$ is a crucial ingredient 
in defining these new phases.

Some of our results show parallels with those of Kubota and of  
Ewert--Meyer, 
who study topological phases associated to Delone sets and the corresponding 
Roe algeba~\cite{EM18,Kubota15b}. 
Briefly, the Roe algebra, by its universal nature, provides a means to compare topological 
phases from  different 
lattice configurations (see~\cite[Lemma 2.19]{Kubota15b}). Conversely, the transversal 
groupoid algebra is used to determine the 
topological phase of Hamiltonians associated to a fixed lattice configuration. 
Because the groupoid algebra is separable (while the Roe algebra is not), 
it is more susceptible to the use of $KK$-theoretic machinery, which is a central theme of this paper.
In particular, it is generally easier to both 
define and compute the pairings with $KK$-cycles or cyclic cocycles 
that characterise the numerical phase labels; see~\cite[Section 3]{BP17} for 
numerical simulations.

Lastly, the groupoid of a transversal is typically used to study the dynamics of aperiodic tilings 
and related dynamical systems. We have not emphasised the application to tilings 
in this manuscript, though our 
constructions and results may have broader interest.

\subsection*{Outline}
Because our paper draws from aspects of dynamical systems, operator algebras, 
Kasparov theory and physics, we aim to give a systematic and largely self-contained 
exposition of our results.

We first give a brief overview of the mathematical tools we require in Section \ref{sec:KK_Gpoid_prelim}, which 
include Kasparov theory, semifinite index theory and $C^*$-algebras of \'{e}tale 
groupoids twisted by a $2$-cocycle. 
We consider $C^*$-modules constructed from \'{e}tale groupoids and review how 
groupoid equivalences can be naturally expressed in terms of 
$C^*$-modules. In particular, we consider groupoids with a 
normalised  $2$-cocycle, where groupoid equivalence 
for compatible twists  gives rise to a 
Morita equivalence of the twisted groupoid $C^*$-algebras. 
We also provide a higher dimensional extension of the result in~\cite{MeslandGpoid}, 
where if one has a continuous $1$-cocycle $c:\calG\to\R^n$ that is exact 
in the sense of~\cite{MeslandGpoid}, 
then this cocycle gives rise to a Dirac-like operator and unbounded Kasparov 
module over the twisted $C^{*}$-algebra of $\mathcal{G}$ relative to that of a closed subgroupoid $\calH=\Ker(c)$. 
We further provide a condition on $\R$-valued cocycles that 
guarantees injectivity of the Busby invariant directly. This condition is satisfied in all 
examples considered in the paper.

In Section \ref{sec:Transversal}, we review the construction of the transversal groupoid 
following~\cite{Bel86,Kellendonk95,BHZ00,KellendonkPutnam} 
 and show how it fits into our general $KK$-theoretic framework. 
 In the case of dimension 1, we give an alternative description of the groupoid $C^*$-algebra and 
 unbounded $KK$-cycle using Cuntz--Pimsner algebras and results from~\cite{RRSgpoid, RRS, GMR}.

In Section \ref{sec:factorisation} we show how the unbounded $KK$-cycle we build 
factorises into the product of an `edge' $KK$-cycle modelling a system of codimension $1$ with a 
linking $KK$-cycle that relates the two systems. This can also be extended to higher 
codimension and is related to weak topological insulators.

We then consider spectral triple constructions in Sections \ref{sec:bulk_spectral_triples} and 
\ref{sec:PB_product}. We construct spectral triples using 
the evaluation map of the transversal, an invariant measure (which gives a semifinite 
spectral triple) and the product with a Pearson--Bellissard spectral triple. The latter construction yields an unbounded Fredholm module with
mildly unbounded commutators as in \cite{GM15}, so that the bounded transform represents the Kasparov product.

Lastly, we apply our results to topological phases in Section \ref{sec:Applications}, where the physical 
invariants of interest 
naturally arise as index pairings of classes in $K_n(C^*_r(\calG,\sigma))$ with 
our unbounded $KK$-cycles (or spectral triples). Here we prove Chern number formulas for 
complex phases, analytic strong and weak indices for systems with anti-linear symmetries and the bulk-boundary 
correspondence. Much like the crossed product setting, our bulk 
indices are also well-defined for a much larger algebra that can be constructed using 
noncommutative $L^p$-spaces. A connection  of these extended indices 
to regions of dynamical or spectral localisation remains an open problem.

\subsubsection*{Acknowledgements}
We thank Jean Bellissard, Magnus Goffeng, Johannes Kellendonk, Aidan Sims and Makoto Yamashita for helpful discussions. We thank the anonymous referees for their careful reading of the manuscript and valuable feedback.

CB was supported by a postdoctoral fellowship for overseas researchers from The Japan Society for the 
Promotion of Science (No. P16728) and both authors were 
supported by a KAKENHI Grant-in-Aid for JSPS fellows (No. 16F16728). This work is also supported by World Premier International Research Center Initiative (WPI), MEXT, Japan.
BM gratefully
acknowledges support from the Hausdorff Center for Mathematics and the Max Planck Institute for Mathematics in Bonn, Germany, as well as Tohoku University, Sendai, Japan for its hospitality. 
 
Part of this work was carried out during the Lorentz Center program \emph{KK-theory, Gauge Theory and 
Topological Phases} held in Leiden, Netherlands in March 2017. We also thank the Leibniz Universit\"{a}t Hannover, 
Germany, the Radboud University Nijmegen, Netherlands and 
the Erwin Schr\"{o}dinger Institute, University of Vienna, Austria  for hospitality.

\section{Preliminaries on groupoids and Kasparov theory} \label{sec:KK_Gpoid_prelim}

\subsection{Kasparov modules and semifinite spectral triples}
In this section we establish basic results and notation that we will need 
for this paper. Because we are motivated by topological phases whose 
relation to real $K$-theory is now well-established~\cite{FM13,GSB16,Kellendonk15}, we will work in both 
real and complex vector spaces and algebras.

Given a real or complex right-$B$ $C^*$-module $E_B$, we denote the 
right action by $e\cdot b$ and the $B$-valued inner product $(\cdot\mid\cdot)_B$. 
The set of adjointable endomorphisms on $E_B$ with respect to this inner 
product is denoted $\End^*(E_B)$.  
The rank-$1$ operators $\Theta_{e,f}$, $e,f\in E_B$, are defined such that
$$
   \Theta_{e,f}(g) = e\cdot (f\mid g)_B, \qquad e,f,g\in E_B.
$$
The norm-closure of the algebraic span of the set of such rank-$1$ operators are the 
compact operators on $E_B$ and we denote this set by $\mathbb{K}(E_B)$.
We will often work with $\Z_2$-graded algebras and 
spaces and denote by $\hat\otimes$ the graded tensor product (see~\cite[Section 2]{Kasparov80}
and \cite[Section 14]{Blackadar}). A densely defined closed symmetric operator $T:\Dom T\to E_{B}$ 
is \emph{self-adjoint and regular} if the operators $T\pm i:\Dom D\to E_{B}$ have dense range. 
See~\cite[Chapter 9-10]{Lance} for the basic theory 
of unbounded operators on $C^*$-modules.

\begin{defn}
Let $A$ and $B$ be $\Z_2$-graded real (resp. complex) $C^*$-algebras.
A real (resp. complex) unbounded Kasparov module $(\calA, {}_\pi{E}_B, D)$ 
(also called an unbounded $KK$-cycle) for $(A,B)$ consists of
\begin{enumerate}
\item a
$\Z_2$-graded real (resp. complex) $C^*$-module ${E}_B$, 
\item a graded $*$-homomorphism $\pi:A \to \End^*(E_B)$, 
\item an
unbounded self-adjoint, regular and odd operator $D$ and a dense $*$-subalgebra $\calA\subset A$ such that for all $a\in \calA\subset A$, 
\begin{align*}
  & [D,\pi(a)]_\pm \,\in\, \End^*(E_B)\;,   &&\pi(a)(1+D^2)^{-1}\,\in\, \mathbb{K}(E_B)\;.
\end{align*}
\end{enumerate}
For complex algebras and spaces, one can also remove the gradings, in 
which case the Kasparov module is called odd (otherwise even).
\end{defn}
We will often omit the representation $\pi$ when the left-action is unambiguous.
Unbounded Kasparov modules represent classes in the $KK$-group 
$KK(A,B)$ or $KKO(A,B)$~\cite{BJ83}. 
We note that an unbounded $A$-$\C$ or $A$-$\R$ Kasparov module is 
precisely the definition of a complex or real spectral triple.

Another noncommutative extension of index theory and closely related to unbounded 
Kasparov theory are semifinite spectral 
triples~\cite{CPRS2,CPRS3}.
Let $\tau$ be a fixed faithful, normal, semifinite trace on a von Neumann algebra 
$\calN$. 
We denote by $\calK_\calN$ the $\tau$-compact
operators in $\calN$, that is, the norm closed ideal generated by the 
projections $P\in\calN$ with $\tau(P)<\infty$. 
\begin{defn}
Let $\calN\subset \calB(\calH)$ be a graded semifinite von Neumann algebra with trace $\tau$.
A semifinite spectral triple $(\calA,\calH,D)$ is given by a $\Z_2$-graded 
Hilbert space $\calH$, 
a graded $\ast$-algebra $\calA\subset\calN$ with $C^*$-closure $A$ and 
a graded representation on 
$\calH$, together with a densely defined odd 
unbounded self-adjoint operator $D$ affiliated to $\calN$ such that
\begin{enumerate}
  \item $[D,a]_\pm$ is well-defined on $\Dom(D)$ and extends to a bounded operator on $\calH$ for 
  all $a\in\calA$,
  \item $a(1+D^2)^{-1}\in\calK_\calN$ for all $a\in A$.
\end{enumerate}
\end{defn}

For $\calN=\calB(\calH)$ and $\tau = \Tr$, one recovers the 
usual definition of a spectral triple. 
A semifinite spectral triple relative to $(\calN,\tau)$ with $\calA$ unital is called $p$-summable 
if $(1+D^2)^{-s/2}$ is $\tau$-trace class 
for all $s>p$. We also call a semifinite spectral triple $QC^\infty$ if 
$a,[D,a] \in \Dom(\delta^k)$ for all $k\in \N$ with $\delta(T) = [|D|,T]$ being 
the partial derivation. 

Semifinite spectral triples can be paired with $K$-theory classes in 
$K_\ast(\calA)$ via a semifinite Fredholm index~\cite{BCPRSW}. 
An operator 
$T\in \calN$ that is invertible modulo $\calK_\calN$ has a semifinite Fredholm 
index
$$
   \Index_\tau(T) = \tau(P_{\Ker(T)}) - \tau(P_{\Ker(T^*)}).
$$
If the semifinite spectral triples are $p$-summable and $QC^\infty$, 
the complex index pairing can be computed using 
the resolvent cocycle and the semifinite local index formula~\cite{CPRS2,CPRS3}. By writing the 
index pairing as a pairing with cyclic cohomology, the 
topological invariants of interest can more easily be connected 
to physics~\cite{PSBbook}. See~\cite{FK,BCPRSW,CPRS3} for further details on 
semifinite index theory and~\cite{Prodanbook} for results concerning 
numerical implementation.

Suppose $(\calA,E_B,D)$ is an unbounded Kasparov module for a separable $C^{*}$-algebra $A$ and the right-hand algebra 
$B$ has a faithful, semifinite and norm lower semicontinuous
trace $\tau_B$. We work with faithful traces as 
we can always pass to a quotient algebra $B/\Ker(\tau_B)$ if necessary. 
Assuming such a trace, one can often construct a semifinite spectral 
triple via a dual trace construction~\cite{LN04}. 
 We follow this approach in 
Section \ref{sec:semifinite_construction}. 
By constructing a semifinite spectral triple from a Kasparov module, we 
obtain a $KK$-theoretic interpretation of the semifinite index pairing, 
which can be expressed via the map
\begin{equation} \label{eq:semifinite_K_pairing}
  K_\ast(A) \times KK^\ast (A, B) \to K_0(B) \xrightarrow{(\tau_B)_\ast} \R,
\end{equation}
with $(\calA,E_B,D)$ representing the class in $KK^\ast (A, B)$. 
Equation \eqref{eq:semifinite_K_pairing} allows us to more explicitly characterise 
the range of the semifinite index pairing (which is in general $\R$-valued). 
The local index formula then gives us a computable expression for the 
$KK$-theoretic composition in Equation \eqref{eq:semifinite_K_pairing}.

\subsection{\'{E}tale groupoids, twisted algebras and $C^*$-modules} \label{sec:Gpoidmodules_general}

We start with some basic definitions for convenience. Our standard reference for groupoid $C^{*}$-algebras is \cite{Renault80}.
\begin{defn}
A groupoid is a set $\calG$ with a subset $\calG^{(2)} \subset \calG\times\calG$, 
a multiplication map $\calG^{(2)}\to\calG$, $(\gamma,\xi)\mapsto \gamma\xi$ 
and an inverse $\calG \to \calG$ $\gamma\mapsto \gamma^{-1}$ such that
\begin{enumerate}
  \item $(\gamma^{-1})^{-1}=\gamma$ for all $\gamma\in\calG$, 
  \item if $(\gamma,\xi), (\xi,\eta)\in \calG^{(2)}$, then $(\gamma\xi,\eta), (\gamma,\xi\eta)\in\calG^{(2)}$, 
  \item $(\gamma,\gamma^{-1})\in\calG^{(2)}$ for all $\gamma\in\calG$, 
  \item for all $(\gamma,\xi)\in\calG^{(2)}$, $(\gamma\xi)\xi^{-1}=\gamma$ and $\gamma^{-1}(\gamma\xi) = \xi$.
\end{enumerate}
\end{defn}
Given a groupoid we denote by $\calG^{(0)}= \{\gamma \gamma^{-1}\,:\, \gamma\in\calG\}$ the space of 
units and define the source and range maps $r,s:\calG \to\calG^{(0)}$ by the equations
\begin{align*}
  r(\gamma) = \gamma\gamma^{-1},  &&s(\gamma) = \gamma^{-1}\gamma
\end{align*}
for all $\gamma\in\calG$. The source and range maps allow us to characterise
$$
   \calG^{(2)} = \big\{ (\gamma,\xi) \in \calG\times\calG \,:\, s(\gamma) = r(\xi) \big\}.
$$

We furthermore assume that $\calG$ has a locally compact 
topology such that $\calG^{(0)} \subset \calG$ is Hausdorff in the relative topology and multiplication, 
inversion, source and range maps all continuous. In this work we restrict ourselves to groupoids 
that are both Hausdorff and \'{e}tale.
\begin{defn}
A topological groupoid $\calG$ is called \'{e}tale if the range map $r:\calG\to\calG$ is a local homeomorphism.
\end{defn}

\begin{defn}
Let $\calG$ be a locally compact and Hausdorff groupoid. A continuous 
map $\sigma:\calG^{(2)}\to \T\simeq U(1)$ is a $2$-cocycle if 
$$
  \sigma(\gamma,\xi) \sigma(\gamma\xi,\eta) 
    = \sigma(\gamma, \xi\eta) \sigma(\xi,\eta)
$$
for any $(\gamma,\xi),(\xi,\eta)\in\calG^{(2)}$, 
and
$$
  \sigma(\gamma, s(\gamma)) = 1 = \sigma(r(\gamma),\gamma)
$$
for all $\gamma\in\calG$. We will call a groupoid $2$-cocyle normalised 
if $\sigma(\gamma,\gamma^{-1})=1$ for all $\gamma\in\calG$.
\end{defn}

\begin{remark}
We can also define $O(1) \simeq \Z_2$-valued groupoid $2$-cocycles 
whose cocycle relation is the same as the $U(1)$ case. 
Generally speaking, if we are working in the category of complex spaces 
and algebras, we will use  $U(1)$-valued $2$-cocycles. If we are in the 
real category, we work with $O(1)$-valued $2$-cocycles.

Because the algebraic 
structure is the same in either setting, we will abuse notation slightly and 
denote by $\sigma$ a generic groupoid $2$-cocycle, where the 
range of this $2$-cocycle will be clear from the context.
\end{remark}

Given an \'{e}tale groupoid $\calG$ and $2$-cocycle $\sigma$, we define $C_c(\calG,\sigma)$ 
to be the $*$-algebra of compactly supported functions 
on $\calG$ with twisted convolution and involution 
\begin{align*}
   &(f_1\ast f_2)(\gamma) = \sum_{\gamma=\xi\eta}f_1(\xi)f_2(\eta) \sigma(\xi,\eta), 
   &&f^*(\gamma) = \sigma(\gamma,\gamma^{-1}) \ol{f(\gamma^{-1})}.
\end{align*}
The $2$-cocycle condition ensures that $C_c(\calG,\sigma)$ is an associative 
$\ast$-algebra. 
In the present paper, we restrict ourselves to considering 
{normalised} cocycles, which covers all examples of interest to us. 
Our definition of the groupoid $2$-cocyle and twisted convolution algebra 
comes from Renault~\cite{Renault80}. For a broader version of twisted groupoid algebra, 
see~\cite{KumjianDiagonal}.

\subsection{The $C^*$-module of a groupoid and the reduced twisted $C^{*}$-algebra}
\label{sec:twisted_gpoid_alg}

Take an \'{e}tale groupoid $\mathcal{G}$ with a normalised $2$-cocycle $\sigma$. 
The space $C_{c}(\mathcal{G},\sigma)$ is a right module over $C_{0}(\calG^{(0)})$ via $(f\cdot g)(\xi)=f(\xi)g(s(\xi))$. 
Since $\calG^{(0)}\subset\mathcal{G}$ is closed, we can consider the restriction map 
$\rho:C_c(\mathcal{G})\to C_0(\calG^{(0)})$. This defines a $C_0(\calG^{(0)})$ valued inner product on the right module 
$C_{c}(\mathcal{G},\sigma)$ via
\begin{align*}
 ( f_1\mid f_2 )_{C_0(\calG^{(0)})}(x):&=\rho(f_1^{*}*f_2)(x)\\
 &=\sum_{\xi \in s^{-1}(x)}\overline{f_1(\xi^{-1})}f_2(\xi^{-1}) \sigma(\xi,\xi^{-1})
     =\sum_{\xi\in r^{-1}(x)}\overline{f_1(\xi)}f_2(\xi)
\end{align*}
as $\sigma$ is normalised.
Denote by $E_{C_0(\mathcal{G}^{(0)})}$ the $C^{*}$-module completion of $C_{c}(\mathcal{G})$ in this inner product. There is an action of the $*$-algebra $C_c(\calG,\sigma)$ on the 
$C^*$-module $E_{C_0(\calG^{(0)})}$ by bounded adjointable endomorphisms, extending the action of  $C_c(\calG,\sigma)$ on itself by left-multiplication. 
\begin{defn}[cf.~\cite{KS02}]
The reduced groupoid $C^*$-algebra $C^*_r(\calG,\sigma)$ is the completion of 
$C_c(\calG,\sigma)$ in the norm inherited from the embedding 
$C_c(\calG,\sigma) \hookrightarrow \End^*(E_{C_0(\mathcal{G}^{(0)})})$.
\end{defn}

\begin{defn} \label{def:s_cover}
Let $\mathcal{G}$ be an \'{e}tale groupoid over $\calG^{(0)}$. An \emph{$s$-cover of $\mathcal{G}$} is a locally finite
countable open cover $\mathcal{V}:=\{V_{i}\}_{i\in\mathbb{N}}$ consisiting of pre-compact sets, 
such that $s:{V}_{i}\to \calG^{(0)}$ is a homeomorphism onto its image.
\end{defn}

\begin{lemma} 
Let $\mathcal{V}:=\{V_{i}\}_{i\in\mathbb{N}}$ be an $s$-cover of $\mathcal{G}$ and $\chi_{i}:V_{i}\to\mathbb{R}$ 
a partition of unity subordinate to $\mathcal{V}$, that is, $\sum_{i}\chi_{i}(\eta)^{2}=1$ for all $\eta\in\mathcal{G}$. 
Write $u_{m}:=\sum_{i\leq m} \Theta_{\chi_i,\chi_i} $. Then for all $f\in C_{c}(\mathcal{G},\sigma)$ there exists $N\in\mathbb{N}$ such that for all $n\geq N$
\[
f(\eta)=u_{n}f(\eta)=\sum_{i\leq n}\chi_{i} * \rho(\chi_{i}^{*}*f)(\eta).
\]
In particular $u_{m}f$ converges to $f$ in the norm of $E_{C_0(\mathcal{G}^{(0)})}$.
\end{lemma}
\begin{proof}The above, together with the fact that we have an $s$-cover gives
\begin{align*}
\sum_{i} \big(\chi_{i} \cdot \rho(\chi_{i}^{*}*f) \big)(\eta)&=\sum_{i}\chi_{i}(\eta) \rho(\chi_{i}^{*}*f)(s(\eta))
  =\sum_{i}\sum_{\xi\in r^{-1}(s(\eta))}\chi_{i}(\eta)\chi_{i}^{*}(\xi)f(\xi^{-1}) \sigma(\xi,\xi^{-1})  \\
&=\sum_{i}\sum_{\xi\in r^{-1}(s(\eta))}\chi_{i}(\eta)\chi_{i}(\xi^{-1})f(\xi^{-1})=\sum_{i}\sum_{\xi\in s^{-1}(s(\eta))}\chi_{i}(\eta)\chi_{i}(\xi)f(\xi)\\
&=\sum_{\{i:\eta\in V_{i}\}}\chi_{i}^{2}(\eta)f(\eta)=f(\eta).
\end{align*}
Since $f$ has compact support, there exists $N=N_f$ such that $\chi_{n}|_{\textnormal{supp} f}=0$  for all $n\geq N$. 
Thus the sum above is uniformly finite and hence convergent in the $\rho$-norm.
\end{proof}
Note that the above result implies that $u_n$ is a sequence of local units for $C_{c}(\mathcal{G}^{(2)})\subset\mathbb{K}(E_{C_0(\calG^{(0)})})$. 

\begin{lemma} We have 
$\sup_{n}\left\|u_{n}\right\|_{\End^{*}(E_{C_0(\calG^{(0)})})}\leq 1.$
\end{lemma}
\begin{proof} We compute the operator  norm of the $u_{n}$ directly. Let $f\in C_{c}(\mathcal{G},\sigma)$: 
\begin{align*}
 (u_n f \mid  u_n f)_{C(\calG^{(0)})} (x) &=\sum_{\xi\in r^{-1}(x)}|u_{n}f(\xi)|^{2}\\
&=\sum_{\xi\in r^{-1}(x)}\left|\sum_{i\leq n} \chi_{i}(\xi)^{2}f(\xi)\right|^{2}\leq\sum_{\xi\in r^{-1}(x)}\left(\sum_{i\leq n} \chi_{i}(\xi)^{2}|f(\xi)|\right)^{2}\\
&\leq \sum_{\xi\in r^{-1}(x)}\left|f(\xi)\right |^{2}= (f\mid f)_{C(\calG^{(0)})} (x)
\end{align*}
Thus it follows that
\[
\|u_{n}f\|^{2}_{C_{0}(\calG^{(0)})}=\sup_{x\in \calG^{(0)}} (u_n f \mid u_n f)_{C(\calG^{(0)})}(x)
     \leq \sup_{x\in \calG^{(0)}} (f\mid f)_{C(\calG^{(0)})} (x)   =\|f\|^{2}_{C_{0}(\calG^{(0)})}, 
\]
and we find that $\sup\|u_{n}\|_{\End^{*}(E_{C_0(\mathcal{G}^{(0)})})}\leq 1$ as claimed.
\end{proof}
\begin{prop} \label{prop:s-cover_frame}
The sequence $u_{n}$ forms an approximate unit for $\mathbb{K}(E_{C_0(\mathcal{G}^{(0)})})$. 
In other words, the ordered set of elements $\chi_{i}\in E_{C_0(\mathcal{G}^{(0)})}$ forms a frame for $E_{C_0(\mathcal{G}^{(0)})}$.
\end{prop}
\begin{proof} The sequence $u_{n}$ is uniformly bounded in operator norm and converges strongly to 1 on a dense subset. 
This implies that it converges strongly to 1 on all of $E_{C_0(\mathcal{G}^{(0)})}$, which is equivalent to being an approximate unit for 
$\mathbb{K}(E_{C_0(\mathcal{G}^{(0)})})$.
\end{proof}

\subsection{Morita equivalence of twisted groupoid $C^{*}$-algebras} \label{sec:twisted_gpoid_equiv}

In this section we work with an arbitrary \'{e}tale groupoid $\calG$ with closed 
subgroupoid $\calH$ that admits a Haar system and a normalised $2$-cocycle $\sigma:\calG^{(2)}\to \T$ or $\{\pm 1\}$. 
The map $\sigma$ restricts to a $2$-cocycle on the subgroupoid $\calH$. 
Denote by $$\rho_\calH: C_c(\calG,\sigma)\to C_c(\calH,\sigma),$$ the restriction map. This map 
is a generalised conditional expectation by~\cite[Proposition 2.9]{Renault80}.
It gives rise to a $C_c(\calH,\sigma)$-valued inner product, where 
$$
(f_1\mid f_2)_{C_c(\calH,\sigma)}(\eta) = \rho_\calH( f_1^{*}\ast f_2)(\eta) 
  = \sum_{\xi \in r_\calG^{-1}(r_\calH(\eta))} f_1^*(\eta^{-1} \xi) f_2(\xi^{-1}) \sigma(\eta^{-1}\xi,\xi^{-1}).
$$ 
This map is compatible with 
the right-action,
$$
  (f \cdot h)(\gamma) = \sum_{\eta \in r_{\calH}^{-1}(s_\calG(\gamma))} 
       f(\gamma\eta) h(\eta^{-1}) \sigma(\gamma\eta, \eta^{-1}), \qquad f\in C_c(\calG,\sigma), \,\, h\in C_c(\calH,\sigma).
$$

We again take the completion of $C_c(\calG,\sigma)$ in the 
$C^*_r(\calH,\sigma)$-valued inner-product to obtain a right $C^*$-module $E_{C^*_r(\calH,\sigma)}$. The left-action of $C_{c}(\calG,\sigma)$ on itself makes
$E_{C^*_r(\calH,\sigma)}$ into a \newline $(C^{*}_{r}(\mathcal{G},\sigma),C^{*}_{r}(\mathcal{H},\sigma))$-bimodule by \cite[Theorem 1.4]{SOU}. 
These bimodules often support a natural operator making them into $KK$-cycles, as we will discuss in Section \ref{sec:bulk_Kasmod}.

At present we wish to describe the compact operators on $E_{C^{*}_{r}(\mathcal{H},\sigma)}$. To this end we first define
$$
   \calG/\calH \cong  
   \big\{ [\xi] \,:\, \xi\in\calG, \, [\gamma]=[\xi] \iff \text{ there exists } \eta\in\calH \text{ with } \gamma\eta = \xi \big\}.
$$
We can define a new groupoid by considering a left-action of $\calG$ on this quotient space. Namely, 
we take
$$
   \calG \ltimes \calG/\calH := \big\{ (\xi,[\gamma]) \in \calG \times \calG/\calH \,:\, s_\calG(\xi) = r_\calG(\gamma) \big\},
$$
where we have $(\calG \ltimes \calG/\calH)^{(0)}= \calG/\calH$ and 
\begin{align*}
   &r(\xi,[\gamma]) = [\xi\gamma],    &&s(\xi,[\gamma]) = [\gamma],  \\
   &(\xi,[\gamma])^{-1} = (\xi^{-1},[\xi\gamma]),  
   &&(\xi,[\gamma])\circ(\eta,[\eta^{-1}\gamma]) = (\xi\eta,[\eta^{-1}\gamma]).
\end{align*}
Furthermore, we can again use the $2$-cocycle $\sigma$ on $\calG$ to define a $2$-cocycle on 
$\calG\ltimes \calG/\calH$,
$$
  \sigma((\xi,[\gamma]), (\eta,[\eta^{-1}\gamma])) = \sigma(\xi,\eta).
$$
The groupoid $\calG$ naturally implements an 
equivalence between $\calG\ltimes \calG/\calH$ and $\calH$ in the sense of \cite{MRW}. Namely $\calG$ is a 
free and proper left $(\calG\ltimes \calG/\calH)$-space and a free and 
proper right $\calH$-space via the groupoid actions,
\begin{align*}
   &(\xi,[\gamma])\cdot \eta = \xi\eta,\quad s(\xi)=r(\gamma)=r(\eta),  
   &&\gamma\cdot \eta = \gamma\eta,\quad s(\gamma)=r(\eta).
\end{align*}

 In particular $C_{c}(\mathcal{G})$ 
can be completed into Morita equivalence bimodules for both the 
full and reduced $C^*$-algebras of $\calH$ and $\calG \ltimes \calG/\calH$~\cite{MRW,SimsWilliams}. 
In case the $2$-cocycles on $\calH$ and $\calG \ltimes \calG/\calH$ are compatible (e.g. if 
both are inherited from a fixed $2$-cocycle on $\calG$), then 
the full twisted groupoid $C^*$-algebras are Morita equivalent by~\cite[Theorem 9.1]{Daenzer09}. 
Morita equivalence 
was extended to the reduced $C^*$-algebras of Fell bundles in~\cite{MT11,  MoutuouThesis, SimsWilliamsFell}, 
which includes twisted reduced groupoid $C^*$-algebras (see~\cite[Proposition 6.2]{MT11}).
We briefly review this construction for the special case in which we are working.

We define a left-action of $C_c(\calG\ltimes \calG/\calH,\sigma)$ on $C_c(\calG,\sigma)$ 
(seen as a right $C_c(\calH,\sigma)$-module) by the formula
$$
  (\pi(g)f)(\gamma) = \sum_{\xi \in r^{-1}(r(\gamma))} g(\xi,[\xi^{-1} \gamma]) f(\xi^{-1}\gamma) 
     \sigma(\xi, \xi^{-1}\gamma), \quad g\in C_c(\calG\ltimes\calG/\calH,\sigma), \,\, f\in C_c(\calG,\sigma).
$$
As we argue below, this action extends to an isomorphism
$$ C^*_r(\calG \ltimes \calG/\calH, \sigma) \xrightarrow{\simeq} \mathbb{K}(E_{C^*_r(\calH,\sigma)}),$$
to obtain the following result.

\begin{prop}[\cite{MT11}, Theorem 5.5, \cite{SimsWilliams}, Theorem 4.1,  \cite{SimsWilliamsFell}, Theorem 14] \label{prop:twisted_morita_reduced}
The $C^*$-module $E_{C^*_r(\calH,\sigma)}$ is a Morita equivalence bimodule 
between the $C^{*}$-algebras $C^*_r(\calH,\sigma)$ and $C^*_r(\calG\ltimes \calG/\calH,\sigma)$.
\end{prop}

This statement is derived from the proof in~\cite{SimsWilliams}  with fairly minor 
alterations. The more general Fell bundle setting requires more machinery, see~\cite{MT11, SimsWilliamsFell}. 
We define the linking groupoid as the topological disjoint union, 
$$
  L = (\calG\ltimes \calG/\calH) \sqcup \calG \sqcup \calG^\mathrm{op} \sqcup \calH
$$
where $\calG^\mathrm{op}$ is the opposite groupoid 
$\calG^\mathrm{op} = \{\ol{\gamma}\,:\, \gamma\in\calG\}$, which we can equip with a 
$2$-cocycle $ \sigma^\mathrm{op}(\ol{\gamma_1}, \ol{\gamma_2}) = \sigma(\gamma_2,\gamma_1)$. 
As the name suggests, $L$ is a groupoid with unit space $\calG/\calH \sqcup \calH$ and 
source and range maps inherited from the groupoid structure on its parts~\cite[Lemma 2.1]{SimsWilliams}. 
We can consider the twisted convolution algebra of $L$, with respect to the cocycle $\hat{\sigma}:L\to \mathbb{T}$ 
which coincides with the given cocycles on each of the components of the disjoint union.

The algebraic machinery used in~\cite{MuhlyWilliams, SimsWilliams} also works in the twisted case (see 
~\cite[Chapter II, Lemma 2.5]{Renault80} or~\cite[Chapter 5]{MoutuouThesis}) and the 
argument in~\cite{SimsWilliams} follows through to obtain the result.

\subsection{Exact cocycles and unbounded $KK$-cycles} \label{sec:bulk_Kasmod}

We now discuss the construction of $KK$-cycles from the data of a continuous 1-cocycle $c:\mathcal{G}\to \mathbb{R}^{n}$ which is \emph{exact}
in the sense of \cite[Definition 3.3]{MeslandGpoid}. We will assume $\mathcal{G}$ is \'{e}tale, so that in this higher dimensional setting exactness entails that $\Ker(c)$ admits a Haar system  and the map
\[r\times c:\mathcal{G}\to\mathcal{G}^{(0)}\times\mathbb{R}^{n},\quad \xi\mapsto (r(\xi),c(\xi)),\]
is a quotient map onto its image.

Given $\calH = \Ker(c)$ a closed subgroupoid of $\calG$, 
we will construct a $KK$-cycle from $c$ supported on the module $E_{C^{*}_{r}(\mathcal{H},\sigma)}$ 
constructed in the previous section. We use the representation of $C_c(\calG,\sigma)$ on 
$E_{C^{*}_{r}(\calH,\sigma)}$ by left-multiplication, 
$\pi(f_1)f_2 = f_1\ast f_2$ for $f_2\in C_c(\calG,\sigma)\subset E_{C^{*}_{r}(\mathcal{H},\sigma)}$. 
Again by \cite[Theorem 1.4]{SOU} this action extends to a representation of $C^*_r(\calG,\sigma)$.

The components of the exact cocycle $c:\calG\to \R^n$ give $n$ real cocycles 
$c_{k}(\xi):=(\pi_{k}\circ c)(\xi)$ by composition with the $k$-th coordinate projection 
$$\pi_{k}:\mathbb{R}^{n}\to \mathbb{R}, \quad x=(x_{1}\cdots , x_{n})\mapsto x_{k}.$$ 
Following~\cite{Renault80},  
 $C^{*}_{r}(\mathcal{G},\sigma)$ has $n$ mutually commuting one-parameter groups of 
automorphisms 
$\{u_t^{(k)}\}_{k=1}^n$, which on $C_c(\calG, \sigma)$ are given by
$$
   (u_t^{(k)} f)(\xi) =  e^{it c_k(\xi)} f(\xi),\qquad t\in\R
$$ 
with $c_k=\pi_{k}\circ c$ as above.
The generators of these automorphisms are derivations
$\{\partial_j\}_{j=1}^n$ on $C_c(\calG,\sigma)$, 
where $(\partial_j f)(\xi) = c_j(\xi) f(\xi)$ (pointwise multiplication). 
We denote by $D_{c_j}$ the extension of these derivations to an unbounded operator 
on $E_{C^*_r(\calH,\sigma)}$.

We use this differential structure to define an unbounded operator that plays the r\^{o}le of 
an elliptic differential operator. 
Our construction mimics the construction of the elements $\alpha$ and $\beta$ 
in~\cite[Section 5]{Kasparov80} and, as such, uses the exterior algebra $\bigwedge^*\R^n$. 
We briefly establish our Clifford algebra notation, where 
$Cl_{r,s}$ is the (real) $\Z_2$-graded $C^*$-algebra generated by the mutually 
anti-commuting 
odd elements $\{\gamma^j\}_{j=1}^r$, $\{\rho^k\}_{k=1}^s$ such that
$$
  (\gamma^j)^2=1, \qquad (\gamma^j)^* = \gamma^j, \qquad (\rho^k)^2 = -1, \qquad (\rho^k)^*=-\rho^k.
$$
The exterior algebra $\bigwedge^*\R^n$ has
representations of $Cl_{0,n}$ and $Cl_{n,0}$
with generators
\begin{align*}
  &\rho^j(\omega) = e_j\wedge \omega - \iota(e_j)\omega, 
  &\gamma^j(\omega) = e_j \wedge \omega + \iota(e_j)\omega,
\end{align*}
where $\{e_j\}_{j=1}^n$ is the standard basis of $\R^n$ 
and $\iota(\nu)\omega$ is the contraction of $\omega$ along $\nu$. 
One readily checks that $\rho^j$ and $\gamma^j$ mutually anti-commute 
and generate representations of $Cl_{0,n}$ and $Cl_{n,0}$ 
respectively. An analogous construction holds in the complex case 
where $\End_{\C}(\bigwedge^*\C^n) \cong \C l_{n}\hat\otimes \C l_{n}$, 
where the two representations graded-commute.

\begin{prop} \label{prop:bulk_kasmod} 
Let $\mathcal{G}$ be an \'{e}tale groupoid and $c:\mathcal{G}\to \mathbb{R}^{n}$ an exact cocycle with kernel $\mathcal{H}$.
The triple 
$$
 {}_{n}\lambda_{\calH}^{c} = \bigg( C_{c}(\calG,\sigma) \hat\otimes Cl_{0,n}, \, E_{C^{*}_{r}(\mathcal{H},\sigma)}\hat\otimes 
    \bigwedge\nolimits^{\! *} \R^n, \, D_{c}= \sum_{j=1}^n D_{c_j} \hat\otimes \gamma^j \bigg)
$$
is an unbounded real Kasparov module for $(C^{*}_{r}(\mathcal{G},\sigma),C^{*}_{r}(\mathcal{H},\sigma))$. 
If we use complex algebras and $\bigwedge^*\C^n$, 
the Kasparov module is complex.
\end{prop}
\begin{proof}
The essential self-adjointness and regularity of $D$ follow since the subset 
$$C_{c}(\mathcal{G},\sigma)\hat\otimes 
    \bigwedge\nolimits^{\! *} \R^n \subset E_{C^{*}_{r}(\mathcal{H},\sigma)}\hat\otimes 
    \bigwedge\nolimits^{\! *} \R^n,$$
    is a core for $D$ and $D^{2}=c^{2}\hat\otimes 1_{\bigwedge^*\R^d}$ with 
$(c^2 f)(\xi) = (c_1(\xi)^2+\cdots +c_n(\xi)^2)f(\xi).$ 
Therefore $1+D^{2}$ has dense range. We note that in particular
    $$(1+D^2)^{-1} = (1+c^2)^{-1}\hat\otimes 1_{\bigwedge^*\R^n}.$$ 
    Using exactness of $c$, the same argument 
as~\cite[Theorem 3.9]{MeslandGpoid} can now be applied to show that $(1+D^2)^{-1}$ is compact in $E_{C^{*}_{r}(\mathcal{H})}$.
 For $f\in C_{c}(\mathcal{G},\sigma)$, a simple computation using the regular representation 
 gives that 
$$
  [D_c, \pi(f)] = \sum_{j=1}^n [D_{c_j},\pi(f)]\hat\otimes \gamma^j = \sum_{j=1}^n \pi(\partial_j f) \hat\otimes \gamma^j
$$
which is adjointable as $C_{c}(\mathcal{G},\sigma)$ is invariant under the derivations $\{\partial_j\}_{j=1}^n$. 
\end{proof}

We remark that there is additional structure on the $KK$-cycles constructed in Proposition 
\ref{prop:bulk_kasmod}. Namely, using the action of $\mathrm{Spin}_{n,0}$ or $\mathrm{Spin}_{0,n}$ 
on $\bigwedge^* \R^n$ defined in~\cite[{\S}2.18]{Kasparov80} and using the 
notation from~\cite[{\S}5]{Kasparov80}, the unbounded $KK$-cycle ${}_{n}\lambda_{\calH}^{c}$ 
determines a class in the equivariant Kasparov group $KKO^{\R^n}_{\mathrm{Spin}_n}(C^*_r(\calG,\sigma),C^{*}_{r}(\mathcal{H},\sigma))$ 
or $KK^{\C^n}_{\mathrm{Spin}_n}(C^*_r(\calG,\sigma),C^{*}_{r}(\calH,\sigma))$. We can then restrict 
the $C^*$-module $E_{C^{*}_{r}(\mathcal{H},\sigma)}\hat\otimes \bigwedge^*\R^n$ to the irreducible 
spinor representation space. By~\cite[{\S}5, Lemma 1]{Kasparov80}, this restriction gives 
 an isomorphism of $KK$-groups.

For concreteness, we write out the unbounded representatives of the spinor Kasparov modules explicitly.
Denote by $S_n^\C$ and $S_n$ the (trivial) complex and real spinor 
bundles of $\R^n$ and let $K_n=\R$, $\C$ or $\mathbb{H}$ be the maximal commuting subalgebra for the 
irreducible (real) representation of $Cl_{n,0}$ on $S_n$ (see~\cite[Chapter I, {\S}5]{Spingeo}).
\begin{prop} \label{prop:spin_KasMod}
Let $\mathcal{G}$ be an \'{e}tale groupoid and $c:\mathcal{G}\to \mathbb{R}^{n}$ an exact cocycle and $\calH:=\Ker(c)$. 
Then the triple 
$$
  {}_{n}\lambda_{\calH}^{S_\C} = \bigg( C_{c}(\mathcal{G},\sigma), \, E_{C^{*}_{r}(\mathcal{H},\sigma)} \hat\otimes S_n^\C, \, 
   \sum_{j=1}^n D_{c_j} \hat\otimes \gamma^j \bigg) 
$$
is a complex Kasparov module of parity $n \,\,\mathrm{mod}\, 2$. 

Let $S_n$ be the real spinor bundle of $\R^n$. If $n\not\equiv 1\,\mathrm{mod}\,4$, then 
$$
  {}_{n}\lambda_{\calH}^{S} = \bigg( C_{c}(\mathcal{G},\sigma), \, \big(E \hat\otimes S_n \big)_{C^{*}_{r}(\mathcal{H},\sigma)\hat\otimes K_n}, \, 
   \sum_{j=1}^n D_{c_j} \hat\otimes \gamma^j \bigg), \qquad 
   K_n = \begin{cases}  \R, & n=0,2 \,\,\mathrm{mod}\, 8, \\ \C, & n= 3\,\,\mathrm{mod}\,4, \\  \mathbb{H}, & n=4,6\,\,\mathrm{mod}\,8 \end{cases}
$$
is a real graded unbounded Kasparov module.  If $n\equiv 1\,\mathrm{mod}\,4$, then
$$
 {}_{n}\lambda_{\calH}^{S} = \bigg( C_{c}(\mathcal{G},\sigma) \hat\otimes Cl_{0,1},\, \begin{pmatrix} E_{C^{*}_r(\calH,\sigma)} \otimes S_{n} \\  E_{C^{*}_r(\calH,\sigma)} \otimes S_{n} 
  \end{pmatrix}_{ K_n}, \, \Big(\sum_{j=1}^n D_{c_j} \otimes \gamma^j \Big) \hat\otimes \sigma_1  \bigg), 
    \,\,\, K_n = \begin{cases} \R, & n=1\,\mathrm{mod}\,8, \\ \mathbb{H}, & n=5\,\mathrm{mod}\, 8 \end{cases}
$$
is an unbounded Kasparov module, where the left-action of $Cl_{0,1}$ is generated by $\begin{pmatrix} 0 & -1 \\ 1 & 0 \end{pmatrix}$.
\end{prop}

The spinor Kasparov modules have the advantage that the left algebra is no longer 
graded, which is useful if we wish to apply the local index formula (for complex semifinite 
spectral triples constructed from ${}_{n}\lambda_{\calH}^{S_\C}$). We will predominantly work 
with the `oriented' Kasparov module ${}_{n}\lambda_{\calH}^{c}$ and 
class $[{}_{n}\lambda_{\calH}^{c}]\in KK^n(C^*_r(\calG,\sigma),C^{*}_{r}(\calH,\sigma))$ (real or complex)  
as the representations are more 
tractable and we can work in the real or complex category interchangeably. Though 
we emphasise that at the level of $K$-groups (and up to a possible normalisation), there is 
no loss of information working with either the spin or oriented $KK$-cycles.

\subsection{The extension class of an $\mathbb{R}$-valued cocycle} \label{subsec:exact_extension}
Consider the Kasparov module from Proposition \ref{prop:bulk_kasmod}. In case $n=1$ we obtain an 
ungraded Kasparov module to which we can associate an extension of $C^{*}$-algebras with 
positive semi-splitting.

In this section we fix $\varepsilon>0$ and a continuous non-decreasing 
function $\chi_{+}:\mathbb{R}\to\mathbb{R}$ satisfying
\[\chi_{+}(x):=\left\{\begin{matrix}  0  &\textnormal{if } x\leq -\varepsilon \\ 1  & \textnormal{if } x\geq 0. \end{matrix}\right.\]
\begin{lemma}  \label{lem:pi_c_almost_proj}
The operator 
\[\Pi_{c}:C_{c}(\mathcal{G},\sigma)\to C_{c}(\mathcal{G},\sigma), \quad \Pi_{c}f(\eta):=\chi_{+}(c(\eta))f(\eta),\] 
extends to a self-adjoint operator $\Pi_{c}\in \End^{*}_{C^{*}_{r}(\mathcal{H},\sigma)}(E)$ with $\|\Pi_{c}\|\leq 1$. 
For all $f\in C_{c}(\mathcal{G},\sigma)$ it holds that $\pi(f)(\Pi_{c}^{2}-\Pi_{c})\in\mathbb{K}(E)$.
\end{lemma}
\begin{proof}
The operator $D_{c}$ is self-adjoint and regular in $E_{C^{*}_{r}(\mathcal{H},\sigma)}$ and $\Pi_{c}:=\chi_{+}(D)$ as defined by the continuous functional calculus. It follows  that $\Pi_{c}$ is a selfadjoint operator on $E_{C^{*}_{r}(\mathcal{H},\sigma)}$. 

The action of $\Pi_{c}^{2}-\Pi_{c}$ is implemented by the function 
$K_{c}\in C_{c}(\mathcal{G}\ltimes \mathcal{G}/\mathcal{H},\sigma)$
\[K_{c}(\xi,[\eta])=(\chi_{c}(c(\eta))^{2}-\chi_{c}(c(\eta)))f(\xi),\]
and thus defines a compact operator.
\end{proof}

The same argument as the proof of Lemma \ref{lem:pi_c_almost_proj} 
shows that the (ungraded) Kasparov module $( C^*_r(\calG,\sigma), E_{C^*_r(\calH,\sigma)}, 2 \Pi_c -1)$ 
represents the same class in $KKO^1(C^*_r(\calG,\sigma),C^*_r(\mathcal{H},\sigma))$ (or complex) as  the bounded 
representative of 
$\big(C_c(\calG, \sigma)\hat\otimes Cl_{0,1}, E_{C^*_r(\mathcal{H},\sigma)}\hat\otimes \bigwedge^*\R, D_{c}\big)$.

In order to construct an extension of $C^*_r(\calG,\sigma)$ by 
 $C^{*}_{r}(\mathcal{G}\ltimes \mathcal{G}/\mathcal{H})$, which is Morita equivalent to $C^*_r(\calH,\sigma)$, 
we need to consider the Busby invariant. To this end, we first note the following 
result on \'{e}tale groupoids.

\begin{prop}[\cite{Renault80}, Chapter II, Proposition 4.2 and \cite{SimsNotes}, Proposition 3.3.3] \label{prop:j_map}
Let $\calG$ be an \'{e}tale groupoid with a fixed $2$-cocycle $\sigma$. 
The identity map $C_{c}(\calG,\sigma)\to C_{c}(\calG,\sigma)$ extends to a 
continous injection $j:C^{*}_{r}(\calG,\sigma)\to C_{0}(\calG)$. For $a\in C^{*}_{r}(\mathcal{G})$ the map $j$ is given by
\begin{equation}
\label{jSims}
j_{\calG}(a)(\eta):=\big( \pi(a)\delta_{s(\eta)} \mid  \delta_{\eta}\big)_{C_{0}(\calG^{(0)})} (r(\eta)),
\end{equation}
where $\delta_{\eta}\in C_{c}(\calG,\sigma)$ is any function for which 
$$r:\textnormal{supp } (\delta_{\eta})\to r(\textnormal{supp }(\delta_{\eta})),\quad s:\textnormal{supp }(\delta_{\eta})\to s(\textnormal{supp }(\delta_{\eta})),$$
are homemorphisms and $\delta_{\eta}(\eta)=1$ on a neighborhood of $\eta$. 
\end{prop}

\begin{defn} Let $c:\mathcal{G}\to \mathbb{R}$ be a continuous cocycle. We say that $c$ is $r$-\emph{unbounded} if for all $x\in\mathcal{G}^{(0)}$ 
and all $M>0$ there exists $\eta\in r^{-1}(x)$ for which $c(\eta)>M$.
\end{defn}
Recall that the \emph{Calkin algebra} of a $C^{*}$-module $E$ over a $C^{*}$-algebra $B$ is the quotient 
$\mathcal{Q}(E_B):=\End^{*}_{B}(E)/\mathbb{K}(E_B)$. We denote by $q:\End^{*}_{B}(E)\to \mathcal{Q}(E_B)$ the quotient map. Lastly, we denote by $\mathcal{M}(B)$ the multiplier algebra of $B$.
\begin{prop} \label{prop:injective_busby}
Let $c:\mathcal{G}\to \mathbb{R}$ be an exact cocycle on a Hausdorff \'{e}tale groupoid 
$\mathcal{G}$, $\mathcal{H}=\Ker(c)$ and $\sigma$ a $2$-cocycle on $\mathcal{G}$. If $c$ is $r$-unbounded, then the $*$-homomorphism
$$\varphi: C^*_r(\calG,\sigma) \to \calQ(E_{C^*_r(\mathcal{H},\sigma)}),\quad \varphi(a) = q(\Pi_c a \Pi_c) $$ 
is injective.
\end{prop}

\begin{proof}
Using Proposition \ref{prop:j_map}, we can view elements of 
$C^{*}_{r}(\calG\ltimes \calG/\calH,\sigma)$ as $C_0$-functions on $\calG\ltimes \calG/\calH$. 
Let $E=E_{C^{*}_{r}(\mathcal{H},\sigma)}$ and $F=F_{\mathcal{G}/\mathcal{H}}$ the 
$C^*$-module over the unit space of $\calG\ltimes \mathcal{G}/\mathcal{H}$. Then
$$
\End^{*}_{C^{*}_{r}(\mathcal{H},\sigma)}(E)= \calM \big(\mathbb{K}(E_{C^*_r(\calH,\sigma)}) \big)
= \calM(C^{*}_{r}(\mathcal{G}\ltimes \mathcal{G}/\mathcal{H}, \sigma)).
$$ 
Since the representation of $C^{*}_{r}(\mathcal{G}\ltimes \mathcal{G}/\mathcal{H},\sigma)$ on $F$ is essential, we see that 
$\End^{*}_{C^{*}_{r}(\mathcal{H},\sigma)}(E)$ acts on $F$. 
Thus if $j=j_{\mathcal{G}\ltimes\mathcal{G}/\mathcal{H}}$ and 
$T\in\End^{*}_{C^{*}_{r}(\mathcal{H},\sigma)}(E)$ then the formula \eqref{jSims} 
defines a continuous function $j(T):\mathcal{G}\ltimes\mathcal{G}/\mathcal{H}\to \mathbb{C}$. 
The functions $\delta_{(\xi,[\eta])}$ can be chosen so that $\|\delta_{(\xi,[\eta])}\|_{F}\leq 1$, 
so we obtain the pointwise estimate
\begin{align*}
|j(T)(\xi,[\eta])| &=\big| ( \pi(T)\delta_{[\eta]} \mid \delta_{(\xi,[\eta])} )_{C_0(\calG/\calH)} ([\xi\eta]) \big|
\leq \big\| ( \pi(T)\delta_{[\eta]} \mid \delta_{(\xi,[\eta])} ) \big\|_{C_{0}(\mathcal{G}/\mathcal{H})}\\
&\leq \|T\|_{\End^{*}_{C^{*}_{r}(\mathcal{H},\sigma)}(E)}  \, \|\delta_{[\eta]}\|_{F}\, \|\delta_{(\xi,[\eta])}\|_{F}\, \leq \|T\|.
\end{align*}
In particular, if $T_{n}\to T$ in norm in $\End^{*}_{C^{*}_{r}(\mathcal{H},\sigma)}(E)$ 
then $j(T_{n})\to j(T)$ pointwise on $\mathcal{G}\ltimes\mathcal{G}/\mathcal{H}$.

Suppose that $a\neq 0\in C^{*}_{r}(\mathcal{G},\sigma)$ and choose $\xi\in \calG$ with $|j(a)(\xi)|\geq 3\delta> 0$. 
Choose $f\in C_{c}(\mathcal{G},\sigma)$ 
with $\|f-a\|_{C^{*}_{r}(\mathcal{G},\sigma)}<\delta$ and $|f(\xi)|\geq 2\delta $. Then for every $(\xi,[\eta])\in \mathcal{G}\ltimes \mathcal{G}/\mathcal{H}$ it holds that
\begin{align*}
|j(\Pi_{c}(f-a)\Pi_{c})(\xi,[\eta])|\leq \|\Pi_{c}(f-a)\Pi_{c}\|\leq \|f-a\|<\delta.
\end{align*}
For $f\in C_{c}(\mathcal{G})$  it holds that 
$$j(\Pi_{c}f\Pi_{c})(\xi,[\eta])=\chi_{+}(c(\xi\eta))\chi_{+}(c(\eta))f(\xi).$$ 
Thus for all $[\eta]=(r(\eta),c(\eta))$ satisfying $c(\eta)\geq\max\{0, -c(\xi)\}$ we estimate
\begin{align*}
\big| j(\Pi_{c}a\Pi_{c})(\xi,[\eta]) \big|&\geq \big| j(\Pi_{c}f\Pi_{c})(\xi,\eta)\big |- \big|j(\Pi_{c}(f-a)\Pi_{c})(\xi,[\eta]) \big|\\
&=|f(\xi)|-|(j(\Pi_{c}(f-a)\Pi_{c})(\xi,[\eta])|\\
&\geq |f(\xi)|-\|\Pi_{c}(f-a)\Pi_{c}\|\\
&\geq |f(\xi)|-\|f-a\|>\delta.
\end{align*}
Since $c$ is exact there is a homeomorphism
\[\mathcal{G}/\mathcal{H}\to \{(r(\xi), c(\xi)): \xi\in \mathcal{G}\}\subset \mathcal{G}^{(0)}\times \mathbb{R},\]
where the latter set carries the relative topology. Since $c$ is $r$-unbounded, for fixed $\xi$ there is a noncompact set of pairs $(\xi,[\eta])\in\mathcal{G}\ltimes\mathcal{G}/\mathcal{H}$ with 
$|j(\Pi_{c}a\Pi_{c})(\xi,[\eta])|>\delta$.
Therefore $j(\Pi_c a\Pi_c)\notin C_{0}(\mathcal{G}\ltimes\mathcal{G}/\mathcal{H})$ and
$$\Pi_{c}a\Pi_c\notin C^{*}_{r}(\mathcal{G}\ltimes\mathcal{G}/\mathcal{H},\sigma)=\mathbb{K}_{C^*_r(\calH,\sigma)}(E).$$
This is equivalent to the statement that the map 
$$
\varphi: C^*_r(\calG,\sigma) \to \calQ(E_{C^*_r(\mathcal{H},\sigma)}),\quad \varphi(a) = q(\Pi_c a \Pi_c),
$$ 
is injective.
\end{proof}

Using the
isomorphism $\mathbb{K}(E_{C^*_r(\calH,\sigma)}) \cong C^*_r(\calG\ltimes \calG/\calH, \sigma)$ 
and the injectivity of $\varphi$, 
we construct 
the generalised Toeplitz extension 
\begin{equation*} 
   0 \to C^*_r(\calG \ltimes \calG/ \calH, \sigma) \to 
     C^*( \Pi_c  C^*_r(\calG,\sigma) \Pi_c, C^*_r(\calG \ltimes \calG/ \calH, \sigma) ) \to C^*_r(\calG,\sigma) \to 0
\end{equation*}
with completely positive semi-splitting $a\mapsto \Pi_c a \Pi_c$ and Busby invariant $\varphi$.
The algebra $$\calT = C^*( \Pi_c  C^*_r(\calG,\sigma) \Pi_c, C^*_r(\calG \ltimes \calG/ \calH, \sigma) )$$ is represented on 
$\Pi_c E_{C^{*}(\calH,\sigma)}$.
\section{Delone sets and the transversal groupoid} \label{sec:Transversal}

We briefly summarise the construction of a groupoid of an aperiodic hull. Results 
and further details 
can be found in~\cite{AndersonPutnam, BHZ00, Kellendonk95, Kellendonk97, KellendonkPutnam, BBG06}. 
We most closely follow the perspective of~\cite{BHZ00, BBG06} and construct a dynamical 
system and transversal groupoid from the topology of point measures in $\R^d$.

\begin{defn}
Let $\calL \subset\R^d$ be discrete and infinite and fix $0<r<R$.
\begin{enumerate}
  \item $\calL$ is $r$-uniformly discrete if 
    $|B(x;r)\cap \calL| \leq 1$ for all $x\in\R^d$. 
  \item $\calL$ is $R$-relatively dense if 
   $|B(x;R)\cap \calL| \geq 1$ for all $x\in\R^d$.
\end{enumerate}
An $r$-uniformly discrete and $R$-relatively dense set $\calL$ is called an 
$(r,R)$-Delone set.
\end{defn}

We will occasionally want extra structure on our Delone set.
\begin{defn} \label{def:FLC_et_al}
Let $\calL\subset\R^d$ be discrete and infinite.
\begin{enumerate}
  \item A patch of radius $R > 0$ of $\calL$ is a subset of $\R^d$ 
  of the form $(\calL - x) \cap B(0; R)$, for some $x \in\calL$. 
  If for all $R > 0$ the set of its patches of radius $R$ is finite, then
   $\calL$ has finite local complexity. 
  \item We call $\calL$ repetitive if given any finite subset $p\subset \calL$ and $\varepsilon >0$, 
  there is an $R>0$ such that in any ball $B(x;R)$ there is a subset $p' \subset \calL \cap B(x;R)$ 
  that is a translation of $p$ within the distance $\varepsilon$; that is, there is an 
  $a\in \R^d$ such that the Hausdorff distance between $p'$ and $p+a$ is less that $\varepsilon$. 
  \item We call $\calL$ aperiodic if there is no $x \neq 0 \in \R^d$ such that 
  $\calL - x = \calL$.
\end{enumerate}
\end{defn}

There is an equivalence between discrete sets and point measures in $\R^d$. 
Let $\calM(\R^d)$ denote the space of measures on $\R^d$ and consider
\begin{align*}
  QD(\R^d) &= \{ \nu \in\calM(\R^d)\,:\, \forall x\in \R^d, \,\nu 
    \text{ is pure point and } \nu(\{x\}) \in \N \}, \\
  UD_r(\R^d) &= \{\nu\in QD(\R^d)\,:\, \forall x\in \R^d, \,\nu(B(x;r))\leq 1\}.
\end{align*}
For $\nu\in QD(\R^d)$, $\calL^{(\nu)}=\mathrm{supp}(\nu)$ is discrete. Similarly for 
a discrete set $\calL$ we can define a measure 
$\delta_\calL = \sum_{x\in\calL} \delta_x \in QD(\R^d)$, 
where $\delta_x$ is the point measure. We can also relate measures and 
Delone sets.

\begin{prop}
Let $\nu\in UD_r(\R^d)$ be a measure such that for all $x\in\R^d$,
$ \nu(\ol{B(x;R)}) \geq 1$.
Then $\calL^{(\nu)}$ is an $(r,R)$-Delone set.
\end{prop}
As $\calM(\R^d)$ is a subspace of $C_c(\R^d)^{*},$ it can be given the weak $\ast$-topology.

\begin{prop}[\cite{BHZ00}, Theorem 1.5] \label{prop:disc_meas_compact}
The set $UD_r(\R^d)$ is a compact subspace of $\calM(\R^d)$.
\end{prop}

\begin{prop}[cf.~\cite{Lagragias03}, Section 3, \cite{FHK}, Chapter 1]
The set of $(r,R)$-Delone sets is a compact and metrizable space. Let $d_H$ denote the Hausdorff distance between sets. A
neighborhood base at $\omega\in\Omega_{\mathcal{L}}$ is given by the sets
$$
   U_{\epsilon,M}(\omega) = \big\{ \eta \in \mathrm{Del}_{(r,R)}\,:\, 
      d_H\big( \calL^{(\omega)} \cap B(0;M), \, \calL^{(\eta)} \cap B(0;M) \big) < \epsilon \big\}
$$ 
with $M,\varepsilon>0$.
\end{prop}

The translation action on $\R^d$ gives an action on $C_c(\R^d)$ and thus 
an action on $UD_r(\R^d)$, where
$$
  (T_a \nu)(f) = \nu( T_{-a}f),  \qquad (T_{-a}f)(x) = f(x-a), \,\, f\in C_c(\R^d).
$$
As expected, the $\R^d$-action on $UD_r(\R^d)$ induces an $\R^d$-action on the discrete lattices $\mathcal{L}^{(\nu)}$ by translation,
$T_a(\calL^{(\nu)}) = \calL^{(\nu)} + a$.

\begin{defn}[cf.~\cite{Bel86}, Section 2, \cite{BHZ00}, Definition 1.7]
Let $\calL$ be a uniformly discrete subset of $\R^d$. The \emph{continuous hull of} $\calL$ is the 
dynamical system $(\Omega_{\calL}, \R^d, T)$, where $\Omega_\calL$ is the closure 
of the orbit of $\nu \in UD_r(\R^d)$ such that $\mathrm{supp}(\nu) = \calL$.
\end{defn}

We note that $\Omega_\calL$ is compact by Proposition \ref{prop:disc_meas_compact}. 
The translation action on $UD_r(\R^d)$ gives the family of homeomorphisms 
$\{T_a\}_{a\in\R^d}$ on $\Omega_\calL$. Thus, starting from a Delone set $\calL$, we 
may associate to it a continuous topological dynamical system $(\Omega_\calL,T,\R^d)$. 
This dynamical system is minimal if and only if the lattice $\calL$ is repetitive~\cite[Theorem 2.13]{BHZ00}.

\begin{example}
Let $\calL$ be a periodic and cocompact group $G$, then it is 
immediate that $\Omega_\calL \cong \R^d/G$. 
This is the classical picture with no aperiodicity or disorder on our lattice. 
We can use Rieffel induction on the $C^*$-dynamical system to simplify the 
crossed product algebra
$$
 C(\Omega_\calL) \rtimes \R^d \cong C(\R^d/G) \rtimes \R^d \cong C^*(G) \otimes \mathbb{K},
$$
which then implies that, for $\calL=\Z^d$, $K_\ast(C(\Omega_\calL) \rtimes \R^d) \cong K^{-\ast}(\T^d)$. 
Considering applications to topological phases, we see that for periodic lattices 
the dynamics of the hull reproduces the $K$-theoretic phases of the Bloch bundle 
over the Brillouin torus. 
\end{example}

There is a loose equivalence between Delone sets and tilings of $\R^d$, where much of 
the terminology we use was originally 
formulated~\cite{Kellendonk95, FHK, AndersonPutnam, KellendonkPutnam}.

\begin{defn}
A tile of $\R^d$ is a compact subset of $\R^d$ that is homeomorphic to the closed unit ball. 
A tiling of $\R^d$ is a covering of $\R^d$ by a family of tiles whose interiors are 
pairwise disjoint.
\end{defn}

Given a tiling $\calT$ with a uniform minimum and maximum bound on the radius 
of each tile, we can choose a point from the interior of every tile to obtain a Delone set 
$\calL_\calT$. There is also an explicit passage from Delone sets to tilings via the Voronoi tiling.

\begin{defn}
Let $\calL$ be an  $(r,R)$-Delone set in $\R^d$. The Voronoi tile around a point $x\in\calL$ 
is the set
$$
  V_x = \big\{ y\in \R^d\,:\,  \|y-x\| \leq \|y-x'\| \text{ for all } x'\in\calL \big\}.
$$
The Voronoi tiling $\calV$ associated to $\calL$ is the family $\{V_x\}_{x\in\calL}$.
\end{defn}

\begin{remark}[A note on topologies]
Given a Delone set, one may instead consider the corresponding Voronoi tiling. If each 
tile in the Voronoi tiling comes from a finite collection of prototiles, there is a 
canonical tiling space with tiling metric (cf.~\cite[Chapter 1]{SadunBook}).  
The topology of the tiling space is strictly finer than 
the topology coming from the weak-$\ast$ topology on the space of Delone sets. However, if 
the Delone set is repetitive and has finite local complexity, then the topologies are 
equivalent, see~\cite{KellendonkPutnam} and~\cite[Section 2]{BBG06}. 

We will mostly work under the assumption that $\calL$ is $(r,R)$-Delone only. 
Therefore if one wishes to apply our work to tilings, one should also assume 
that $\calL$ is repetitive and has finite local complexity.
\end{remark}

\subsection{The transversal groupoid}
The notion of an abstract transversal in a groupoid allows one to replace a topological groupoid by 
a smaller subgroupoid, up to Morita equivalence. 
\begin{defn}
\label{transversal}
A topological groupoid $\mathcal{F}$ admits 
an \emph{abstract transversal} if there is a closed subset $X\subset\mathcal{F}^{(0)}$ such that 
\begin{enumerate}
\item
$X$ meets 
every orbit of the $\mathcal{F}$-action on $\mathcal{F}^{(0)}$;
\item for the relative topologies on $X$ and 
\[\mathcal{F}_{X}:=\{\xi\in\mathcal{F}: r(\xi)\in X\}\subset \calF,\]
the restrictions $r:\mathcal{F}_{X}\to X$ and $s:\mathcal{F}_{X}\to \mathcal{F}^{(0)}$ are open maps for the relative topologies on $\calF_{X}$ and $X$.
\end{enumerate}
\end{defn}
The set $\mathcal{G}:=\mathcal{F}_{X}\cap\mathcal{F}_{X}^{-1}$ is a closed subgroupoid and 
$\mathcal{F}_{X}$ is a groupoid equivalence between $\mathcal{F}$ and $\mathcal{G}$ 
(with its relative topology), see \cite[Example 2.7]{MRW}. Abstract transversals were 
studied more generally in~\cite[Section 3]{PutnamSpielberg}. We will describe an abstract transversal 
$\mathcal{G}\subset\Omega_{\mathcal{L}}\rtimes\mathbb{R}^{d}$ which is Hausdorff and \'{e}tale in the relative topology.

\begin{defn}
The transversal of a lattice $\calL$ is given by the set
$$
   \Omega_0 = \{ \omega \in \Omega_\calL\,:\, 0 \in \calL^{(\omega)} \},
$$
\end{defn}
We see that $\Omega_0$ is a closed subset of $\Omega_\calL$ and so is 
compact by Proposition \ref{prop:disc_meas_compact}.

\begin{prop}[\cite{BHZ00}, Proposition 2.3, \cite{BBG06}, Proposition 2.24] \label{prop:tranversal_properties}
Let $\mathcal{L}$ be a Delone set. 
\begin{enumerate}
  \item If $\calL$ has finite local complexity, then $\Omega_0$ is 
  totally disconnected.
  \item If $\calL$ is repetitive, aperiodic and of finite 
  local complexity, then $\Omega_0$ is a Cantor set (totally disconnected 
  with no isolated points).
\end{enumerate}
\end{prop}

The passage from the continuous hull $\Omega_\calL$ to the transversal $\Omega_0$ discretises 
the $\R^d$-action at the cost that we no longer have a group action, but only a groupoid structure.

\begin{prop}[\cite{Bel86}, Section 3, \cite{Kellendonk97}, Lemma 2] \label{prop:etale}
Given a Delone set $\calL$ with transversal $\Omega_0$, define the set 
$$
  \calG := \big\{ (\omega, x)\in \Omega_0 \times\R^d \, :\, T_{-x}\omega \in \Omega_0 \big\}=\big\{ (\omega, x)\in \Omega_0 \times\R^d \, :\, x\in \mathcal{L}^{(\omega)}  \big\}.
$$
Then $\calG$ is a Hausdorff \'{e}tale groupoid with maps 
\begin{align} \label{eq:derivation_properties}
   &(\omega,x)^{-1}= (T_{-x}\omega,-x),  &&(\omega,x)\cdot (T_{-x}\omega,y) = (\omega, x+y), 
     &s(\omega,x) = T_{-x}\omega,   &&r(\omega,x) = \omega
\end{align}
and unit space $\calG^{(0)}=\Omega_0$.
\end{prop}

The transversal groupoid $\calG$ and its corresponding (twisted) $C^*$-algebra 
will be our central object of study. The space 
$\Omega_{0}$ is an abstract transversal in the sense of Definition \ref{transversal}, so that 
$\calG \subset \Omega_0\times \mathbb{R}^{d}$ with its subspace topology is 
Morita equivalent to $\Omega_{\mathcal{L}}\rtimes\R^{d}$. This result is well-known to experts, 
see~\cite[Chapter 2, Section 2]{FHK} for the case of tilings. We find it worthwhile 
to give a detailed proof in the Delone lattice setting. 
To this end we first make the following observation.
\begin{lemma} 
Let  $0<\varepsilon<r/2$. For any $\omega\in\Omega_0$, the intersection 
$\mathcal{L}^{(\omega)}\cap B(y;\varepsilon)$ contains at most one point.
\end{lemma}
\begin{proof} 
Suppose that the intersection is nonempty and $x_{1},x_{2}\in \mathcal{L}^{(\omega)}\cap B(y;\varepsilon)$. 
Then $d(x_{1},x_{2})<2\varepsilon<r$ so it must hold that $x_1=x_2$.
\end{proof}

For $\mu\in \mathbb{R}_{>0}$ we denote by
\[P_{\mu}:=\{\mathcal{L}^{(\omega)}\cap B(0;\mu):\omega\in\Omega_0\},\]
the \emph{set of patterns of radius} $\mu$. 
The sets
$$ 
U_{p,\mu}:=\{\omega\in \Omega_\calL \,:\, 0\in\mathcal{L}^{(\omega)}, \mathcal{L}^{(\omega)}\cap B(0;\mu)=p\}\subset\Omega_0,\quad \mu\in\mathbb{R}_{>0},\quad p\in P_{\mu},
$$ 
define the relative topology on the closed subset $\Omega_{0}:=\{\omega \in\Omega_\calL : 0\in\mathcal{L}^{(\omega)}\}$. In case $\calL$ has finite local complexity each set $P_{\mu}$ is finite and the clopen sets $U_{p,\mu}$ determine the 
totally disconnected topology on $\Omega_0$. We now provide the proof that $\Omega_0$ is indeed an abstract transversal.

\begin{prop}\label{dtop} 
Let $\mathcal{L}\subset\mathbb{R}^{d}$ be a uniformly $r$-discrete subset with transversal $\Omega_0$ and associated groupoid $\mathcal{G}$. For $U\subset\Omega_0$ an open set,  the sets
\begin{align*}
V_{(U,y,\varepsilon)}&:=\left(U\times B(y;\varepsilon)\right){}\cap {}\mathcal{G}\\ 
     &=\{(\omega, x)\in \Omega_{0}\times \mathbb{R}^{d}: \omega\in U,
      \quad x \in \mathcal{L}^{(\omega)}\cap B(y;\varepsilon)\},
\end{align*}
form a base for the topology on $\mathcal{G}$. For $0<\varepsilon<r/2$, the restriction
$s:V_{(U,y,\varepsilon)}\to \Omega_0$ is a homeomorphism onto its image. Moreover
 the restrictions 
\[s:\Omega_{\mathcal{L}}\rtimes\mathbb{R}^{d}\cap r^{-1}(\Omega_{0})\to \Omega_{\mathcal{L}},\quad r:\Omega_{\mathcal{L}}\rtimes\mathbb{R}^{d}\cap r^{-1}(\Omega_{0})\to \Omega_0,\]
are open maps. Therefore the set $\Omega_0$ is an abstract transversal and 
the groupoid $\calG\subset \Omega_\mathcal{L}\rtimes \R^d$, with the subspace topology, 
is Morita equivalent to $\Omega_{\mathcal{L}}\rtimes\mathbb{R}^{d}$. 
\end{prop}
\begin{proof} The sets $V_{(U,y,\varepsilon)}$ generate the relative topology on $\mathcal{G}$ 
as a subset of the crossed product groupoid
$\Omega_\calL \rtimes \mathbb{R}^d$ of the hull of $\mathcal{L}$. 
To see that each of the basic sets is an $s$-set, 
we adapt the proof of \cite[Lemma 2.10]{BBG06}. 
The map $s$ is injective on $V_{(U,y,\varepsilon)}$, for $(\omega ,x), (\eta,z)\in V_{(p,\mu,y,\varepsilon)}$ 
the equality $T_{-x}\omega=T_{-z}\eta$ implies that $\omega=T_{x-z}\eta$,
 and $x-z\in\mathcal{L}^{(\omega)}$. Now $x,z\in B(y;\varepsilon)$ so $d(x,z)<2\varepsilon<r$, 
 and thus $x=z$ because $\mathcal{L}^{(\omega)}$ is $r$-discrete. 
 It then follows that $\omega=\eta$ as well. Now consider $s:(\omega, x)\mapsto T_{-x}\omega$ and 
 the image
\begin{align*} 
s\left(V_{(U,y,\varepsilon)}\right)&=\{\omega\in\Omega_{0}: \exists x\in B(y;\varepsilon), \quad T_{x}\omega\in U \}\\
&=\{\omega\in\Omega_{0}: \exists x\in B(0;\varepsilon), \quad T_{x+y}\omega\in U \}\\
&=\Omega_{0}\cap T_{-y}\left(\{\omega\in\Omega_0  : \exists x\in B(0;\varepsilon), \quad T_{x}\omega\in U\}\right)\\
&=\Omega_{0}\cap T_{-y}\left(s(U\times B(0;\varepsilon))\right),
\end{align*}
with $s(\omega ,x)=\phi(\omega,-x)$ and $\phi$ as in~\cite[Lemma 2.10]{BBG06}, 
and by that result the map $s$ is a homeomorphism onto its image. 
Thus, since $y$ is fixed, the set $s\left(V_{(U,y,\varepsilon)}\right)$ is open in $\Omega_{0}$. 
Now $\omega\in s\left(V_{(U,y,\varepsilon)}\right)$ implies that 
$B(-y;\varepsilon)\cap\mathcal{L}^{(\omega)}\neq \emptyset$ and thus contains a unique point $x_{\omega}^{-y}$. 
The map
\[
t_{y}:s\left(V_{(U,y,\varepsilon)}\right)\to V_{(U,y,\varepsilon)},\quad 
     \omega \mapsto (T_{-x_{\omega}^{-y}}\omega, -x_{\omega}^{-y}),
\]
is an inverse for $s$: If $\omega=T_{-x}\eta$ with $\{x\}=B(y;\varepsilon)\cap\mathcal{L}^{(\eta)}$ then
\[
x^{-y}_{\omega}=x^{-y}_{T_{-x}\eta}=B(-y;\varepsilon)\cap \mathcal{L}^{(T_{-x}\eta)}=-x,
\]
and so indeed
\[
t_{y}\circ s(\eta,x)=s_{y}(\omega)=(T_{x}\omega,x)=(\eta,x).
\]

The points $x^{y}_{\omega}$ satisfy the equality $x^{y}_{\omega}=y + x^{0}_{T_{-y}\omega}$ 
and thus the  map $t_{y}$ can be written
$$
t_{y}(\omega)=(T_{-x_{\omega}^{-y}}\omega, -x_{\omega}^{-y})=
   (T_{y}T_{-x^{0}_{T_{y}\omega}}\omega, y-x^{0}_{T_{y}\omega})=
     (T_{y}\times T_{y})\circ t_{0}\circ T_{y}(\omega). 
$$
The map $t_{0}$ is continuous by \cite[Lemma 2.10]{BBG06} and $y$ is fixed, proving continuity of $t_y$.\newline
We now proceed to show the maps $s,r$ are open when restricted to $r^{-1}(\Omega_0)$. As above we have
\[s(U\times B(y;\varepsilon) \cap r^{-1}(\Omega_0))=\{T_{-x}(\omega):\omega\in U\cap\Omega_0, x\in B(y;\varepsilon)\cap\mathcal{L}^{(\eta)}\},\]
and to prove that the map $s|_{r^{-1}(\Omega_0)}$ is open we may restrict ourselves to 
sets $U=U_{\delta,M}(\omega)\cap\Omega_0$ and $M$ sufficiently large, $\delta$ sufficiently small. It then suffices to show that the set
$s(U\times B(y;\varepsilon)\cap r^{-1}(\Omega_0))$ contains a basic open neighborhood of any of its elements $T_{-x}(\omega)$. Let $\delta<\varepsilon<r/2$ and $M>\delta$. Then if $\eta\in\Omega_{\mathcal{L}}$ is such that 
\[d_{H}\left(B(0;M+\|y\|+r)\cap\mathcal{L}^{(T_{-x}\omega)},B(0;M+\|y\|+r)\cap\mathcal{L}^{(\eta)}\right)<\delta/2,\]
we have that $-x\in \mathcal{L}^{(T_{-x}\omega)}\cap B(0;M+\|y\|+r)$. By definition of the Hausdorff distance, we have
\[\inf_{w\in B(0;M+\|y\|)\cap\mathcal{L}^{(\eta)}}\|w+x\|<\delta/2,\]
and since the sets involved are discrete, there exists a point $w\in\mathcal{L}^{(\eta)}$ with $\|w+x\|\leq\delta/2$. Moreover, if $\delta<r$ then this point $w$ is unique because $\mathcal{L}^{(\eta)}$ is $r$-discrete. Then for $z\in B(0;M)\cap\mathcal{L}^{(\omega)}$ and $v\in B(0;M)\cap \mathcal{L}^{(T_{-x}\eta)}$ we have
\[ (z-w)\in B(0;M+\|y\|+r)\cap\mathcal{L}^{(T_{-x}\omega)},\quad (v+x)\in B(0;M+\|y\|+r)\cap\mathcal{L}^{(\eta)}, \]
from which we deduce
\[\|z-v\|\leq \|(v+x)-(z-w)\|+\|x+w\|<\delta,\]
and therefore it follows that
\[d_{H}(B(0;M)\cap\mathcal{L}^{(\omega)},B(0;M)\cap \mathcal{L}^{(T_{-w}\eta)})<\delta.\]
Since $\|w+y\|\leq \|w+x\|+\|x-y\|<\delta<\varepsilon$  it holds that  $(T_{-w}\eta,-w)\in U\times B(y,\varepsilon)$ and $0\in T_{-w}\eta$. Therefore $\eta\in s((U\times B(y,\varepsilon))\cap r^{-1}(\Omega_0))$ and $s:r^{-1}(\Omega_0)\to \Omega_{\mathcal{L}}$ is an open map. The statement that $r$ is an open map is immediate because $\Omega_{0}$ carries the relative topology inherited from $\Omega_{\mathcal{L}}$. This completes the proof.
\end{proof}

From this we derive several structure statements for the groupoid $\mathcal{G}$. 

\begin{prop} \label{prop:exact_cocycles}
For any $1\leq k \leq d$ the groupoid cocycles 
$$\hat{c}_{k}:=(c_{1},\cdots , c_{k}) :(\omega,x)\mapsto (x_{1},\cdots,x_{k}),$$ are exact in the 
sense of \cite[Definition 3.3]{MeslandGpoid}.
\end{prop}
\begin{proof} The subspace topology on $\mathcal{G}$ has a base consisting of the sets 
$$
\left(U_{(p,\mu)}\times B(y;\varepsilon)\right)\cap\mathcal{G}=\{(\omega,x)\in\Omega_0 \times\mathbb{R}^{d}: \mathcal{L}^{(\omega)}\cap B(0;\mu)=p,\quad x\in\mathcal{L}^{(\omega)}\cap B(y;\varepsilon)\},
$$ 
with $\mu\in [0,\infty)$, $p\in P_{\mu}$, $y\in\mathbb{R}^{d}$ and $0<\varepsilon<r/2$. For $(\omega,x)\in\calG$, choose $\mu>\|x\|+r/2$ and let $p:=\mathcal{L}^{(\omega)}\cap B(0;\mu)$. Consider
$(\eta,z)\in (U_{p,\mu}\times B(x;\varepsilon))\cap \mathcal{G}$. Then it holds that $\|z-x\|<\varepsilon<r/2$ and 
\[z,x\in \mathcal{L}^{(\eta)}\cap B(0,\mu)=\mathcal{L}^{(\omega)}\cap B(0,\mu),\]
from which we conclude that $z=x$.
 In particular each $\hat{c}_{k}$ is locally constant and $\hat{c}^{-1}_{k}(0)$ 
is a clopen subgroupoid. Since $\mathcal{G}$ is \'{e}tale, counting measures define a Haar system 
on $\hat{c}_{k}^{-1}(0)$.  
Exactness of the cocycles $\hat{c}_k$ entails that the map 
$(\omega,x)\mapsto (\omega,\hat{c}_{k}(x))=(\omega,x_{1},\cdots,x_{k})$ is a complete quotient map onto its image. 
This map is equal to the restriction of the map $\textnormal{id}\times \pi_{k}$ to $\mathcal{G}$, 
with $\pi_{k}:\mathbb{R}^{d}\to \mathbb{R}^{k}$ the projection onto the first $k$ coordinates, 
which is a complete quotient map.
\end{proof}

Note that the above proof applies to any cocycle $c:\mathcal{G}\to \mathbb{R}^{k}$ that factors through the cocycle $\hat{c}_{d}:\mathcal{G}\to\mathbb{R}^{d}$. Now that we have characterised the \'{e}tale topology on $\calG$, we recall the constructions 
in Section \ref{sec:KK_Gpoid_prelim} and consider an $s$-cover for $\mathcal{G}$ (Definition \ref{def:s_cover}), 
which will then give a frame for the $C^*$-module over the unit space, which we denote 
$E_{C(\Omega_0)}$.
We fix a choice of $0<\varepsilon<r/2$ and a countable set of points $Y\subset\mathbb{R}^{d}$ 
for which $B(y;\varepsilon)$ form an open cover of $\mathbb{R}^{d}$. Note that we can choose 
the set $Y=\lambda\mathbb{Z}^{d}$ with $\lambda>0$ sufficiently small, which is convenient but not necessary. 

\begin{prop}\label{Deloneframe} 
Let $\mathcal{L}\subset\mathbb{R}^{d}$ be a uniformly discrete subset, 
$\mathcal{G}$ the associated groupoid and $E_{C(\Omega_0)}$ the $C^*$-module 
over the unit space. 
For any $0<\varepsilon<r/2$ and any countable cover $\{B(y;\varepsilon)\}_{y\in Y}$ the open sets
\[
V_{y}:=V_{(0, 0, y, \varepsilon)}=\{(\omega,x)\in \Omega_{0}\times \mathbb{R}^{d}:x\in \mathcal{L}^{(\omega)}\cap B(y;\varepsilon)\},
\]
form an $s$-cover for $\mathcal{G}$. Any partition of unity $\chi_{y}$ subordinate to the cover $\{B(y;\varepsilon)\}_{y\in Y}$
of $\mathbb{R}^{d}$ can be lifted to a partition of unity subordinate to the cover $V_{y}$ of $\mathcal{G}$ via $\chi_{y}(\omega,x)=\chi_{y}(x)$.
Consequently the functions $\chi_{y}:\mathcal{G}\to\mathbb{R}$ define a frame for $E_{C(\Omega_0)}$.
\end{prop}
\begin{proof} 
The sets $V_{y}$ form an open cover of $\mathcal{G}$ because each $(\omega,x)\in\mathcal{G}$ is an 
element of $V_{y}$ whenever $x\in B(y;\varepsilon)$ and such $y$ exists because $B(y;\varepsilon)$ 
form an open cover. Moreover, each of the $V_{y}$ is an $s$-set by Lemma \ref{dtop}. 
The functions $\chi_{y}$ define a frame by Proposition \ref{prop:s-cover_frame}.
\end{proof}

\subsection{The twisted groupoid algebra and its $K$-theory}

Given our transversal groupoid, we fix a normalised  $2$-cocycle 
$\sigma:\calG^{(2)}\to \T$ (or $\{\pm 1\}$ in the real case). 
Our central motivation for working with twisted groupoid algebras comes 
from the following example.

\begin{example}[Magnetic twists] \label{ex:mag_twist}
For the transversal groupoid, we can encode 
the action of a magnetic field that twists the 
translation action of the lattice. 
Working first with the continuous hull $\Omega_\calL \rtimes \R^d$, we 
follow~\cite[Section 2.2]{BLM13} and define a 
$2$-cocycle,
$$
   \sigma:\R^d\times\R^d \to \calU(C(\Omega_\calL)), \qquad 
   \sigma(x,y) = \exp\big( -i\Gamma \langle 0, x, x+y \rangle \big)
$$
where $\Gamma \langle 0, x, x+y \rangle$ is the magnetic 
flux through the triangle defined by the points $0,x,x+y\in \R^d$. 
The magnetic field need not be constant over $C(\Omega_\calL)$ and can 
generally be described by a continuous map $B: \Omega_\calL \to \bigwedge^2 \R^d$, 
where $\Gamma \langle x,y,z \rangle = \int_{\langle x,y,z \rangle} \! B_\omega$ 
and $\langle x,y,z \rangle \subset \R^{2d}$ is the triangle with corners 
$x,y,z \in \R^d$. 
If the magnetic field is constant over $\Omega_\calL$, then our general flux equation 
can be simplified by a skew-symmetric matrix $B$ with 
$$
   \sigma( x,y) = \exp \big( -i \langle x, B(x+y) \rangle \big) = 
   \exp \big( -i \langle x, By \rangle \big).
$$

Our choice of $2$-cocycle on the crossed product $C(\Omega_\calL)\rtimes_\sigma \R^d$ 
restricts to a $2$-cocycle on the transversal groupoid, which we also denote by $\sigma$. 
Namely, we define 
$$
 \sigma((\omega,x),(T_{-x}\omega,y)) = 
 \exp\big( -i\Gamma_{\calL^{(\omega)}} \langle 0, x, x+y \rangle \big)
$$
where $\Gamma_{\calL^{(\omega)}} \langle 0, x, x+y \rangle$ is the magnetic 
flux through the triangle defined by the points $0,x,x+y\in\calL^{(\omega)}$. 
We note that our twist will always be trivial for $d=1$ and is 
normalised because
$$
  \sigma((\omega,x),(T_{-x},-x)) = 
  \exp\big( -i\Gamma_{\calL^{(\omega)}} \langle 0, x, 0 \rangle \big) = 1.
$$
The cocycle condition on $\sigma$ 
translates into the condition that 
for $x,\, x+y,\, x+y+z\in \calL^{(\omega)}$,
$$
  \Gamma_{\calL^{(\omega)}}\langle 0,x,x+y \rangle + 
     \Gamma_{\calL^{(\omega)}} \langle 0,x+y,x+y+z \rangle 
  = \Gamma_{\calL^{(\omega)}}\langle 0,x,x+y+z \rangle + 
      \Gamma_{\calL^{(T_{-x}\omega)}}\langle 0, y, y+z \rangle,
$$
which follows from Stokes' Theorem and the observation that 
$$
  \Gamma_{\calL^{(T_{-x}\omega)}}\langle 0, y, y+z \rangle 
  = \Gamma_{\calL^{(\omega)}} \langle x , x+y, x+y+z \rangle.
$$
\end{example}

Given our groupoid $\calG$ and cocycle $\sigma$, we can construct the 
groupoid $C^*$-algebra by the method given in Section \ref{sec:twisted_gpoid_alg}, 
acting on the $C^*$-module over the unit space. The $K$-theory of the twisted groupoid 
algebra is used to describe topological phases of gapped Hamiltonians, which we will then 
pair with $KK$-cycles to obtain numerical labels for these phases. 
In the absence of a $2$-cocycle twist, the continuous dynamical system $(\Omega_\calL, T, \R^d)$ can 
be described via the crossed product groupoid $\Omega_\calL \rtimes \R^d$, which is then groupoid-equivalent 
to $\calG$. Applying the equivalence theorem of~\cite{MRW,SimsWilliams} and the Connes--Thom 
isomorphism~\cite{ConnesThom},
$$
  K_\ast(C^*_r(\calG)) \cong K_\ast (C(\Omega_\calL)\rtimes \R^d) \cong K_{\ast-d}(C(\Omega_\calL)) 
  \cong K^{d-\ast}(\Omega_\calL),
$$
in both real and complex $K$-theory. This result remains true for twists  by $2$-cocycles.

\begin{prop}Let $\mathcal{L}$ be a Delone set and 
$\sigma:\big(\Omega_{\mathcal{L}}\rtimes\mathbb{R}^{d}\big)^{(2)}\to \mathbb{T}$ (or $\{\pm 1\}$ in the real case) a continuous $2$-cocycle. 
Then the twisted groupoid $C^{*}$-algebra $C^{*}(\mathcal{G},\sigma)$ is 
Morita equivalent to the twisted crossed product $C(\Omega_{\mathcal{L}})\rtimes_{\sigma}\mathbb{R}^{d}$ 
and there is an isomorphism $K_{*}(C^{*}_{r}(\mathcal{G},\sigma))\to K^{d-*}(\Omega_{\mathcal{L}})$.
\end{prop}
\begin{proof}
As the $2$-cocycle on $\calG$ comes from the restriction of a $2$-cocycle on 
$\Omega_\calL \rtimes \R^d$, we can apply~\cite[Theorem 9.1]{Daenzer09}, which gives 
that $C^*_r(\calG,\sigma)$ is Morita equivalent to the twisted crossed 
product $C(\Omega_\calL) \rtimes_\sigma \R^d$.
Then, by Packer--Raeburn stabilisation, \cite[Section 3]{PR89}, and the Connes--Thom 
isomorphism we obtain that
$$
  K_\ast(C^*_r(\calG,\sigma)) \cong K_\ast (C(\Omega_\calL) \rtimes_\sigma \R^d) \cong 
   K_\ast ((C(\Omega_\calL) \otimes \mathbb{K})\rtimes \R^d) \cong K_{\ast-d} (C(\Omega_\calL)\otimes\mathbb{K}) 
   \cong K^{d-\ast}(\Omega_\calL).
$$
Hence the $K$-theory of the twisted groupoid $C^{*}$-algebra reduces to that 
of the continuous hull $\Omega_\calL$. 
\end{proof}

Let us emphasise that the computation 
of the $K$-theory of $\Omega_\calL$ is highly non-trivial. 
A homological description of the $K$-theory of $\Omega_{\calL}$ for a large 
class of tilings with finite local complexity is given in~\cite{FHK} as well as 
computational techniques. See also the review~\cite{HuntonTiling}.
 In the case that 
$\calL$ is repetitive, aperiodic and has finite local complexity, one can characterise 
$\Omega_\calL$ as a projective limit~\cite{AndersonPutnam, KellendonkPutnam, BBG06} and 
compute its $K$-theory using the 
Pimsner--Voiculescu spectral sequence~\cite{BelSav} (adapted 
from the spectral sequence used by Kasparov~\cite[{\S}6.10]{KasparovNovikov}), 
whose $E_{2}$-page is isomorphic to the \v{C}ech cohomology of $\Omega_\calL$ with 
integer coefficients. 
In the case of low-dimensional substitution tilings with finite local complexity and a primitive and 
injective substitution map, Gon\c{c}alves--Ramirez-Solano relate the 
\v{C}ech cohomology of $\Omega_\calL$ to 
the $K$-theory of the groupoid $C^*$-algebra of the unstable equivalence relation 
on $\Omega_\calL$ (note that this groupoid $C^*$-algebra is Morita equivalent 
to $C^*_r(\calG)$)~\cite[Theorem 2.3]{TilingKTheory}. 
See~\cite{TilingKTheory} for a detailed exposition on these (and other) matters.

In contrast to $\Omega_\calL$, the $K$-theory of the transversal $\Omega_0$ is often 
very simple to compute. If $\calL$ is a Delone set with finite local complexity, then 
by Proposition \ref{prop:tranversal_properties} $\Omega_0$ is totally disconnected and, 
by continuity of the $K$-functor, $K_\ast(C(\Omega_0)) \cong C(\Omega_0, K_\ast(\mathbb{F}))$, 
where $\mathbb{F} = \C$ or $\R$.

\subsubsection{The bulk $KK$-cycle}  \label{subsec:bulk_cycle}
We now introduce our main tool to extract numerical invariants from $K_{*}(C^{*}_{r}(\mathcal{G},\sigma))$ (see Section \ref{sec:Applications}).
The transversal groupoid $\calG$ is \'{e}tale and the cocycles 
$\hat{c}_{k}:\calG \to \R^k$, $\hat{c}_{k}(\omega,x) = (x_1,\ldots, x_k)$ are exact by Proposition 
\ref{prop:exact_cocycles}. Hence we can construct a family of unbounded $KK$-cycles 
for $\calG$ by Proposition \ref{prop:bulk_kasmod}.

 We call the special case 
$c(\omega,x):=\hat{c}_{d}(\omega,x) = x$, where $\Ker(c) \cong \calG^{(0)}\cong \Omega_0$, the bulk $KK$-cycle as it spans 
all dimensions of the lattice, where the terminology is taken from topological phases. 
Explicitly, 
\begin{equation} \label{eq:bulk_K-cycle}
  {}_d\lambda_{\Omega_0} = \bigg( C_c(\calG, \sigma)\hat\otimes Cl_{0,d},\, E_{C(\Omega_0)} \hat\otimes 
   \bigwedge\nolimits^{\! *} \R^d, \, \sum_{j=1}^d X_j \hat\otimes \gamma^j \bigg),
\end{equation}
is an unbounded Kasparov module, 
with $X_j$ the self-adjoint regular operator $(X_j f)(\omega,x) = x_jf(\omega,x)$ on 
$E_{C(\Omega_0)}$. 
We will consider other unbounded $KK$-cycles from cocycles on $\calG$ and their 
properties in Section \ref{sec:factorisation}.

\subsection{One dimensional Delone sets as Cuntz--Pimsner algebras} \label{sec:CP_description}
Given an $(r,R)$-Delone set $\calL\subset\R^d$, we have constructed the 
groupoid $\calG$ and a class in $KK^d(C^*_r(\calG,\sigma),C(\Omega_0))$ that encodes 
the translation action on the transversal. For the case $d=1$ and trivial 
cocycle $\sigma=1$, we now give an equivalent description of 
$C^*_r(\calG)$ as a Cuntz--Pimsner algebra. We also find that 
the Kasparov cycle from Equation \eqref{eq:bulk_K-cycle} is equivalent 
to the class of the defining extension of 
the Cuntz--Pimsner algebra.  We remark that a similar construction is done in~\cite{WilliamsonThesis} 
that includes higher dimensions but for more restrictive substitution tilings. Here we leave open the question of higher dimensions
where, in analogy with crossed products by $\mathbb{Z}^{d}$, a description of $C^{*}(\mathcal{G},\sigma)$ as 
an iterated Cuntz--Pimnser algebra or  
$C^{*}$-algebra of a product system~\cite{SimsYeend} is a natural aim.

In the case $d=1$, recall the cocycle $c(\omega,x)=x \in \R$ and write
\[
\mathcal{G}^{(0)}=\mathcal{G}_{0}:=c^{-1}(0),\quad \mathcal{G}_{1}:=c^{-1}(r,R), \quad \quad 
\mathcal{G}_{-1}:=c^{-1}(-R,-r).
\]

\begin{lemma}\label{unperf} 
Let $(\omega,x)\in\mathcal{G}$ and $x>0$. There exist $(\omega_{j},x_{j})\in \mathcal{G}_{1}$, $j=1,\cdots n$ such that 
\[
(\omega,x)=\prod_{j=1}^{n}(\omega_{j},x_{j}),
\]
and this decomposition is unique. A similar satement holds for $(\omega,x)$ with $x<0$ where we replace $\mathcal{G}_{1}$ with $\mathcal{G}_{-1}$.
\end{lemma}
\begin{proof} 
The lattice $\mathcal{L}^{(\omega)}\subset \mathbb{R}$ is discrete, so we can order it as
\begin{equation}\label{order}
\mathcal{L}^{(\omega)}=\{y_n\}_{n\in\mathbb{Z}}, \quad y_{0}=0,\quad y_{j}<y_{j+1},\quad r<y_{j+1}-y_{j}<R.
\end{equation}
Then $(\omega,x)=(\omega,y_n)$ for some $n$ and we set
\[\omega_{j}:=T_{-x_{j-1}}\omega,\quad x_{j}:=y_{j}-y_{j-1}.\] 
It follows that
\[
(\omega,x)=(\omega,y_{n})=(\omega, y_{1})\cdot(T_{-y_{1}}\omega, y_{2}-y_{1})\cdots 
   (T_{-y_{n-1}}\omega, y_{n}-y_{n-1})=\prod_{j=1}^{n}(\omega_{j},x_{j}),
\]
as claimed. Suppose that 
\[(\omega,x)=\prod_{j=1}^{m}(\eta_{j},z_{j}),\]
is another such decomposition and assume without loss of generality that $m\geq n$. Then $\eta_{1}=\omega_{1}=\omega$. 
Since $z_{1},x_{1}\in \mathcal{L}^{(\omega)}\cap (r,R)$ it follows that $z_{1}=x_{1}$. 
This argument can be repeated to find $\eta_{j}=\omega_{j}$ and $x_{j}=z_{j}$ for $1\leq j\leq n$, so 
the decompositions are the same if $m=n$.  
If $m>n$, then
\[
(\eta_{n+1},0)=(\eta_{n+1},z_{n+1})\cdots (\eta_{m}, z_{m})=(\eta_{n+1},z_{n+1}+\cdots+z_{m}),
\]
so $0<z_{n+1}+\cdots+z_{m}=0$, a contradiction.
\end{proof}

The previous result indicates that the $1$-dimensional tranversal groupoid is in some 
sense generated by $\calG_1 = c^{-1}(r,R)$. This then gives us a pathway to recharacterise 
$C^*_r(\calG)$ as a Cuntz--Pimsner algebra. The following result comes from standard arguments.
\begin{lemma}
Suppose $d=1$ and let $E^{(r,R)}_{C(\Omega_0)}$ be the completion of $C_c( \calG_1 )$ in $C^*_r(\calG)$. 
Then $E^{(r,R)}_{C(\Omega_0)}$ is a $C^*$-bimodule over $C(\Omega_0)$ with structure 
\begin{align*}
    &(f_1\mid f_2)_{C(\Omega_0)}(\omega) = (f_1^* \ast f_2)(\omega, 0),  &&{}_{C(\Omega_0)}(f_1\mid f_2)(\omega) = (f_1\ast f_2^*)(\omega,0), \\
    &\hspace{0cm}(g_1\cdot f \cdot g_2)(\omega,x) = g_1(\omega) f(\omega,x) g_2(T_{-x}\omega).
\end{align*}
An analogous result also holds for the completion of $C_c(\calG_{-1})$.
\end{lemma}

Denote by $d:\mathcal{G}\to\mathbb{Z}$ the map that associates to an element $(\omega,x)$ the 
integer $n$ for which $x=y_{n}$ with $\mathcal{L}^{(\omega)}=\{y_{n}\}_{n\in\mathbb{Z}}$ as in 
Equation \eqref{order} on page \pageref{order}. We call $d(\omega,x)$ the \emph{degree} of $(\omega,x)$.
\begin{prop}  \label{prop:CP_alg_iso}
The map $d:\mathcal{G}\to \mathbb{Z}$ is a continuous 1-cocycle that is unperforated in the sense of \cite{RRSgpoid}. 
Consequently $C^{*}_r(\mathcal{G})$ is isomorphic to the Cuntz--Pimsner algebra $\mathcal{O}_{E^{(r,R)}}$ 
and for $n>0$ the sets
\[
\mathcal{G}_{\pm n}:=\{\xi_{1}\cdots \xi_{n}: \xi_{i}\in\mathcal{G}_{\pm 1}\},
\]
define a decompositon $\mathcal{G}=\bigcup_{n\in\mathbb{Z}}\mathcal{G}_{n}$
into clopen subsets. 
\end{prop}
\begin{proof}
We prove that $d$ is locally constant. Let $(\omega,x)\in\mathcal{G}$ and choose 
$\mu,y$ such that $x\in  B(y;\varepsilon)\subset B(0;\mu)$ with $\varepsilon<r/2$. 
Then $(\omega,x)\in V_{(p,\mu,y,\varepsilon)}$ for $p=\mathcal{L}^{(\omega)}\cap B(0;\mu)$ and consider 
$(\eta,z)\in V_{(p,\mu,y,\varepsilon)}$. Since
\[
x,z\in B(y;\varepsilon)\subset \mathcal{L}^{(\omega)}\cap B(0;\mu)=\mathcal{L}^{(\eta)}\cap B(0;\mu),
\]
and $\varepsilon<r/2$ it follows that $x=z$. Then since
$$\mathcal{L}^{(\omega)}\cap B(0;\mu)=\mathcal{L}^{(\eta)}\cap B(0;\mu)$$
it follows that $d(\omega,x)=d(\eta,z)$. Thus the degree is locally constant on $\mathcal{G}$.
By Lemma \ref{unperf} the degree is additive, and it thus defines a continuous 1-cocycle with 
$$
d^{-1}(n)=\mathcal{G}_{n}:=(\mathcal{G}_{\frac{n}{|n|}})^{|n|},
$$ 
and each $\mathcal{G}_{n}$ is clopen. We thus satisfy the hypothesis of~\cite[Proposition 10]{RRSgpoid}, 
which gives the isomorphism $\mathcal{O}_{E^{(r,R)}}\to C_r^{*}(\mathcal{G})$.
\end{proof}

\subsubsection{The Cuntz--Pimsner extension class}
We extend the equivalence of the one-dimensional transversal groupoid with a Cuntz--Pimsner algebra 
to a compatibility of the bulk $KK$-cycle from Equation \eqref{eq:bulk_K-cycle} on page \ref{eq:bulk_K-cycle} 
with the class in 
$KK^1(\calO_{E^{(r,R)}}, C(\Omega_0))$ that comes from the defining extension of $\calO_{E^{(r,R)}}$.

\begin{lemma}
The $C^*$-module $E^{(r,R)}_{C(\Omega_0)}$ is a self-Morita equivalence bimodule (SMEB).
\end{lemma}
\begin{proof}
Given $\omega\in\Omega_0$ with ordering $\calL^{(\omega)}=\{x_n\}_{n\in\Z}$ with 
$x_0=0$ and $x_n-x_{n-1} \in (r,R)$, a generic element in 
$c^{-1}(r,R)$ can be written as $(T_{-x_n}\omega, x_{n+1}-x_n)$.
We first compute
\begin{align*}
    \big(  {}_{C(\Omega_0)}(f_1\mid f_2)\cdot f_3 \big)(T_{-x_n}\omega,x_{n+1}-x_n) 
      &=  (f_1 \ast f_2^*)(T_{x_n}\omega, 0) f_3(T_{-x_n}\omega,x_{n+1}-x_n)  \\
      &\hspace{-3cm}= f_1(T_{-x_n}\omega, x_{n+1}-x_n) f_2^*(T_{x_{n+1}}\omega,x_n-x_{n+1})  f_3(T_{-x_n}\omega,x_{n+1}-x_n) \\
      &\hspace{-3cm}= f_1(T_{-x_n}\omega, x_{n+1}-x_n) \ol{f_2(T_{-x_n}\omega, x_{n+1}-x_n)}  f_3(T_{-x_n}\omega,x_{n+1}-x_n)
\end{align*}
and then compare to 
\begin{align*}
  \big( f_1\cdot (f_2\mid f_3)_{C(\Omega_0)}\big)(T_{-x_n}\omega,x_{n+1}-x_n) 
    &= f_1(T_{-x_n}\omega,x_{n+1}-x_n) (f_2^* \ast f_3)(T_{x_{n+1}}\omega,0) \\
    &\hspace{-3cm}= f_1(T_{-x_n}\omega,x_{n+1}-x_n) f_2^*(T_{x_{n+1}}\omega, x_n-x_{n+1}) f_3(T_{x_n}\omega, x_{n+1}-x_n) \\
    &\hspace{-3cm}= f_1(T_{-x_n}\omega, x_{n+1}-x_n) \ol{f_2(T_{-x_n}\omega, x_{n+1}-x_n)}  f_3(T_{-x_n}\omega,x_{n+1}-x_n)
\end{align*}
as required. Lastly the bi-module is full as by the compactness of $\Omega_0$, any $g\in C(\Omega_0)$ can be written 
\begin{align*}
 g(\omega) &= f_1(\omega, x_1) \ol{f_2(\omega,x_1)} = {}_{C(\Omega)}(f_1\mid f_2)(\omega)   \\
   &= \ol{ \tilde{f}_1(T_{-x_1}\omega,-x_1)} \tilde{f}_2(T_{-x_1}\omega, -x_1) =  ( \tilde{f}_1\mid  \tilde{f}_2)_{C(\Omega_0)}(\omega)
\end{align*}
for some $f_1,f_2, \tilde{f}_1, \tilde{f}_2 \in C_c(c^{-1}(r,R))$.
\end{proof}

Given the bimodule $E^{(r,R)}_{C(\Omega_0)}$, the Cuntz--Pimsner algebra 
$\calO_{E^{(r,R)}}$ is defined 
by a short exact sequence
\begin{equation} \label{eq:CP_SES}
   0 \to \mathbb{K}\left( (F_{E^{(r,R)}})_{C(\Omega_0)} \right)  \to \calT_{E^{(r,R)}} \to \calO_{E^{(r,R)}} \to 0,
\end{equation}
where $\calT_{E^{(r,R)}}$ is generated by creation and annihilation operators on the 
Fock module $F_{E^{(r,R)}}= \bigoplus_{n\geq 0} (E^{(r,R)})_{C(\Omega_0)}^{ \otimes n}$ 
with $(E^{(r,R)})_{C(\Omega_0)}^{ \otimes 0}:= C(\Omega_0)$.

The extension Equation \eqref{eq:CP_SES}  
gives a  $KK^1$-class $[\mathrm{ext}]$ which can be composed with the 
natural Morita equivalence between $\mathbb{K}\left( (F_{E^{(r,R)}})_{C(\Omega_0)} \right)$ 
and $C(\Omega_0)$. Thus the Cuntz--Pimsner algebra gives an element 
$[\mathrm{ext}] \hat\otimes_{\mathbb{K}} [F_E] \in KK^1(\calO_{E^{(r,R)}}, C(\Omega_0))$. 
We can use~\cite[Section 3.1]{RRS} to construct an unbounded 
Kasparov module  representing this class.

Using the conjugate module $\ol{E}_{C(\Omega)}^{(r,R)}$, we define for $n<0$, 
$\big(E^{(r,R)} \big)^{ \otimes n} = \big(\ol{E}^{(r,R)} \big)^{\otimes |n|}$. We can then 
consider the bi-infinite Fock module
$$
   F_{E,\Z} := \bigoplus_{n\in \Z} \big(E^{(r,R)} \big)_{C(\Omega_0)}^{ \otimes n},
$$
which carries a natural representation of $\calO_{E}$ and an operator making it into a $KK$-cycle.
\begin{prop}[\cite{RRS}, Theorem 3.1] \label{prop:CP_Kasmod}
Define an operator $N$ on the algebraic direct sum 
$\bigoplus_{m\in\mathbb{Z}}^{\textnormal{alg}} E^{\otimes m}$ by $N\xi = n \xi$ for $\xi \in E^{\otimes n}$.
There is a $*$-homomorphism $\calO_{E^{(r,R)}} \to \End^*\big((F_{E,\Z})_{C(\Omega_0)}\big)$ such
that $S_f\cdot \xi := f \otimes \xi$ for all $f \in E^{(r,R)}$ and $\xi \in \big(E^{(r,R)} \big)^{\otimes n}$. 
The triple $\big( \calO_{E^{(r,R)}}, (F_{E,\Z})_{C(\Omega_0)}, N)$ is an unbounded Kasparov module 
that represents the class $[\mathrm{ext}] \hat\otimes_{\mathbb{K}} [F_E] \in KK^1(\calO_{E^{(r,R)}}, C(\Omega_0))$.
\end{prop}

\begin{cor} \label{cor:CP_extension_KK_equiv}
The odd Kasparov module from Proposition \ref{prop:CP_Kasmod} defines the same class 
in \newline $KK^1(C^*_r(\calG), C(\Omega_0))$ as the bulk Kasparov module ${}_{d}\lambda_{\Omega_0}$ 
from Equation \eqref{eq:bulk_K-cycle} on page \ref{eq:bulk_K-cycle} with $d=1$.
\end{cor}
\begin{proof}
The $C^{*}$-algebras are isomorphic by Proposition \ref{prop:CP_alg_iso}. 
Furthermore, the positive semi-splitting from both 
the groupoid and Cuntz--Pimsner Kasparov modules is the projection onto elements with 
non-negative cocycle values. Hence the extensions are equivalent, which also gives 
equivalence within $KK^1$.
\end{proof}

By the presentation of $C^*_r(\calG)$ as a Cuntz--Pimser algebra, we can use the long (cyclic) exact 
sequence as a tool for the computation of $K_\ast(C^*_r(\calG))$. Namely, for complex algebras,
\begin{displaymath}
  \xymatrix{ 
    K_0(C(\Omega_0)) \ar[rrr]^{\otimes([C(\Omega_0)] - [E^{(r,R)}])} & &  & K_0(C(\Omega_0)) \ar[rrr]^{\iota_\ast} & &   & K_0(C^*_r(\calG)) \ar[d]^{\partial} \\
    K_1(C^*_r(\calG)) \ar[u]^{\partial} & & &  K_1(C(\Omega_0)) \ar[lll]^{\iota_\ast}  & & &  K_1(C(\Omega_0)) \ar[lll]^{\otimes([C(\Omega_0)] - [E^{(r,R)}])}
   }
\end{displaymath}
where the map $K_\ast(C(\Omega_0)) \xrightarrow{\otimes [E^{(r,R)}]} K_\ast (C(\Omega_0))$ comes 
from the internal product of the $K$-theory class with the element 
$[E^{(r,R)}] \in KK(C(\Omega_0),C(\Omega_0))$.
There is an analogous but longer exact sequence for real $C^*$-algebras,
\begin{align*}
 \cdots \to KO_j(C(\Omega_0)) \xrightarrow{\otimes([C(\Omega_0)] - [E^{(r,R)}])} KO_j(C(\Omega_0)) \xrightarrow{\iota_\ast} 
      KO_j(C^*_r(\calG)) \xrightarrow{\partial} KO_{j-1}(C(\Omega_0)) \to \cdots.
\end{align*}
By the Morita equivalence of $C^*_r(\calG)$ and $C(\Omega_\calL)\rtimes \R$, we know that 
$K_\ast(C^*_r(\calG)) \cong K_{\ast -1}(C(\Omega_\calL))$ by the Connes--Thom isomorphism. 
As the $K$-theory of $C(\Omega_\calL)$ is generally quite difficult to compute, the Pimsner exact 
sequence for $C^*_r(\calG)$ is a helpful tool for such $K$-theory computations.
For example, if $K_1(C(\Omega_0))=0$ (e.g. if $\calL$ has finite local complexity), then 
we immediately obtain that
\begin{align*}
   K_0(C^*_r(\calG)) \cong \mathrm{Coker}(1-[E^{(r,R)}]),  &&K_1(C^*_r(\calG)) \cong \Ker(1-[E^{(r,R)}]).
\end{align*} 
Hence, for a one-dimensional lattice $\calL$ with finite local complexity, 
\begin{align*}
  K_0(C(\Omega_\calL)) \cong \Ker(1-[E^{(r,R)}]),  &&K_1(C(\Omega_\calL)) \cong \mathrm{Coker}(1-[E^{(r,R)}]).
\end{align*}
Of course, this result is  
restricted to one-dimensional lattices or tilings. A description of $C^*_r(\calG)$ 
for higher dimensions using the $C^*$-algebra of a product system or as an iterated Cuntz-Pimsner algebra 
may be possible. We leave this analysis to future research.

\begin{remark}
As a brief cautionary remark, we note that our bimodule ${}_{C(\Omega_0)}E^{(r,R)}_{C(\Omega_0)}$ 
looks quite similar but is different from the crossed product bimodule ${}_\alpha A_A$ with $\alpha:\Z \to \Aut(A)$ and 
 such that $\calO_\alpha \cong A\rtimes_\alpha \Z$. Indeed, given $\omega\in \Omega_0$ 
 and $x_1\in \calL^{(\omega)}\cap (r,R)$,  there is no 
guarantee that $T_{-2x_1}\omega \in \Omega_0$ as would be the case for a $\Z$-action.
\end{remark}

\section{Factorisation and the bulk-boundary correspondence} \label{sec:factorisation}

A key attribute of the operator algebra approach to topological phases via crossed products 
is that both bulk and boundary systems can be treated under the same general framework with 
an extension of $C^*$-algebras linking the two systems. 
Namely, up to stabilisation the 
edge algebra can be descibed via $C(\Omega)\rtimes_\sigma \Z^{d-1}$ and, we can recover 
the bulk algebra by the iterated crossed product 
$(C(\Omega)\rtimes_\sigma \Z^{d-1})\rtimes \Z \cong C(\Omega)\rtimes_\sigma \Z^d$ for 
normalised twists. 

In this section we use the groupoid cocycle $c_d:\calG \to \R$ to consider the closed 
subgroupoid $\Upsilon = \Ker(c_d)$. This subgroupoid is too small to completely model 
an edge system but is groupoid equivalent to one that we argue encodes the translation 
dynamics on the transversal in $(d-1)$-directions. Furthermore, we show that the subgroupoid 
$\Upsilon$ gives rise to a canonical \emph{bulk-boundary extension} of reduced $C^*$-algebras that generalises 
the Toeplitz extension for crossed products. In particular, we use this extension to factorise the groupoid 
$KK$-cycle into a product of a $(d-1)$-dimensional system and the bulk-boundary extension 
that recovers the bulk system.

\subsection{The edge groupoid} \label{sec:edge_gpoid_toeplitz}

We now apply our results on twisted groupoid equivalences to the transversal groupoid 
and the bulk-boundary short exact sequence.

Recall the groupoid cocycle $c_d:\calG\to\R$, $c_d(\omega,x) = x_d$. Because $c_d$ is exact, 
we can apply the results from Section \ref{sec:bulk_Kasmod} and construct an 
unbounded $KK$-cycle.
We consider the closed subgroupoid $\Upsilon = \Ker(c_d)$, namely
$$
  \Upsilon = \big\{ (\omega, y)\in \Omega_0 \times\R^{d-1} \, :\, T_{(-y,0)}\omega \in \Omega_0 \big\}.
$$
with multiplication, range and source maps inherited from $\calG$. 
Furthermore, the restriction of 
$\sigma$ to $\Upsilon$ gives a well defined $2$-cocycle for $\Upsilon$. 
Recalling Section \ref{sec:twisted_gpoid_equiv}, 
the restriction map 
$$\rho_\Upsilon:C_c(\calG,\sigma) \to C_c(\Upsilon,\sigma)$$
defines a $(C^{*}_{r}(\mathcal{G},\sigma),C^{*}_{r}(\Upsilon,\sigma))$-bimodule $E_{C^*_r(\Upsilon,\sigma)}$.
Applying Proposition \ref{prop:bulk_kasmod} to the cocycle 
$c_{d}:\mathcal{G}\to\mathbb{R}$ and writing $X_{d}:=D_{c_{d}}$, gives us the following.

\begin{prop}[\cite{MeslandGpoid}, Theorem 3.9] \label{prop:extension_like_class}
The triple
$$
 {}_{d}\lambda_{d-1} =  \left( C_c(\calG,\sigma)\hat\otimes Cl_{0,1}, \, E_{C^*_r(\Upsilon,\sigma)} \hat\otimes 
    \bigwedge\nolimits^{\!*}\R, X_d \hat\otimes \gamma \right)
$$
is a real or complex unbounded Kasparov module.
\end{prop}

The groupoid $\Upsilon$ is too small to be thought of as representing 
an edge system. Instead, we will consider the groupoid $\calG \ltimes \calG/ \Upsilon$ 
whose twisted reduced $C^*$-algebra is Morita equivalent to $C^*_r(\Upsilon,\sigma)$, cf. Section \ref{sec:twisted_gpoid_equiv}.

The cocycle $c_d$ determines the subset $\Ran(c_d)\subset \R^d$ (which need not be a subgroup). Having fixed 
this set, the groupoid 
$\calG \ltimes \calG/ \Upsilon$ allows us to put a groupoid 
structure back into our system with the translation action in 
$(d-1)$-directions. 

The space $\calG/ \Upsilon$ is given by equivalence classes 
of elements $[(\omega,x)]\in\calG$ under the relation
\[(\omega,x)\sim (\omega',x')\,\Leftrightarrow \quad \exists (T_{-x}\omega,(y,0))\in\Upsilon\quad (\omega,x+(y,0)) = (\omega',x').\] Hence  the quotient 
$\calG/\Upsilon$ can be described by equivalence classes of  pairs $[(\omega,x_d)]$ with 
$(\omega,x_d) \in \Omega_0 \times \Ran(c_d)$. We have the presentation of 
$\calG \ltimes \calG/ \Upsilon$ by pairs
$$
  \calG \ltimes \calG/ \Upsilon \cong 
    \left\{ ((\omega,x),[(\omega',y_d)])\,: \, r_{\calG}(c_d^{-1}(y_d)) = T_{-x}\omega \right\} 
    \cong \left\{ ((\omega,x),[(T_{-x}\omega,y_d)]) \right\} \subset \calG\times \calG/\Upsilon.
$$
Recall that $(\omega,x)\in\calG$ if $x\in\calL^{(\omega)}$. Our presentation says that 
$((\omega,x),[(T_{-x}\omega,y_d)]) \in \calG\ltimes \calG/\Upsilon$ if there is some 
$u\in \R^{d-1}$ such that $x, x+ (u,y_d) \in \calL^{(\omega)}$. The unit space is given by 
$$\left(\calG \ltimes \calG/ \Upsilon \right)^{(0)} = \calG/\Upsilon,$$ 
and the groupoid structure is determined by 
\begin{align*}
  &\hspace{1.5cm} s((\omega,x),[(T_{-x}\omega,y_d)]) = [(T_{-x}\omega,y_d)], \quad  r((\omega,x),[(T_{-x}\omega,y_d)]) = [(\omega,x_d+y_d)], \\
    &\hspace{3cm} ((\omega,x),[(T_{-x}\omega,y_d)])^{-1} = ((T_{-x}\omega,-x),[(\omega,x_d+y_d)]), \\
  &\hspace{0cm} ((\omega,x),[(T_{-x}\omega,y_d)])\cdot ((T_{-x}\omega,z),[(T_{-x-z}\omega,y_d-z_d)] ) = ((\omega,x+z),[(T_{-x-z}\omega,y_d-z_d)]).
\end{align*}
We note that for $((T_{-x}\omega,z),[(T_{-x-z}\omega,y_d-z_d)] )$ to be in $\calG\ltimes \calG/\Upsilon$, 
there must be some 
$v\in\R^{d-1}$ such that $x+(v,y_d) \in \calL^{(\omega)}$. Because 
$((\omega,x),[(T_{-x}\omega,y_d)])\in \calG\ltimes \calG/\Upsilon$ 
implies $x+(u,y_d) \in \calL^{(\omega)}$ 
for some $u\in \R^{d-1}$, 
the groupoid multiplication involves a translation in 
$(d-1)$-dimensions only. Thus we see that the groupoid $\calG \ltimes \calG/ \Upsilon$ 
models the dynamics of the transversal $\Omega_0$ relative to the fixed set $\Ran(c_d)$.
We  use the 2-cocycle $\sigma$ on $\calG$ to define a 2-cocycle on 
$\calG \ltimes \calG/\Upsilon$ via
$$
  \sigma\big( ((\omega,x),[(T_{-x}\omega,y_d)]), ((T_{-x}\omega,z),[(T_{-x-z}\omega,y_d-z_d)]) \big) = 
      \sigma((\omega,x),(T_{-x}\omega,z)).
$$
Applying Proposition \ref{prop:twisted_morita_reduced}, we obtain the following.

\begin{prop} 
The $C^*$-module $E_{C^*_r(\Upsilon,\sigma)}$ is a Morita equivalence bimodule between 
$C^*_r(\calG \ltimes \calG/ \Upsilon,\sigma)$ and $C^*_r(\Upsilon,\sigma)$. 
In particular, 
$C^*_r(\calG \ltimes \calG/ \Upsilon,\sigma) \cong \mathbb{K}(E_{C^*_r(\Upsilon,\sigma)})$.
\end{prop}

The 
Morita equivalence bimodule gives an invertible element in 
$KK(C^*_r(\calG\ltimes \calG/\Upsilon, \sigma), C^*_r(\Upsilon,\sigma))$.
From the 
perspective of index theory, we can work with either 
$\Upsilon$ or $\calG \ltimes \calG/ \Upsilon$. While we consider 
$\calG \ltimes \calG/ \Upsilon$ to be our edge groupoid, the subgroupoid $\Upsilon\subset \mathcal{G}$ 
will be easier to work with for some of our mathematical arguments.

\subsection{The bulk-boundary extension}

Because the Kasparov module from Proposition \ref{prop:extension_like_class} 
comes from an exact $\R$-valued cocycle, we can construct an extension 
of $C^*$-algebras.
As in Lemma \ref{lem:pi_c_almost_proj} in Section \ref{subsec:exact_extension}, we fix an $\varepsilon >0$ and a function $\chi_{+}$ to define 
a self-adjoint operator 
$\Pi_{d}:=\Pi_{c_{d}} \in \End_{C^*_r(\Upsilon,\sigma)}^*(E)$  on the $C^*$-module $E_{C^*_r(\Upsilon,\sigma)}$ satisfying 
$\Pi_{d}^2- \Pi_{d} \in \mathbb{K}_{C^*_r(\Upsilon,\sigma)}(E)$.
Since the Delone set $\mathcal{L}$ is relatively dense,  the cocycle 
$c_d$ takes arbitrarily large values in each $r$-fibre and is $r$-unbounded. Therefore the map 
$$\varphi: C^*_r(\calG,\sigma) \to \calQ(E_{C^*_r(\Upsilon,\sigma)}),\quad \varphi(a) = q(\Pi_{d} a \Pi_{d}),$$ is 
injective by Proposition \ref{prop:injective_busby}. Hence we can construct 
the generalised Toeplitz extension 
\begin{equation} \label{eq:bulkedge_extension}
   0 \to C^*_r(\calG \ltimes \calG/ \Upsilon, \sigma) \to 
     C^*( \Pi_{d}  C^*_r(\calG,\sigma) \Pi_d, C^*_r(\calG \ltimes \calG/ \Upsilon, \sigma) ) \to C^*_r(\calG,\sigma) \to 0
\end{equation}
with completely positive semi-splitting $a\mapsto \Pi_d a \Pi_d$ and Busby invariant $\varphi$. 

In the case that our lattices have a canonical $\Z^d$-labelling, then this extension 
reduces to the usual bulk-boundary short exact sequence considered in~\cite{PSBbook}.
 In the case $d=1$ we have $\Upsilon \cong \calG^{(0)}\cong \Omega_0$, and the extension \eqref{eq:bulkedge_extension}
is equivalent to the Toeplitz--Cuntz--Pimsner  extension for $C^*_r(\calG,\sigma)$  of Corollary \ref{cor:CP_extension_KK_equiv}.

For a fixed $\omega\in\Omega_0$, the Toeplitz algebra can be represented on the 
the space $\Pi_d^\omega \ell^2(\calL^{(\omega)})$, which we can interpret 
as a half-infinite system with boundary. Because we work with general Delone sets, 
$0$ need not be an isolated point in $\Ran(c_d)$. As such, our boundary 
operator $\Pi_d$ is not a projection in general, but if we have a $\Z^d$-labelling, the above 
construction yields a genuine projection.

\begin{remark}[Integer-valued cocycles and the Pimsner--Voiculescu extension]
It is shown in~\cite[Proposition 3.22]{MeslandGpoid} that if the an exact cocycle $c_d$ is integer-valued, then the 
associated $KK$-class $[D_{c_d}] \in KK^1(C^*_r(\calG,\sigma), C^*_r(\Upsilon,\sigma))$ coincides with the 
$KK$-class defined from the circle action 
$$\alpha^{c}:\T \to \mathrm{Aut}(C^*_r(\calG,\sigma)),\quad \alpha^{c}_{t}(f)(\xi):=e^{itc_d(\xi)}f(\xi),$$ 
via the construction in~\cite{CNNR}. 
For crossed products by $\mathbb{Z}$, the Kasparov module of a circle action 
is the same as the Kasparov module constructed from the Toeplitz extension of 
the crossed product (constructed in, for example, \cite{BKR1}) 
$$
   0 \to C^*_r(\Upsilon,\sigma) \otimes\mathbb{K} \to \calT \to C^*_r(\Upsilon,\sigma)\rtimes \Z \to 0
$$
with $C^*_r(\Upsilon,\sigma)\rtimes \Z \cong C^*_r(\calG,\sigma)$. A similar result holds for 
semisaturated circle actions (see~\cite[Section 3]{AriciKaadLandi} or \cite[Section 3.3]{ADL16}).
Hence we recover the `usual' bulk-boundary 
extension considered in~\cite{PSBbook} for special cases of integer-valued cocycles $c_d$. 
This applies in particular if $c_{d}$ is unperforated.
\end{remark}

\begin{remark}[The Connes--Thom class]
Let us now consider the relation between the Kasparov module of 
Proposition \ref{prop:extension_like_class} and its Toeplitz extension with the Connes--Thom isomorphism and the 
Wiener--Hopf extension of~\cite{KR08}, when $d\geq 2$. 

The transversal groupoid $\calG$ is Morita equivalent to the crossed product 
groupoid $\Omega_\calL \rtimes \R^d$ and  for $d\geq 2$ the boundary groupoid $\Upsilon$ is 
equivalent to $\Omega_\calL \rtimes \R^{d-1}$. Given a normalised $2$-cocycle  
$\sigma:\R^d\times \R^d \to \calU(C(\Omega_\calL))$, 
there is an isomorphism 
$C(\Omega_\calL)\rtimes_\sigma \R^d \cong \big(C(\Omega_\calL)\rtimes_\sigma \R^{d-1} \big)\rtimes \R$ 
and a Wiener--Hopf extension
$$
  0 \to (C(\Omega_\calL)\rtimes_\sigma \R^{d-1}  ) \otimes \mathbb{K}[L^2(\R)] \to \mathcal{W} \to 
  C(\Omega_\calL)\rtimes_\sigma \R^{d} \to 0,
$$
see~\cite{KR08}.
In~\cite[Section 6]{BRCont}, it was shown that the Wiener--Hopf extension can be 
represented by the unbounded Kasparov module 
\begin{equation} \label{eq:ConnesThom_kasmod}
  \Big( C_c \big(\R, C(\Omega_\calL)\rtimes_\sigma \R^{d-1} \big) \hat\otimes Cl_{0,1}, F_{C(\Omega_\calL)\rtimes_\sigma \R^{d-1} } \hat\otimes 
  \bigwedge\nolimits^{\! *} \R, \, X \hat\otimes \gamma \Big),
\end{equation}
where $F_{C(\Omega_\calL)\rtimes_\sigma \R^{d-1} }$ is the bimodule obtained from the 
conditional expectation induced from the restriction to the closed subgroupoid 
$\Omega_\calL \rtimes \R^{d-1} \subset \Omega_\calL \rtimes \R^{d}$, 
$$\rho: C_{c}(\Omega_\calL\rtimes_\sigma \R^{d}) \to C_{c}(\Omega_\calL\rtimes_\sigma \R^{d-1}), 
\quad \rho(f)(x_{1},\cdots , x_{d-1}) = f(x_{1},\cdots, x_{d-1},0).$$

We consider the composition of $KK$-classes 
\begin{align*}
   \Big( C(\Omega_\calL)\rtimes_\sigma \R^{d}, \, F^d_{C^*_r(\calG,\sigma)}, \, 0 \Big) \hat\otimes_{C^*_r(\calG,\sigma)} 
   \big[ {}_{d}\lambda_{d-1} \big] \hat\otimes_{C^*_r(\Upsilon, \sigma)}
    \Big( C^*_r(\Upsilon,\sigma), \, (F^{d-1})^*_{ C(\Omega_\calL)\rtimes_\sigma \R^{d-1}}, \, 0 \Big),
\end{align*}
where the left and right Kasparov modules represent the Morita equivalence (resp. dual Morita equivalence)
of the groupoid algebras and crossed products.
The end result of this triple product is a Kasparov module representing a class in 
$KK^1(C(\Omega_{\mathcal{L}}) \rtimes_\sigma \R^d, C(\Omega_\calL)\rtimes_\sigma \R^{d-1} )$.
Its relation to the Kasparov module in Equation \eqref{eq:ConnesThom_kasmod} is as follows. 
By Definition \ref{transversal} 
and Proposition \ref{dtop} the Morita equivalence bimodule $F^d_{C^*_r(\calG,\sigma)}$ is obtained from restriction 
of the crossed product dynamics to the transversal $\Omega_0$. 
The $C^*$-module $E_{C^*_r(\Upsilon,\sigma)}$ from 
${}_d\lambda_{d-1}$ in Proposition \ref{prop:extension_like_class} is defined by a 
restriction $C_c(\calG,\sigma) \to C_c(\Upsilon,\sigma)$. Lastly the 
dual Morita equivalence bimodule $(F^{d-1})^*_{ C(\Omega_\calL)\rtimes_\sigma \R^{d-1}}$ 
is induced by the inclusion of $\Upsilon$ into $\Omega_\calL \rtimes \R^{d-1}$. 
Hence the inner product on the balanced tensor product can be 
considered as coming from a generalised conditional expectation 
$C_c( \Omega_\calL \rtimes \R^d,\sigma) \to C_c( \Omega_\calL \rtimes \R^{d-1},\sigma)$ and there is a natural identification
$$
  F^d  \otimes_{C^*_r(\calG,\sigma)}  E 
  \otimes_{C^*_r(\Upsilon,\sigma)} (F^{d-1})^*_{ C(\Omega_\calL)\rtimes_\sigma \R^{d-1}} \xrightarrow{\sim} 
  F_{C(\Omega_\calL)\rtimes_\sigma \R^{d-1} }.
$$

An argument similar to that in the proof of Theorem \ref{thm:weak_bulkedge} 
below shows that the operator $X$ in  $F_{C(\Omega_\calL)\rtimes_\sigma \R^{d-1}} $ satisfies 
the connection condition with respect to the operator $X\otimes 1$ in $
E \otimes_{C^*_r(\Upsilon,\sigma)} (F^{d-1})^*_{ C(\Omega_\calL)\rtimes_\sigma \R^{d-1}}$.
As these maps are also compatible with the Clifford actions, we recover the 
unbounded representative of the Wiener--Hopf extension from 
Equation \eqref{eq:ConnesThom_kasmod} on page \pageref{eq:ConnesThom_kasmod}. The boundary 
maps in $K$-theory and $K$-homology from the Wiener--Hopf extension, i.e. the product with the 
unbounded Kasparov module from Equation \eqref{eq:ConnesThom_kasmod},  
implement the inverse of the Connes--Thom isomorphism~\cite{Rieffel82}. Hence, these maps are represented by our Toeplitz extension up to groupoid/Morita equivalence. 
\end{remark}

\subsection{Factorisation}
By the same basic argument as the bulk algebra, we can build a 
$KK$-cycle for $C^*_r(\Upsilon,\sigma)$ which is stably 
isomorphic to the edge algebra $C^*_r(\calG \ltimes \calG/ \Upsilon, \sigma)$. 
We denote by $F_{C(\Omega_0)}$ the $(C^*_r(\Upsilon,\sigma),C(\Omega_0))$-$C^*$-bimodule coming from the 
restriction of $C^*_r(\Upsilon,\sigma)$ to the unit space. The notation $F_{C(\Omega_0)}$ distinguishes it from the $C^*$-module $E_{C(\Omega_0)}$ constructed 
from $C^*_r(\calG,\sigma)$.
Specifically,
$$
  {}_{d-1}\lambda_{\Omega_0} = \bigg( C_c(\Upsilon,\sigma)\hat\otimes Cl_{0,d-1}, \, F_{C(\Omega_0)}\hat\otimes 
    \bigwedge\nolimits^{\! *}\R^{d-1}, \, 
    \sum_{j=1}^{d-1} X_j \hat\otimes \gamma^j  \bigg)
$$
is an unbounded Kasparov module and gives rise to a class in 
$KK^{d-1}(C^*_r(\Upsilon,\sigma), C(\Omega_0))$ (real or complex).
We have constructed a class representing our bulk-boundary extension 
and a $KK$-cycle for the edge algebra. The key $K$-theoretic result that 
drives the bulk-boundary correspondence is that these two $KK$-cycles can be 
put together using the unbounded Kasparov product to reconstruct the bulk $KK$-cycle.

\begin{thm} \label{thm:bulk-edge_main}
Under the boundary map coming from the extension of Equation 
\eqref{eq:bulkedge_extension} on page \pageref{eq:bulkedge_extension},
$$
   \partial[{}_{d-1}\lambda_{\Omega_0}] = [{}_{d}\lambda_{d-1}] \hat\otimes_{C^*_r(\Upsilon,\sigma)} [{}_{d-1}\lambda_{\Omega_0}]
     = (-1)^{d-1}[{}_{d}\lambda_{\Omega_0}],
$$
with ${}_{d}\lambda_{\Omega_0}$ the bulk $KK$-cycle from Equation \eqref{eq:bulk_K-cycle} on page \pageref{eq:bulk_K-cycle}
and where $-[x]$ denotes the inverse in the $KK$-group. Furthermore, the equality is 
an unbounded equivalence up to a permutation of the Clifford algebra basis.
\end{thm}

Theorem \ref{thm:bulk-edge_main} is a special case of 
Theorem \ref{thm:weak_bulkedge} with $k=d-1$. Hence we delay the proof until 
Section \ref{sec:weak}.

\subsubsection{A remark on more general boundaries}
Our edge groupoid $\Upsilon$ can be thought of as the result of a cut of the 
Delone sets $\calL^{(\omega)}\in \Omega_0$ along the plane defined by $\Ker(c_d)\cong\R^{d-1}\times\{0\}$. 
This choice of 
cut or boundary is somewhat arbitrary. Let us briefly consider more general boundary choices 
though, as we will show, our $KK$-theoretic factorisation still applies.

Let $b:\R^d\to\R$ be a continuous homomorphism such that as a vector 
space $\mathrm{dim}(\Ran(b)) = 1$.  The plane defined by $\Ker(b)$ defines a new 
$(d-1)$-dimensional plane in $\R^d$ which 
we can cut along to make a new boundary. 
It is easy to check that the corresponding 
map $c_b:\calG\to \R$, $c_b(\omega,x) = b(x)$ is an exact groupoid cocycle.
Hence, we can study this boundary via the  closed subgroupoid $\Upsilon_b = \Ker(c_b)$ and 
equivalent groupoid $\calG \ltimes \calG / \Upsilon_b$, which models the 
$(d-1)$-dimensional dynamics of the transversal relative to $\Ran(b)$. 
Because $c_b$ is exact, we can 
construct an ungraded and unbounded Kasparov module
$\left( C_c(\calG,\sigma), E_{C^*_r(\Upsilon_b,\sigma)}, D_b \right)$ that gives a class  
$[\mathrm{ext}_b] \in KK^1(C^*_r(\calG,\sigma), C^*_r(\Upsilon_b,\sigma))$ and a
bulk-boundary short exact sequence
$$
  0 \to C^*_r(\calG \ltimes \calG / \Upsilon_b, \sigma) \to \calT_b \to C^*_r(\calG,\sigma) \to 0.
$$

As a vector space, $\Ker(b)$ is $(d-1)$-dimensional and so fix an orthonormal basis $\{z_1,\ldots, z_{d-1}\}$. 
These basis vectors give rise to an exact $\R^{d-1}$-valued cocycle on the groupoid 
$\Upsilon_b$, which we use to build an unbounded Kasparov module and 
a class $[{}_{d-1}\lambda^{b}_{\Omega_0}] \in KK^{d-1}(C^*_r(\Upsilon_b), C(\Omega_0))$.

Following the proof of Theorem \ref{thm:weak_bulkedge} with $k=d-1$, the 
product of the class of the Kasparov modules $[\mathrm{ext}_b]$ 
and $[{}_{d-1}\lambda^{b}_{\Omega_0}]$  is represented by the 
unbounded Kasparov module
$$
  \bigg( C_c(\calG,\sigma) \hat\otimes Cl_{0,d}, \, E_{C(\Omega_0)} \hat\otimes \bigwedge\nolimits^{\! *} \R^d, \, 
  \sum_{j=1}^{d-1} Z_j \hat\otimes \gamma^{j+1} + D_b \hat\otimes \gamma^1 \bigg).
$$
At this point, we can take transformation from the basis $\{z_1,\ldots,z_{d-1},z_d\}$ 
to the standard basis of $\R^d$.
This transformation recovers the bulk $K$-cycle ${}_{d}\lambda_{\Omega_0}$ 
up to a Clifford basis re-ordering. 
Fixing the Clifford basis re-ordering, we have that 
$$
   [\mathrm{ext}_b] \hat\otimes_{C^*_r(\Upsilon_b, \sigma)} [{}_{d-1}\lambda^b_{\Omega_0}] 
   = (-1)^{d-1} [{}_{d}\lambda_{\Omega_0}]
$$
and our factorisation result extends.

Let us briefly note that while any crystallographic group $G\subset \R^d$ is a Delone set and 
our choice of boundary is quite general, the 
factorisation and bulk-boundary result in Theorem \ref{thm:bulk-edge_main} is too coarse to detect 
boundary indices derived from the crystalline structure as in~\cite{GT18}.

\subsection{$KK$-cycles with higher codimension} \label{sec:weak}

Let us now generalise the constructions and ideas from the previous section to consider 
subinvariants of arbitrary codimension. Such invariants are linked to so-called weak topological phases 
which are characterised by elements in $K_\ast (C^*_r(\calG,\sigma))$ that are not detected 
by the `top degree form' that comes via a pairing with ${}_{d}\lambda_{\Omega_0}$.

Once again we use a groupoid homomorphism $\check{c}_{k}:\calG\to \R^{d-k}$ via 
$\check{c}_{k}(\omega,x) = (x_{k+1},\ldots,x_{d})$ and define $\Upsilon_{k} = \Ker(\check{c}_{k})$, 
where we characterise 
$$
  \Upsilon_k = \left\{ (\omega,x_1,\ldots,x_k)\in 
    \Omega_0 \times\R^k  \,:\, (x_1,\ldots,x_k,0,\ldots,0) \in \calL^{(\omega)} \right\}.
$$ 
As in the case of $k=d-1$, $\Upsilon_k$ is a closed subgroupoid of $\calG$ and is equivalent to 
$\calG \ltimes \calG/\Upsilon_k$. By Proposition \ref{prop:bulk_kasmod}, we can build a 
Kasparov module
$$
  {}_{d}\lambda_{k} = \bigg( C_c(\calG,\sigma)\hat\otimes Cl_{0,d-k}, \, 
     E_{C^*_r(\Upsilon_k,\sigma)}^{d-k} \hat\otimes \bigwedge\nolimits^{\! *} \R^{d-k}, \, 
       \sum_{j=k+1}^d X_j \hat\otimes \gamma^{j-k} \bigg).
$$  
In the case $k=d-1$, ${}_{d}\lambda_{d-1}$ is the unbounded $KK$-cycle representing 
the bulk-boundary extension considered in Section \ref{sec:edge_gpoid_toeplitz}.
We will be interested in pairings of ${}_{d}\lambda_{k}$ with the $K$-theory of $C^*_r(\calG,\sigma)$. 
Such pairings naturally take values in the $K$-theory of $C^*_r(\Upsilon_k,\sigma)$.

As in the case of $k=d-1$, we can construct an unbounded $KK$-cycle for the 
subgroupoid $\Upsilon_k$, 
\begin{equation} \label{eq:k-dim_K-cycle}
  {}_{k}\lambda_{\Omega_0} = \bigg( C_c(\Upsilon_k,\sigma) \hat\otimes Cl_{0,k}, \, F^k_{C(\Omega_0)} \hat\otimes 
    \bigwedge\nolimits^{\! *} \! \R^k, \, \sum_{j=1}^k X_j \hat\otimes \gamma^j \bigg),  \qquad 
    F^k_{C(\Omega_0)} := \ol{C_c(\Upsilon_k,\sigma)}_{C(\Omega_0)},
\end{equation}
which represents the class 
$[{}_{k}\lambda_{\Omega_0}] \in KK(C_r^*(\Upsilon_k,\sigma) \hat\otimes Cl_{0,k}, C(\Omega_0))$. 

We now present our main factorisation result, which allows us to decompose the bulk Kasparov module 
${}_{d}\lambda_{\Omega_0}$ as the product of ${}_{d}\lambda_{k}$ with ${}_{k}\lambda_{\Omega_0}$ 
(up to a sign related to the orientation 
of Clifford algebras).

\begin{thm} \label{thm:weak_bulkedge}
Taking the Kasparov product,
$$
  [{}_{d}\lambda_{k}] \hat\otimes_{C^*_r(\Upsilon_k,\sigma)} [{}_{k}\lambda_{\Omega_0}] = (-1)^{k(d-k)} [{}_{d}\lambda_{\Omega_0}].
$$
Furthermore, our equivalence is at the unbounded level up to a permutation of the 
 Clifford basis.
\end{thm}
\begin{proof}
Much of this proof is book-keeping and is very similar to the proof in~\cite[Theorem 3.4]{BKR1}.
To take the product of the $C^*_r(\calG,\sigma)\hat\otimes Cl_{0,d-k}$--$C^*_r(\Upsilon_k,\sigma)$ Kasparov 
module with a $C^*_r(\Upsilon_k,\sigma)\hat\otimes Cl_{0,k}$--$C(\Omega_0)$ Kasparov module, we first 
take the external product of ${}_{d}\lambda_{k}$ with a $KK$-cycle representing the 
identity in $KK(Cl_{0,k},Cl_{0,k})$. This identity class can be represented by 
$\left( Cl_{0,k}, {Cl_{0,k}}_{Cl_{0,k}}, 0 \right)$ with right and left actions by multiplication. 
We then take the product of a 
$(C^*_r(\calG,\sigma)\hat\otimes Cl_{0,d},C^*_r(\Upsilon_k,\sigma) \hat\otimes Cl_{0,k})$ Kasparov 
module with a $(C^*_r(\Upsilon_k,\sigma)\hat\otimes Cl_{0,k},C(\Omega_0))$ Kasparov module.
First the balanced tensor product gives the $C^*$-module
\begin{align*}
  &\Big( E^{d-k} \hat\otimes \bigwedge\nolimits^{\!*}\R^{d-k} \hat\otimes Cl_{0,k} \Big)  \hat\otimes_{C^*_r(\Upsilon_k,\sigma)\hat\otimes Cl_{0,k}} 
    \Big( F^k \hat\otimes \bigwedge\nolimits^{\!*}\R^{k} \Big)_{C(\Omega_0)} \\
  &\hspace{2cm} \cong  \big( E^{d-k} \otimes_{C^*_r(\Upsilon_k,\sigma)} F^k_{C(\Omega_0)} \big) \hat\otimes  \bigwedge\nolimits^{\!*}\R^{d-k} \hat\otimes 
     \Big( Cl_{0,k} \hat\otimes_{Cl_{0,k}} \bigwedge\nolimits^{\!*}\R^{k} \Big) \\
    &\hspace{2cm} \cong  \big(E^{d-k} \otimes_{C^*_r(\Upsilon_k,\sigma)} F^k \big)_{C(\Omega_0)}  \hat\otimes  \bigwedge\nolimits^{\!*}\R^{d-k} 
         \hat\otimes \bigwedge\nolimits^{\!*}\R^{k}
\end{align*}
as $Cl_{0,d-1}$ acts on $\bigwedge^* \R^{d-1}$ non-degenerately. 
Next we define a unitary isomorphism 
$$E^{d-k}  \otimes_{C^*_r(\Upsilon_k,\sigma)} F^k_{C(\Omega_0)} \to E_{C(\Omega_0)},$$ by 
first considering defining a map on dense submodules,
\begin{align*}
 v:    C_c(\calG,\sigma) \otimes_{C_c(\Upsilon_k,\sigma)} C_c(\Upsilon_k,\sigma)_{C(\Omega_0)} 
      \ni f \otimes h &\mapsto f \cdot h \in C_c(\calG,\sigma)_{C(\Omega_0)}.
\end{align*}
This map preserves the inner-product structures, is thus  uniformly bounded and, hence,
extends to an isomorphism of $C^*$-modules. Furthermore the map commutes with the 
representation of $C^*_r(\calG,\sigma)$ as elements in $\End^*(E^{d-k}_{C^*_r(\Upsilon_k,\sigma)})$ commute 
with the right-action of $C^*(\Upsilon_k,\sigma)$. 
Similarly, $\{X_l\}_{l=k+1}^d$ also commutes with this map as $X_l$ is right $C^*_r(\Upsilon_k,\sigma)$-linear on 
$E^{d-k}_{C^*_r(\Upsilon_k,\sigma)}$. The operators $\{X_{j}\}_{j=1}^{k}$ satisfy the connection condition under the unitary isomorphism $v$. Let $f\in C_{c}(\mathcal{G},\sigma)$ and consider the map
\[v\circ |f\rangle: C_{c}(\Upsilon_{k},\sigma)\to C_{c}(\mathcal{G},\sigma),\quad h\mapsto f\cdot h.\]
Then we need to check that $X_{j}\circ v\circ | f\rangle-v\circ  | f\rangle\circ X_{j}$ defines a bounded operator $F^k_{C(\Omega_0)} \to E_{C(\Omega_0)}$. 
This follows since each $X_{j}$ acts as a derivation of $C_{c}(\mathcal{G},\sigma)$:
\begin{align*}
  \big( (X_j f)\cdot h + f\cdot (X_j h) \big)(\omega,x) &= \sum_{(y,0_{d-k})\in\calL^{(\omega)}-x}\!\! \Big( (x_j+y_j)f(\omega,x+y) h(T_{-x-y}\omega,-y) \\
             &\hspace{0.5cm}+ f(\omega,x+y)(-y_j)h(T_{-x-y}\omega,-y) \Big) \sigma((\omega,x+y),(T_{-x-y}\omega,-y))   \\
        &\hspace{-2cm}= x_j \left(\sum_{(y,0_{d-k})\in\calL^{(\omega)}-x} f(\omega,x+y) h(T_{-x-y}\omega,-y) \sigma((\omega,x+y),(T_{-x-y}\omega,-y))\right)\\
        &\hspace{-2cm}= X_j ( f\cdot h)(\omega,x).
\end{align*}
It follows that $X_{j}\circ v\circ | f\rangle-v\circ  | f\rangle\circ X_{j}=v\circ|X_{j}f\rangle,$ which is a bounded adjointable operator.
The left and right Clifford actions on $\bigwedge^*\R^{d-k} \hat\otimes \bigwedge^*\R^{k}$ are given by
\begin{align*}
   &\rho^l \hat\otimes 1(\omega_1\hat\otimes\omega_2) = (e_l \wedge \omega_1 - \iota(e_l)\omega_1)\hat\otimes \omega_2,  
     &&1\hat\otimes \rho^j(\omega_1\hat\otimes\omega_2) = (-1)^{|\omega_1|} \omega_1\hat\otimes(e_j\wedge \omega_2 - \iota(e_j)\omega_2),  \\
   &\gamma^l \hat\otimes 1(\omega_1 \hat\otimes \omega_2) =  (e_l\wedge \omega_1 + \iota(e_l)\omega_1)\hat\otimes \omega_2,     
     &&1\hat\otimes \gamma^j(\omega_1\hat\otimes\omega_2) = (-1)^{|\omega_1|} \omega_1\hat\otimes(e_j\wedge \omega_2 + \iota(e_j)\omega_2),
\end{align*}
with $|\omega|$ the degree of the form and $\{e_l\}_{l=1}^{d-k}$ and $\{e_j\}_{j=1}^k$ 
the standard bases of $\R^{d-k}$ and $\R^{k}$ 
respectively.

We relate 
$\bigwedge^*\R^{d-k} \hat\otimes \bigwedge^*\R^{k} \cong \bigwedge^*\R^d$, 
which sends the $Cl_{0,d-k}\hat\otimes Cl_{0,k} \to Cl_{0,d}$ by the map on generators,
\begin{align*}
    &\rho^l \hat\otimes 1 \mapsto \rho^l,  &&1\hat\otimes \rho^j \mapsto \rho^{d-k+j}
\end{align*}
with $l\in\{1,\ldots,d-k\}$ and $j\in\{1,\ldots, k\}$ (see~\cite[{\S}2.16]{Kasparov80}). There is 
an analogous map for the right-action of $Cl_{d-k,0}\hat\otimes Cl_{k,0}$.

This leads us to conclude that the 
unbounded Kasparov module
\begin{equation}
\label{eq:prod_candidate}
  \bigg( C_c(\calG,\sigma) \hat\otimes Cl_{0,d}, \, E_{C(\Omega_0)} \hat\otimes \bigwedge\nolimits^{\! *} \R^d, \, 
  \sum_{l=k+1}^d X_l \hat\otimes \gamma^{l-k} 
    + \sum_{j=1}^{k} X_j \hat\otimes \gamma^{d-k+j} \bigg),
\end{equation}
represents the Kasparov product $[{}_{d}\lambda_{k}] \hat\otimes_{C^*_r(\Upsilon_k,\sigma)} [{}_{k}\lambda_{\Omega_0}]$, 
because it satisfies the hypotheses of ~\cite[Theorem 13]{Kucerovsky97}. 
Its operator satisfies the connection condition as shown above, the domain of 
the operator is included 
 in the domain of  $\sum_{l=k+1}^d X_l \hat\otimes \gamma^{l-k}$ 
 and since the $X_{j}\hat\otimes \gamma^{j}$ mutually anticommute, the positivity condition is satisfied as well.

The Kasparov module \eqref{eq:prod_candidate} recovers the bulk module ${}_{d}\lambda_{\Omega_0}$ 
up to a re-ordering of the Clifford basis, as we now show. 
We consider the map $\eta_{d-k}(\gamma^j) =\gamma^{j-(d-k)}$ on $Cl_{d,0}$ where we identify 
$\gamma^{l}=\gamma^{d-l}$ if $l\leq 0$. We define the same map on $\rho^j$ and $Cl_{0,d}$. 
The map $\eta$ is an automorphism of Clifford algebras but may 
reverse the canonical orientation, namely
$$
  \eta_{d-k}( \omega_{Cl_{d,0}}) = \eta_{d-k}\big( \gamma^1\cdots\gamma^d\big) = 
    \gamma^{k+1} \cdots \gamma^d \gamma^1 \cdots \gamma^{k} = (-1)^{k(d-k)} \gamma^1\gamma^2\cdots \gamma^d 
    = (-1)^{k(d-k)}\omega_{Cl_{d,0}},
$$
with the same result for the orientation of $Cl_{0,d}$. We can apply the map $\eta_{d-k}$ 
to obtain the bulk-cycle ${}_{d}\lambda_{\Omega_0}$ but at the 
expense that at the level of $KK$-classes 
$[x] \mapsto (-1)^{k(d-k)}[x]$~\cite[{\S}5, Theorem 3]{Kasparov80}. This finishes the proof.
\end{proof}

\subsubsection{Another factorisation}

Let us also show another way our Kasparov modules can be factorised using 
a different short exact sequence. 
Starting with $\Upsilon_k$, $\Upsilon_{k-1}$ is a closed subgroupoid and we can build the 
$C^*$-bimodule $F_{C^*_r(\Upsilon_{k-1},\sigma)}$ via the restriction 
$C_c(\Upsilon_{k},\sigma)\to C_c(\Upsilon_{k-1},\sigma)$. Applying 
Proposition \ref{prop:bulk_kasmod}, we obtain the unbounded Kasparov module 
$$
  {}_{k}\lambda_{k-1} = \bigg( C_c(\Upsilon_k,\sigma)\hat\otimes Cl_{0,1}, \, 
    F_{C^*_r(\Upsilon_{k-1},\sigma)} \hat\otimes \bigwedge\nolimits^{\! *} \R, \, X_k\hat\otimes \gamma \bigg)
$$
and for $(\Pi_k f)(\omega,y) = \chi_{[-\delta,\infty)}(y_k) f(\omega,y)$, we have an extension
$$
  0 \to C^*_r(\Upsilon_k \ltimes \Upsilon_k/\Upsilon_{k-1}, \sigma) \to 
    C^*\big(\Pi_k C^*_r(\Upsilon_k,\sigma) \Pi_k, C^*_r(\Upsilon_k \ltimes \Upsilon_k/\Upsilon_{k-1}, \sigma) \big) 
    \to C^*_r(\Upsilon_k,\sigma) \to 0.
$$

\begin{thm} \label{thm:weak_factorisation2}
Taking the Kasparov product,
$$
  [{}_{d}\lambda_{k}] \hat\otimes_{C^*_r(\Upsilon_k,\sigma)} [{}_{k}\lambda_{k-1}] = (-1)^{d-k}[{}_{d}\lambda_{k-1}].
$$
Furthermore, our equivalence is at the unbounded level up to a permutation of the 
 Clifford basis.
\end{thm}

\begin{proof}
The proof follows a very similar argument 
as the proof of Theorem \ref{thm:weak_bulkedge}. First we define a map 
\begin{align*}
  v:&\, C_c(\calG,\sigma) \otimes_{C_c(\Upsilon_k,\sigma)} C_c(\Upsilon_k,\sigma)_{C_{c}(\Upsilon_{k-1},\sigma)} \to E^{d-(k-1)}_{C^*_r(\Upsilon_{k-1} ,\sigma)},  &&f \otimes h_k \mapsto f\cdot h_k
\end{align*}
where we consider $f\cdot h_k$ as an element in $E^{d-(k-1)}_{C^*_r(\Upsilon_{k-1} ,\sigma)}$
One can check analogously to the proof of Theorem \ref{thm:bulk-edge_main} that 
this map extends to a unitary isomorphism of $C^*$-modules
$$v:E^{d-k}\otimes_{C^*_r(\Upsilon_k ,\sigma)} F_{C^*_r(\Upsilon_{k-1} ,\sigma)} \to 
   E^{d-(k-1)}_{C^*_r(\Upsilon_{k-1} ,\sigma)}.$$
Similarly, we check that 
\begin{align*}
\big(X_{k}\circ v\circ  |f\rangle-v\circ  |f\rangle\circ X_{k}\big)h= (X_{k}f)\cdot h=v\circ |X_{j}f\rangle (h)
\end{align*}
defines a bounded operator
and for $j\in\{k+1,\ldots,d\}$,
$(X_j\otimes 1) \mapsto X_j$ as $X_j$ is right $C^*_r(\Upsilon_k)$-linear. 

Applying the isomorphism and grouping together the Clifford actions, 
we obtain the unbounded Kasparov module
$$
  \bigg( C_c(\calG,\sigma)\hat\otimes Cl_{0,d-(k-1)}, \, E_{C^*_r(\Upsilon_k)}^{d-(k-1)} 
    \hat\otimes \bigwedge\nolimits^{\!*}\R^{d-(k-1)}, \, 
    X_k\hat\otimes \gamma^{d-k+1} + \sum_{j=k+1}^d X_j \hat\otimes \gamma^{j-k} \bigg),
$$
which as before satisfies ~\cite[Theorem 13]{Kucerovsky97} and thus represents the Kasparov product. 
To relate this $KK$-cycle to 
${}_{d}\lambda_{k-1}$, we correct the Clifford labelling by the map 
$\gamma^j \mapsto \gamma^{j+1}$ for $1\leq j \leq d-k$ and 
$\gamma^{d-k+1}\mapsto \gamma^1$. Such a map 
will change the orientation of $Cl_{0,d-(k-1)}$ and $Cl_{d-(k-1),0}$ 
by a factor of $(-1)^{d-k}$. The result follows.
\end{proof}

\section{Spectral triple constructions} \label{sec:bulk_spectral_triples}
We now present two constructions of (semifinite) spectral triples obtained from localising the 
bulk $KK$-cycle for $(C^{*}_{r}(\calG,\sigma),C(\Omega_0))$ over a state of 
$C(\Omega_0)$. Their index theoretic properties are discussed in Section \ref{sec:Applications}.

\subsection{The evaluation spectral triple}
We can directly construct a spectral 
triple on $\ell^2(\calL^{(\omega)})$ by considering the internal 
product of the Kasparov module ${}_{d}\lambda_{\Omega_0}$ with the 
trivially graded Kasparov module 
$\mathrm{ev}_\omega =(C(\Omega_0), {}_{\mathrm{ev}_\omega}\R_\R, 0 )$ 
coming from the evaluation map on $C(\Omega_0)\to \R$ (or $\C$). 
This spectral triple was considered in~\cite{BP17} for complex algebras.
The Kasparov module $\mathrm{ev}_\omega$ gives a class in 
$KKO(C(\Omega_0),\R)$ or $KK(C(\Omega_0),\C)$ if the algebra 
and space is complex.
If we take the internal product, then 
\begin{align*}
   &\bigg( C_{c}(\calG,\sigma) \hat\otimes Cl_{0,d}, \, E_{C(\Omega_0)} \hat\otimes 
    \bigwedge\nolimits^{\! *} \R^d, \, X= \sum_{j=1}^d X_j \hat\otimes \gamma^j \bigg) \hat\otimes_{C(\Omega_0)} 
   \left(C(\Omega_0), {}_{\mathrm{ev}_\omega}\R_\R, 0 \right) \\
  &\hspace{2cm} \cong \bigg( C_c(\calG,\sigma) \hat\otimes Cl_{0,d}, \, (E_{C(\Omega_0)} \otimes_{\mathrm{ev}_\omega} \R) \hat\otimes 
    \bigwedge\nolimits^{\! *} \R^d, \,  \sum_{j=1}^d X_j\otimes 1 \hat\otimes \gamma^j \bigg).
\end{align*}
There is an isometric isomorphism
$E_{C(\Omega_0)} \otimes_{\mathrm{ev}_\omega} \R \to \ell^{2}(s^{-1}(\omega))$ (see for instance \cite[p.50]{KS02}). Since 
$$s^{-1}(\omega)=\left\{(T_{-x}\omega, -x):x\in\mathcal{L}^{(\omega)}\right\}\simeq \mathcal{L}^{(\omega)},\quad (T_{-x}\omega, -x)\mapsto x,$$ 
the Hilbert space $\ell^{2}(s^{-1}(\omega))$ can be canonically identified with $\ell^2(\calL^{(\omega)})$. This gives a map
\[\rho_{\omega}:E_{C(\Omega_0)} \otimes_{\mathrm{ev}_\omega} \R \to \ell^{2}(\mathcal{L}^{(\omega)}),\quad \rho_{\omega} (f\otimes t)(x) = tf(T_{-x}\omega, -x), \]
and the action of $C_{c}(\mathcal{G},\sigma)$ is then computed to be
\begin{align*}
  \rho_\omega(\pi(f_1)f_2)(x) &= (f_1\ast f_2)(T_{-x}\omega,-x)  \\
   &= \sum_{y\in \calL^{(\omega)} - x} \sigma((T_{-x}\omega,y),(T_{-x-y}\omega,-x-y))
        f_1(T_{-x}\omega,y) f_2(T_{-x-y}\omega,-x-y) \\
    &=  \sum_{u\in \calL^{(\omega)}} \sigma((T_{-x}\omega,u-x),(T_{-u}\omega,-u)) 
       f_1(T_{-x}\omega,u-x) f_2(T_{-u}\omega,-u) \\
    &= \sum_{u\in \calL^{(\omega)}} \sigma((T_{-x}\omega,u-x),(T_{-u}\omega,-u))
       f_1(T_{-x}\omega,u-x) (\rho_{\omega} f_2)(u).
\end{align*}
Hence for $f\in C_c(\calG,\sigma)$ the representation of $C^*_r(\calG,\sigma)$ on $\ell^{2}(\mathcal{L}^{(\omega)})$  is given by
$$
  \big(\pi_\omega(f)\psi)(x) = \sum_{y\in\calL^{(\omega)}} \sigma((T_{-x}\omega,y-x),(T_{-y}\omega,-y)) f(T_{-x}\omega,y-x)\psi(y).
$$

\begin{prop}[\cite{BP17}, Proposition 5.1] \label{prop:evaluation_spec_trip}
The triple
$$
  {}_{d}\lambda_\omega = \bigg( C_{c}(\calG,\sigma) \hat\otimes Cl_{0,d}, \, {}_{\pi_\omega}\ell^2(\calL^{(\omega)}) 
    \hat\otimes\bigwedge\nolimits^{\!*}\R^d, \, \sum_{j=1}^d X_j\hat\otimes\gamma^j \bigg)
$$
is a $QC^\infty$ and $d$-summable spectral triple. 
If $\omega,\omega'\in\Omega_0$ are such that $\omega'= T_{-a}\omega$. 
then the spectral triples ${}_{d}\lambda_\omega$ and ${}_{d}\lambda_{\omega'}$ 
define the same class in the $K$-homology of $C_r^*(\calG,\sigma)$.
\end{prop}

\subsection{Invariant measures and the semifinite spectral triple} \label{sec:semifinite_construction}

Measure theoretic properties of the continuous hull $\Omega_\calL$ have been extensively 
studied. We note a useful result below.

\begin{prop}[\cite{BBG06, SavThesis}]
There is a one-to-one 
correspondence between measures on $\Omega_\calL$ invariant 
under the $\R^d$-action and measures on the unit space $\Omega_0$ 
invariant under the groupoid action. Furthermore, if $\calL$ is repetitive, 
aperiodic and
has finite local complexity, then there is a one-to-one 
correspondence between the invariant measures on $\Omega_\calL$ and 
a canonical positive cone in $H_d(\Omega_\calL, \R)$.
\end{prop}

Hence under additional hypotheses, invariant measure theory on the transversal $\Omega_0$ can be reduced to a homological 
condition on the continuous hull $\Omega_\calL$. 
We will now assume that the unit space $\Omega_0$ has a probability measure $\bP$ that 
is invariant under the groupoid action with 
$\mathrm{supp}(\bP) = \Omega_0$. Using~\cite[Theorem 1.1]{LN04}, given the trace 
$$\tau_{\bP}:C(\Omega_0)\to \mathbb{C},\quad f\mapsto \int f(\omega)\,\mathrm{d}\bP(\omega)$$ 
on $C(\Omega_0)$ we can define 
the dual trace on finite-rank endomorphisms $\mathrm{Fin}(E_{C(\Omega_0)}) \subset \mathbb{K}(E_{C(\Omega_0)})$ by the formula
$$
  \Tr_\tau(\Theta_{e_1,e_2}) = \tau_{\bP}((e_2\mid e_1)_{C(\Omega_0)}),
$$
which then extends to a faithful, semifinite and norm lower semicontinuous  trace on the von Neumann algebra 
$\calN = \mathrm{Fin}(E_{C(\Omega_0)})'' \subset \calB (\calH_\tau)$, with $\calH_\tau$  
the completion of $C_c(\calG,\sigma)$ under 
the inner-product 
$$
    \langle f_1, f_2 \rangle = \int_{\Omega_0} (f_1\mid f_2)_{C(\Omega_0)}(\omega)\, \mathrm{d}\bP(\omega) 
      = \int_{\Omega_0}\! (f_1^\ast \ast f_2)(\omega,0)\,\mathrm{d}\bP(\omega).
$$

We note that for $f\in C_c(\calG,\sigma)$, the dual trace $\Tr_\tau$ can also be written 
by the simple formula
\begin{equation} \label{eq:dual_trace_formula}
  \Tr_\tau(f) = \int_{\Omega_0} \! f(\omega,0)\,\mathrm{d}\bP(\omega).
\end{equation}

The semifinite trace we use is quite abstract but can be related to the so-called 
trace per unit volume  if we also assume ergodicity. 
\begin{prop}[\cite{BP17}, Proposition 4.23] \label{prop:ergodic_trace_is_vol_trace}
If the measure on $\Omega_{{\calL}}$ is ergodic under the translation 
action, then for almost all $\omega\in\Omega_0$ and any $f\in C_c(\calG,\sigma)$,
$$
   \Tr_\tau(f) = \Tr_\mathrm{Vol}( \pi_\omega(f)) 
     := \lim_{\Lambda \nearrow \calL^{(\omega)}}\frac{1}{|\Lambda|} \Tr_{\ell^2(\calL^{(\omega)})}\big( P_\Lambda \pi_\omega(f) \big), 
     \qquad P_\Lambda: \ell^2(\calL^{(\omega)}) \to \ell^2(\Lambda),
$$
where the limit $\Lambda \nearrow \calL^{(\omega)}$ is an increasing sequence of 
finite sets approximating $\calL^{(\omega)}$. 
\end{prop}

The following result 
does not require an ergodicity assumption.

\begin{prop} \label{prop:semifinite_smooth_spec_trip}
The triple
$$
 {}_{d}\lambda_{\tau} = \bigg( C_c(\calG,\sigma) \hat\otimes Cl_{0,d}, \, \calH_\tau \hat\otimes \bigwedge\nolimits^{\!*} \R^d, \,
    \sum_{j=1}^d X_j \otimes \gamma^j \bigg)
$$
is a $QC^\infty$ and $d$-summable semifinite spectral triple relative to $(\calN,\Tr_\tau)$. 
Furthermore, for $f\in C_c(\calG,\sigma)$, the identity
$$
  \res_{z=d} \Tr_\tau(\pi(f)(1+|X|^2)^{-s/2}) = \mathrm{Vol}_{d-1}(S^{d-1}) \,\Tr_{\tau}(f),
$$
holds true.
\end{prop}
\begin{proof}
The representation of $C^*_r(\calG,\sigma)$ on $E_{C(\Omega_0)}$ gives 
a representation $\pi:C^*_r(\calG,\sigma)\to \calB(\calH_\tau)$ as 
$\calH_\tau \cong E \otimes_{C(\Omega_0)} L^2(\Omega_0,\bP)$. 
This representation retains the property that $[X_j, \pi(f)] = \pi(\partial_j f)$ and, as such, 
$[|X|^k,\pi(f)]$ is well-defined and bounded for all $k\in\N$. To 
consider the summability, we first note that 
$(1+X^2)^{-s/2} = (1+|X|^2)^{-s/2}\hat\otimes 1_{\bigwedge^* \R^d}$ and so 
it suffices to prove the summability of $(1+|X|^2)^{-s/2}$. 
We then observe that the space of trace class elements under the dual trace 
$\calL^1(\calN,\Tr_\tau)$ contains the trace class 
operators on the space 
$\int_{\Omega_0}^\oplus \ell^2(\calL^{(\omega)})\,\mathrm{d}\bP(\omega)$
and, on this subalgebra, the dual trace acts as the usual trace on the 
direct integral. 
With this in mind, we first compute 
\begin{align*}
  &\big(\pi(f)(1+|X|^2)^{-s/4}\psi\big)(T_{-x}\omega,-x) = \\
  &\hspace{3cm}
  \sum_{y\in \calL^{(\omega)}} \sigma((T_{-x}\omega,y-x),(T_{-y}\omega,-y)) 
    f(T_{-x}\omega,y-x)(1+|y|^2)^{-s/4} \psi(T_{-y}\omega,-y).
\end{align*}
Hence the `integral kernel' of this operator is 
$$
  k_f(\omega;x,y) = \sigma((T_{-x}\omega,y-x),(T_{-y}\omega,-y)) 
    f(T_{-x}\omega,y-x)(1+|y|^2)^{-s/4}.
$$
Similarly, one can compute that the integral kernel of 
$(1+|X|^2)^{-s/4}\pi(f^*)$ is
$$
  k_{f^*}(\omega;x,y) = \sigma((T_{-x}\omega,y-x),(T_{-y}\omega,-y))
     f^*(T_{-x}\omega,y-x)(1+|x|^2)^{-s/4}.
$$
Then we can estimate the $\Tr_\tau$-Hilbert--Schmidt norm
\begin{align*}
  \big\|\pi(f)(1+|X|^2)^{-s/4}\big\|_{2}^2 &= \int_{\Omega_0} \sum_{x,y\in \calL^{(\omega)}} 
    k_{f^*}(\omega;x,y)k_f(\omega;y,x)\,\mathrm{d}\bP(\omega) \\
  &\hspace{-2cm}= \int_{\Omega_0} \sum_{x,y\in \calL^{(\omega)}} 
  \sigma((T_{-x}\omega,y-x),(T_{-y}\omega,-y))\sigma(T_{-y}\omega,x-y),(T_{-x}\omega,-x)) \\
    &\hspace{3cm}\times f^*(T_{-x}\omega,y-x)f(T_{-y}\omega,x-y)(1+|x|^2)^{-s/2} 
  \,\mathrm{d}\bP(\omega) \\
  &\hspace{-2cm}= \int_{\Omega_0} \sum_{x,y\in \calL^{(\omega)}} 
  \sigma((T_{-x}\omega,y-x),(T_{-y}\omega,x-y))\sigma((T_{-x}\omega,0),(T_{-x}\omega,-x)) \\
    &\hspace{5cm}\times |f(T_{-y}\omega,x-y)|^2(1+|x|^2)^{-s/2} \,\mathrm{d}\bP(\omega) \\
  &\hspace{-2cm}= \int_{\Omega_0} \sum_{x,y\in \calL^{(\omega)}} 
  |f(T_{-y}\omega,x-y)|^2(1+|x|^2)^{-s/2} \,\mathrm{d}\bP(\omega) \\
  &\hspace{-2cm}= \int_{\Omega_0} \sum_{x\in\calL^{(\omega)}} \sum_{u\in \calL^{(\omega)-x}} 
    |f(T_{u-x}\omega,u)|^2 (1+|x|^2)^{-s/2} \,\mathrm{d}\bP(\omega) \\
  &\hspace{-2cm}\leq C \int_{\Omega_0} \sum_{x\in\calL^{(\omega)}} 
     (1+|x|^2)^{-s/2} \,\mathrm{d}\bP(\omega) = C \int_{\Omega_0} \!C_s(\omega) \,\mathrm{d}\bP(\omega),
\end{align*}
where in the third line we have used the  cocycle identity, where we then note that 
$$
  \sigma((T_{-x}\omega,y-x),(T_{-y}\omega,x-y))\sigma((T_{-x}\omega,0),(T_{-x}\omega,-x)) 
  = \sigma(\xi,\xi^{-1})\sigma(r(\eta),\eta) = 1.
$$
Because Delone subsets of $\R^d$ display the same summability asymptotics as $\Z^d$, 
we see that $C_s(\omega)$ is bounded for all $\omega\in\Omega_0$ and $s>d$. 
Hence we have that $\pi(f)(1+|X|^2)^{-s/4}$ is $\Tr_\tau$-Hilbert-Schmidt. 
Therefore $(1+|X|^2)^{-s/4}\pi(f^*f)(1+|X|^2)^{-s/4}$ is $\Tr_\tau$-trace class 
for all $f\in C_c(\calG,\sigma)$ and $s>d$. In particular, 
$(1+|X|^2)^{-s/2}$ is $\Tr_\tau$-trace class for $s>d$. 

Let us now consider the residue trace of  $\pi(f)(1+|X|^2)^{-z/2}$ for $\Re(z)<d$.
By the properties of the dual 
trace, we can compute the trace by summing along the diagonal of this 
integral kernel.
\begin{align*}
  \Tr_\tau \big(\pi(f)(1+|X|^2)^{-z/2} \big) &= \int_{\Omega_0} \sum_{x\in\calL^{(\omega)}} k(x,x) \,\mathrm{d}\bP(\omega) \\
    &= \int_{\Omega_0} \sum_{x\in\calL^{(\omega)}} \sigma((T_{-x}\omega,0),(T_{-x},-x)) 
         f(T_{-x}\omega,0) (1+|x|^2)^{-z/2} \,\mathrm{d}\bP(\omega) \\
    &= \int_{\Omega_0} f(\omega,0) \sum_{x\in\calL^{(\omega)}} (1+|x|^2)^{-z/2} \,\mathrm{d}\bP(\omega) \\
    &= C(z) \int_{\Omega_0} f(\omega,0)  \,\mathrm{d}\bP(\omega),
\end{align*}
where we have used that $\sigma(r(\xi),\xi)=1$ for all $\xi \in\calG$ and the 
invariance of the measure $\bP$ under the groupoid action. For  
 \emph{any} Delone set $\omega\in\Omega_0$, we use an integral approximation 
 to compute that
$$
 C(z) = \sum_{x\in\calL^{(\omega)}} (1+|x|^2)^{-z/2} =  \mathrm{Vol}_{d-1}(S^{d-1}) \, 
      \frac{\Gamma\!\left(\frac{d}{2}\right) \Gamma\!\left(\frac{z-d}{2}\right)}{2\Gamma\!\left(\frac{d}{2}\right)} + h(z)
$$
with $h$ a function holomorphic in a neighbourhood of $\Re(z)=d$. 
The function $C(z)$ has a meromorphic extension to the complex plane with a simple pole at $z=d$ 
with $\res_{z=d} C(z) = \mathrm{Vol}_{d-1}(S^{d-1})$. The result follows.
\end{proof}

Also of use to us for complex algebras  is the semifinite spectral triple  
from the spin$^c$ $KK$-cycle in Proposition \ref{prop:spin_KasMod}. That is, 
\begin{equation} \label{eq:spin_semifinite_spec_trip}
 {}_{d}\lambda_{\tau}^{S_\C} = \bigg( C_c(\calG, \sigma),\, \calH_\tau \hat\otimes 
       \C^{2^{\lfloor \frac{d}{2} \rfloor}}, \, \sum_{j=1}^d X_j \hat\otimes \gamma^j \bigg)
\end{equation}
is a $QC^\infty$ and $d$-summable semifinite spectral triple that
is even or odd depending on the parity of $d$. We recall that, as the spin and 
oriented Kasparov modules are equivalent at the level of $KK$-theory 
(up to a renormalisation), we can equivalently consider pairings with the spin semifinite 
spectral triple.

\section{Unbounded Fredholm modules for lattices with finite local complexity} \label{sec:PB_product}

We will now assume that our lattice $\calL$ has finite local complexity. 
Recall from Proposition \ref{prop:tranversal_properties} 
that this implies that the transversal $\Omega_0$ is totally 
disconnected. In particular, we have 
an explicit description of the basis of the topology of $\Omega_0$ by closed and open sets. 
Namely, 
for some $n\in\N$ and $P \subset B(0;n)$ discrete, the sets 
$U_{P,n}=\{ \omega\in\Omega_0\,:\, B(0;n) \cap \calL^{(\omega)}=P \}$  
give a basis of the topology of $\Omega_0$, see~\cite{Kellendonk95}. 
We will use these sets to characterise $\Omega_0$ as the boundary 
of a rooted tree. This then allows us to use the Pearson--Bellissard 
construction to obtain a spectral triple and corresponding 
class in $KK_0(C(\Omega_0),\C)$. We compose this spectral triple with our 
bulk $KK$-cycle via the unbounded Kasparov product. As in~\cite[Section 6]{GM15}, the resulting operator exhibits mildly unbounded commutators with the algebra  $C_{c}(\calG)$ and its bounded transform is a Fredholm module.

Spectral triple constructions for $C^*_r(\calG)$ building from the Pearson--Bellissard framework 
have already appeared in the tiling literature~\cite{JKS15,MW17}. 
While the setting of each construction is quite different, 
it would be interesting to better understand the relationship between these spectral triples and our unbounded 
Fredholm module.

\subsection{The Pearson--Bellissard spectral triple}
In the case that $\mathcal{L}$ has finite local complexity, 
$\Omega_0$ is totally disconnected and can be conveniently described 
as the boundary of a rooted tree $\mathcal{T}=\mathcal{T}_{\mathcal{L}}$ 
using the local patterns $p\in P_{\mu}$. The set of vertices of 
$\mathcal{T}_{\mathcal{L}}$ is denoted $\mathcal{V}_{\mathcal{L}}$ and the set of 
edges by $\mathcal{E}_{\mathcal{L}}$. They are given explicitly by
\[\mathcal{V}_{\mathcal{L}}:=\{p\in P_{nR}: n\in \mathbb{N}\},\quad \mathcal{E}_{\mathcal{L}}:=\{(p,q)\in P_{nR}\times P_{(n+1)R}: p\subset q\}.\]
Thus, the vertices are the patterns seen at all levels $nR$ and there is an 
edge from $p\in P_{nR}$ to $q\in P_{(n+1)R}$ if and only if $p\subset q$. 
The root of this tree is the unique element $\{0\}\in P_{0}$. The vertex set 
$\mathcal{V}$ is naturally $\mathbb{N}$-graded by
\[\mathcal{V}_{n}:=\{p\in\mathcal{V}: p\in P_{nR}\},\]
and we denote the degree of $v\in\mathcal{V}$ by $|v|$. The boundary $\partial \mathcal{T}$ 
is defined to be the set of infinite paths $\alpha=p_{0}\cdots p_{n}\cdots $ with 
$$\{0\}=p_{0}\subset p_{1}\subset\cdots \subset p_{n}\subset p_{n+1}\subset \cdots$$
Such a boundary point determines a unique set 
$\mathcal{L}^{(\alpha)}:=\bigcup_{n=0}^{\infty}p_{n}\subset \mathbb{R}^{d}$ and 
since $0\in\mathcal{L}^{(\alpha)}$ we have $\mathcal{L}^{(\alpha)}\in\Omega_0$. 
Conversely, any element $\mathcal{L}\in\Omega_0$ defines a boundary point by 
setting $p_{n}:=\mathcal{L}\cap B(0;nR)$. 

The topology on the boundary of a tree is defined by the so-called \emph{cylinder sets} associated to vertices
\[\mathcal{C}_{p}:=\{\alpha \in\partial\mathcal{T}: \alpha_{|p|}=p\}\simeq\{\omega\in\Omega_0: \mathcal{L}^{(\omega)}\cap B(0;nR)=p\}=U_{(nR,p)},\]
where the latter identification is given by sending a boundary point to its associated set. 
Thus the topology on $\partial \mathcal{T}$ matches that on $\Omega_0$ and the two 
spaces are homeomorphic. Equivalently the topology on $\partial \mathcal{T}$ can be defined through the ultrametric 
$$\rho(\alpha,\omega)=\min\{e^{-nR}:\exists p\in P_{nR}\quad \alpha,\omega \in \mathcal{C}_{p} \}.$$

By a \emph{choice function} we mean a map $\tau:\mathcal{V}\to \partial\mathcal{T}$ 
such that $\tau(v)\in \mathcal{C}_{v}$. A choice function defines a representation
\[\pi_{\tau}:C(\Omega_0)\to B(\ell^{2}(\mathcal{V})),\quad \pi(f)\phi(v):=f(\tau(v))\phi(v).\]
It is straightforward to verify that for any pair of choice functions $(\tau_{+},\tau_{-})$ the pair 
$(\pi_{\tau_{+}},\pi_{\tau_{-}})$ defines a quasi-homomorphism 
$C(\Omega_0)\to \mathbb{K}(\ell^{2}(\mathcal{V}))$ and hence a class in $KK(C(\Omega_0),\C)$~\cite{Cuntznewlook}. 
We associate a spectral triple to this data in the spirit of 
Pearson--Bellissard \cite{PearsonBellissard}, with some extra flexibility for 
reasons similar to those in \cite{GM15}, related to the pathologies of the unbounded Kasparov product.

\begin{prop}
\label{prop: BPlog}
 Let $(\tau_{+},\tau_{-})$ be a pair of choice functions, $\rho$ an ultrametric on $\Omega_0$ and let $\zeta:\mathbb{N}\to \mathbb{R}_{\geq 0}$ be a sequence with $\zeta_{n}\to\infty$ and for which there exists $C>0$ such that $\zeta_{n}\leq C\left(\sup_{p\in \calV_{n}}\textnormal{diam}_{\rho}\,\mathcal{C}_{p}\right)^{-1}$. The representation 
\[\pi(f)\begin{pmatrix}\phi_{+}\\ \phi_{-}\end{pmatrix}(v):=\begin{pmatrix}f(\tau_{+}(v))\phi_{+}(v)\\ f(\tau_{-}(v))\phi_{-}(v)\end{pmatrix},\]
and self-adjoint operator 
$$D\begin{pmatrix}\phi_{+}\\\phi_{-}\end{pmatrix}(v)=\begin{pmatrix} 0 & D_{-} \\ D_{+} &0\end{pmatrix}\begin{pmatrix}\phi_{+}\\\phi_{-}\end{pmatrix}(v):=\begin{pmatrix}\zeta_{|v|}\phi_{-}(v)\\ \zeta_{|v|}\phi_{+}(v)\end{pmatrix},$$
define a spectral triple $(\textnormal{Lip}(\Omega_0), \ell^{2}(\mathcal{V},\mathbb{C}^{2}),D)$ for $C(\Omega_0)$ whose $K$-homology class coincides with that of the quasi-homomorphism $(\pi_{\tau_{+}},\pi_{\tau_{-}})$.
\end{prop}
\begin{proof} The only thing to check is that the Lipschtiz functions for the metric $\rho$ have bounded commutators with each such $D$. This follows since
\[\|[D,f]\phi(v)\|=\zeta_{|v|}\|f(\tau_{+}(v))-f(\tau_{-}(v))\|\|\phi(v)\|,\]
and by assumption the sequence $\zeta_{|v|}$ satisfies $\zeta_{|v|}\leq C \rho(\tau_{+}(v),\tau_{-}(v))^{-1}$.
\end{proof}

The spectral triple constructed in \cite[Proposition 8]{PearsonBellissard} corresponds to choosing the sequence $\zeta_{n}:=e^{nR}$. Here we choose the sequence $\zeta_{n}:=\log(1+n)$. Before we proceed we record the following observation which serves as the main technical tool in our arguments below.
\begin{lemma}
\label{rosetta}
 Let $x,y\in B(0;nR)$ and $\|x-y\|<r$. If $\alpha,\omega\in\Omega_{0}$ are such that $x\in\mathcal{L}^{(\omega)},$ $y\in\mathcal{L}^{(\alpha)}$ and $\rho(T_{-x}\omega,T_{-y}\alpha)\leq e^{-nR}$, then $x=y$ and $\rho(\alpha,\omega)\leq e^{-nR+\|x\|}$.
\end{lemma}
\begin{proof}Since $\rho(T_{-x}\omega,T_{-y}\alpha)\leq e^{-nR}$ it holds that
\[\mathcal{L}^{(T_{-x}\omega)}\cap B(0;nR)=\mathcal{L}^{(T_{-y}\alpha)}\cap B(0;nR),\]
and $x,y\in \mathbb{R}^{d}\cap B(0;nR)$ gives
\[-x,-y \in \mathcal{L}^{(T_{-x}\omega)}\cap B(0;nR)=\mathcal{L}^{(T_{-y}\alpha)}\cap B(0;nR), \]
and since $\|x-y\|<r$ it follows that $x=y$. Then because
\[B(-x;nR-\|x\|)\subset B(0;nR),\quad T_{x}(B(-x;nR-\|x\|))=B(0;nR-\|x\|)\]
it follows that
\[\mathcal{L}^{(\omega)}\cap B(0;nR-\|x\|)=\mathcal{L}^{(\alpha)}\cap B(0;nR-\|x\|).\]
This means that $\rho(\alpha,\omega)\leq e^{-(nR-\|x\|)}=e^{-nR+\|x\|}.$
\end{proof}
\subsection{The product operator}
We now proceed to describe the product operator (in the sense of \cite{MR}) defined from the 
unbounded Kasparov module of Proposition \ref{prop:bulk_kasmod} and the Pearson--Bellissard 
spectral triples of Proposition \ref{prop: BPlog}. 
Because the formulas that appear in this section are somewhat involved, we condense our notation for the 
groupoid $\calG$. Namely, let $\xi = (\omega,x)\in\calG$ be a generic groupoid element and 
let $x(\xi) \in \R^d$ be the image of the cocycle $(\omega,x)\mapsto x \in \R^d$ with 
$x_k(\xi)$ the $k$-th component, $x_k$. 
Furthermore, to reduce computational complexity, for the remainder of this section we set $\sigma=1$. 
The case of a non-trivial 2-cocycle twist requires a separate treament and involves even longer computations, though we expect 
the analytic subtleties to be similar.

  Given a choice function $\tau:\calV\to\partial\calT=\Omega_0$, 
consider the fiber product
\[\mathcal{G}_{s}\times_{\tau}\mathcal{V}:=\{(\xi,v)\in \mathcal{G}\times\mathcal{V}: s(\xi)=\tau(v)\}.\]
Denote by $L^{2}(\calG_{s}\times_{\tau}\calV)$ the Hilbert space completion of 
$C_{c}(\mathcal{G}_{s}\times_{\tau}\mathcal{V})$ in the inner product
\[\langle \phi,\psi\rangle =\sum_{v\in\mathcal{V}}\sum_{\xi,s(\xi)=\tau(v)} \overline{\phi(\xi,v)}\psi(\xi,v).\]

The following lemma is a straightforward verification.
\begin{lemma} 
Let $\tau:\calV\to\Omega_0$ be a choice function. The map
\[\alpha: C_{c}(\mathcal{G})\otimes^{\textnormal{alg}}_{\pi_{\tau}}C_{c}(\calV)\to C_{c}(\calG_{s}\times_{\tau}\calV),\quad \alpha(f\otimes\psi)(\xi,v):=f(\xi)\psi(v),\]
extends to a unitary isomorphism $E_{C(\Omega_0)}\otimes_{\pi_{\tau}}L^{2}(\calV)\xrightarrow{\sim} L^{2}(\calG_{s}\times_{\tau}\calV)$. The left representation of $C_{c}(\mathcal{G})$ is concretely expressed as
\[f * \phi(\eta,w)=\sum_{s(\xi)=r(\eta)}f(\xi) \phi(\xi^{-1}\eta,v).\]
\end{lemma}

Using this lemma, we can decompose the tensor product Hilbert space 
via the choice function,
\[\calH=\calH_{+}\oplus \calH_{-}=L^{2}(\mathcal{G}_{s}\times_{\tau_{+}}\calV)\hat\otimes 
    \bigwedge\nolimits^{\! *} \R^d\oplus L^{2}(\calG_{s}\times_{\tau_{-}}\calV)\hat\otimes 
    \bigwedge\nolimits^{\! *} \R^d,
    \]
    though we note that $\calH_\pm$ is \emph{not} the decomposition of the tensor product 
    Hilbert space due to the grading, which also takes into account the $\Z_2$-graded 
    structure of $\bigwedge\nolimits^{\! *} \R^d$. 
On this Hilbert space the operator $X$ from the bulk $KK$-cycle in 
Equation \eqref{eq:bulk_K-cycle} on page \pageref{eq:bulk_K-cycle} is mapped to the operator 
\[X=\sum_{k=1}^{d}X_{k} \hat\otimes \gamma^{k}:\calH_{\pm}\to \calH_{\pm},\quad X(\phi\otimes w)(\xi,v)
  =\sum_{k=1}^{d}x_{k}(\xi)\phi(\xi,v) \hat\otimes\gamma^{k} w.\]
We fix $\varepsilon$ with $0<\varepsilon<\frac{r}{2}$, a discrete lattice 
$Y\subset\mathbb{R}^{d}$ and a uniformly locally finite cover for $\R^d$ with subordinate partition of unity
\[\mathcal{Y}:=\{B(y;\varepsilon)\}_{y\in Y},\quad\chi_{y}:B(y;\varepsilon)\to [0,1],\quad\sum_{y\in Y}\chi_{y}^{2}=1,\]
Recalling Proposition \ref{Deloneframe}, from $\mathcal{Y}$, we consider the sets $\{V_y\}_{y\in Y}$,   
$$
  V_y = \big\{ \xi=(\omega,x)\in \Omega_0\times\R^d \,:\, x \in \calL^{(\omega)}\cap B(y;\varepsilon) \big\},
$$
which form an $s$-cover of $\calG$. Consequently 
the functions $\chi_{y}:\calG\to\mathbb{R}$ define 
a frame for $E_{C(\Omega_0)}$. In order to construct the connection operator we wish to describe the maps
\[\langle \chi_{y}^{\pm}|:L^{2}(\mathcal{G}\times_{\tau_{\pm}}\mathcal{V})\to \ell^{2}(\mathcal{V}),\quad |\chi_{y}^{\pm}\rangle:\ell^{2}(\mathcal{V})\to L^{2}(\mathcal{G}_{s}\times_{\tau_{\pm}}\mathcal{V}).\]
Since the support $\chi_{y}$ is a compact subset of $B(y;\varepsilon)$, 
the convolution product takes the form
\begin{align*} \chi_{y}^{*}*f(\eta, v) &=\sum_{\xi\in r^{-1}(r(\eta))} \chi_{y}^{*}(\xi)f(\xi^{-1}\eta,v) 
      =\sum_{\xi\in r^{-1}(r(\eta))} \chi_{y}(\xi^{-1})f(\xi^{-1}\eta,v)  \\
&=\sum_{\{ \xi\in s^{-1}(r(\eta))\cap V_{y}\}}\chi_{y}(\xi)f(\xi\eta,v) 
 =\chi_{y}(\xi)f(\xi\eta,v)  ,\quad\textnormal{with } \xi\in s^{-1}(r(\eta))\cap V_{y},
\end{align*}
and $0$ when the latter set is empty. This shows that the maps become
\[\langle \chi_{y}^{\pm}|:L^{2}(\mathcal{G}\times_{\tau_{\pm}}\mathcal{V})\to \ell^{2}(\mathcal{V}),\quad \langle\chi_{y}^{\pm}|\phi(v):=\chi_{y}(\xi_{\pm}^{y}(v))\phi(\xi_{\pm}^{y}(v),v),\]
whenever $V_{y}\cap s^{-1}(\tau_{\pm}(v))\neq\emptyset$
and $\xi_{\pm}^{y}(v)$ is the unique point in $V_{y}\cap s^{-1}(\tau_{\pm}(v))$. In case $V_{y}\cap s^{-1}(\tau_{\pm}(v))=\emptyset$ we have $\langle\chi_{y}^{\pm}|\phi(v)=0$.
We can now define the operators
\[T_{\pm}:C_{c}(\mathcal{G}\times _{\tau_{\pm}}\mathcal{V})\to C_{c}(\mathcal{G}\times_{\tau_{\mp}}\mathcal{V}),\]
by
\[T_{+}\phi_{+}(\eta,v)=\sum_{y}\sum_{\xi\in s^{-1}(\tau_{+}(v))} \zeta_{|v|}\chi_{y}(\eta)\chi_{y}(\xi)\phi_{+}(\xi,v),\quad s(\eta)=\tau_{-}(v).\]
The above sum is in fact finite for each $(\eta,v)\in \mathcal{G}\times_{\tau_{-}}\mathcal{V}$, 
since the summands are nonzero only for those $y$ with $\eta\in V_{y}$ and 
$V_{y}\cap s^{-1}(\tau_{+}(v))\neq \emptyset$. For $T_{-}$ we have an analogous formula. 

The operators $T_{\pm}$ can be viewed as being constructed from the Grassmann connection associated to the frame $\{\chi_{y}\}$ as in \cite[Section 3.4]{MR}. We use the methods developed there to address self-adjointness properties of these operators.
\begin{lemma}The operator 
$$T:=\begin{pmatrix}0 & T_{-}\\ T_{+} & 0\end{pmatrix}:C_{c}(\mathcal{G}\times_{\tau_{+}}\mathcal{V})\oplus C_{c}(\mathcal{G}_{s}\times_{\tau_{-}}\mathcal{V})\to L^{2}(\mathcal{G}\times_{\tau_{+}}\mathcal{V})\oplus L^{2}(\mathcal{G}_{s}\times_{\tau_{-}}\calV),$$ is essentiallly self-adjoint.
\end{lemma}
\begin{proof} 
For fixed $z$ the continuous functions
\[ \big(\chi_{y} \mid \chi_{z} \big)_{C(\Omega_0)}(\omega) =\sum_{\xi\in s^{-1}(\omega)}\chi_{y}(\xi)\chi_{z}(\xi),\]
are possibly nonzero only for those $y$ with $B(y;\varepsilon)\cap B(z;\varepsilon)\neq\emptyset$. There are only 
finitely many such $y$ since the cover $\mathcal{Y}$ has finite intersection number. Moreover they are locally constant since for 
$\rho(\alpha,\omega)<e^{-nR}$ sufficiently small
 we have $$(\chi_y\mid \chi_z)_{C(\Omega_0)}(\alpha) = (\chi_y\mid \chi_z)_{C(\Omega_0)}(\omega),$$ 
 by Lemma \ref{rosetta}. 
 Thus the frame $\{\chi_{y}\}_{y \in Y}$ is column finite in the sense of \cite[Proposition 3.2]{MR}, 
 the operators $\Theta_{z,z}:=\Theta_{\chi_z,\chi_z}$  preserve a core for $T$ 
 by \cite[Lemma 3.15]{MR} and the commutators $ [T, \Theta_{z,z}]$ 
 extend to bounded operators by \cite[Lemma 3.8]{MR}.
 
It remains to show that there exists an approximate unit $u_n$ in the convex hull of the $\Theta_{z,z}$ that 
satisfies \cite[Definition 3.9]{MR}. For a fixed $n$, consider the set
\[I_{n}:=\bigcup_{v\subset§ B(0;nR)}\{y\in Y: s^{-1}(\tau_{\pm}(v))\cap V_{y}\neq \emptyset\},\]
and consider
the operator $u_{n}:=\sum_{y\in I_{n}} \Theta_{y,y}$. Since
\[[T,\Theta_{z,z} ]\phi(\eta,v)=\sum_{y,\xi}(\chi_{z}(\xi)^{2}-\chi_{z}(\eta)^{2})\zeta_{|v|}\chi_{y}(\xi)\chi_{y}(\eta)\phi(\xi,v),\]
we see that for $|v|\geq nR$, Lemma \ref{rosetta} gives that $x(\xi)=x(\eta)$ and thus $\chi_{z}(\xi)=\chi_{z}(\eta)$, so we have $[T,u_{n}]\phi(\eta,v)=0$. For $|v|\leq nR$ we find
\[\sum_{y\in I_{n}}(\chi_{y}(r(\xi))^{2}-\chi_{y}(r(\eta))^{2})\zeta_{|v|}\chi_{k}(\xi)\chi_{k}(\eta)\phi(\xi,v)=0,\]
because $\xi,\eta\in s^{-1}(\tau_{\pm}(v))$ and $v\in B(0;nR)$ so $\sum_{y\in I_{n}}\chi_{y}^{2}(\xi)=\sum_{y\in I_{n}}\chi_{y}^{2}(\eta)=1$. This proves that $[T,u_{n}]=0,$ so $\{\chi_{y}\}_{y\in Y}$ form a complete frame and $T$ is essentially self-adjoint by \cite[Theorem 3.18]{MR}.
\end{proof}
Denote by 
\[\mathcal{C}=\mathcal{C}_{+}\oplus \mathcal{C}_{-}:=C_{c}(\mathcal{G}\times_{\tau_{+}}\mathcal{V})\hat\otimes 
    \bigwedge\nolimits^{\! *} \R^d\oplus C_{c}(\mathcal{G}\times_{\tau_{-}}\mathcal{V})\hat\otimes 
    \bigwedge\nolimits^{\! *} \R^d,\]
    the common core for $X$ and $T$ and by $\kappa$ the grading operator on $\bigwedge\nolimits^{\! *} \R^d$.
  Then we have $X\kappa=-\kappa X$ and $T\kappa=\kappa T$. We now address self-adjointness of the densely defined symmetric Hilbert space operator $D=X+T\kappa$, using the methods of \cite{MesLesch}.
\begin{prop} 
\label{prop: prodcandidate}
The resolvent $(X\pm i)^{-1}$ maps the core $\mathcal{C}$ bijectively onto itself. For $\phi\in \mathcal{C}_{\pm}$ 
and $w\in \bigwedge^*\R^d$ we have the estimate
\[\big\langle (XT\kappa+T\kappa X) (\phi\otimes w), (XT\kappa+T\kappa X) (\phi\otimes w) \big\rangle 
  \leq r^{2}\| T\kappa( \phi\otimes w)\|^{2}.\]
Consequently the sum operator $D:=X+T\kappa$ is essentially self-adjoint with compact resolvent and the bounded transforms of $X$ and $D$ satisfy the Connes--Skandalis positivity and connection conditions (see \cite[Appendix A]{longitudinal}). 
\end{prop}
\begin{proof}
The statement about the resolvent is immediate since $X$ is given by Clifford multiplication by a real valued function. 
Thus the anticommutator
$XT\kappa+T\kappa X=(XT-TX)\kappa$ is defined on $\mathcal{C}$. The
commutator $XT-TX$ can be explicitly computed as
\begin{align*}
(TX-XT)(\phi(\eta,v)\otimes w)&=\sum_{k=1}^{d}\sum_{y,\xi}(x_{k}(\xi)-x_{k}(\eta))\zeta_{|v|}\chi_{y}(\xi)\chi_{y}(\eta)\phi(\xi,v) \otimes \gamma^k w \\
&=\sum_{k=1}^{d}(x_{k}(\xi)-x_{k}(\eta)) (T\phi)(\eta,v) \otimes \gamma^{k} w ,
\end{align*}
and since the $\gamma^{k}$ are Clifford matrices it holds that
\[
\left\|\sum_{k=1}^{d}(x_{k}(\xi)-x_{k}(\eta))(T\phi)(\eta,v) \otimes \gamma^k w \right\|^{2}
=\sum_{k}\|x_{k}(\xi)-x_{k}(\eta)\|^{2}\| T\kappa( \phi\otimes w)\|^{2}.
\]
Since
$x(\xi),x(\eta)\in B(y;\varepsilon)$ it follows that $\sum_{k}\|x_{k}(\xi)-x_{k}(\eta)\|^{2}< 4\varepsilon^{2}\leq r^{2}$, 
which gives us the desired estimate. Self-adjointness, compact resolvent and positivity follow 
from~\cite[Theorem 4.5, Theorem 7.4, Proposition 7.12]{MesLesch} and the connection condition 
follows from~\cite[Theorem 4.4]{MR}.
\end{proof}
\begin{remark} Note that we have not yet shown that $D$ has bounded commutators with 
$C_{c}(\calG)$ and that this is the only obstruction to $D$ representing the unbounded 
Kasparov product via \cite[Theorem 7.4]{MesLesch}. In fact the operator $X$ has bounded 
commutators with all $f\in C_{c}(\mathcal{G})$, but the operator $T$ does not. In the 
next section we will show that whenever $\zeta_{|v|}$ is chosen so that it grows sufficiently slowly, 
the bounded transform of $D$ will be a Fredholm module. This Fredholm module will satisfy the 
Connes--Skandalis connection and positivity conditions by the previous proposition, and thus 
represents the Kasparov product.
\end{remark}

\subsection{The bounded transform as a Fredholm module}
Recall that a continuous function $b:\mathbb{R}\to [-1,1]$ is a \emph{normalising function} if it is 
odd and $\lim_{x\to\pm\infty}b(x)=\pm 1$. To prove that for the right choice of $\zeta_{|v|}$ 
we obtain a Fredholm module we use the following Lemma.

\begin{lemma}
\label{shortcut}
Let $(S,T)$ be a weakly anticommuting pair of self-adjoint regular operators on the Hilbert $C^{*}$-module $E_B$, 
$a\in\End^{*}(E_B)$, $b:\mathbb{R}\to [-1,1]$  normalising function and $0<\delta<\frac{1}{2}$. Write $D=S+T$ and suppose that $a(D\pm i)^{-1}$ is compact and the operators 
\[[S,a],\quad (1+S^{2})^{-\delta}[T,a],\quad [T,a](1+S^{2})^{-\delta},\] extend boundedly to all of $E_B$. Then $[b(D),a]$ is a compact operator.
\end{lemma}
\begin{proof} We need only show that $[T,a](1+D^{2})^{-\delta}$ and its adjoint extend to bounded operators. Then \cite[Theorem A.6]{GM15} applies to reach the conclusion. Since we have the factorisation
\[[T,a](1+D^{2})^{-\delta}=[T,a](1+S^{2})^{-\delta}(1+S^{2})^{\delta}(1+D^{2})^{-\delta},\]
it suffices to show that $(1+S^{2})^{\delta}(1+D^{2})^{-\delta}$ is bounded. Now if $R$ is a densely defined operator on $E_B$ with bounded adjoint, then
\begin{align*}\sup_{e_1 \in\Dom R,\, e_2\in E}\| (Re_1\mid e_2)_B \|=\sup_{e_1\in\Dom R,\, e_2\in E}\| (e_1\mid R^*e_2)_B \|\leq C\|e_1\|\|e_2\|,
\end{align*}
so $R$ is bounded on its domain.
Since $D=S+T$ is self-adjoint and regular on $\Dom S\cap \Dom T$ the operators $(S\pm i)(D\pm i)^{-1}$ 
are everywhere defined and closed, hence bounded. Their adjoints
are the extension of $(D\pm i)^{-1}(S\pm i)$. Hence
\begin{align*}
( (1+D^{2})^{-1}e\mid e)_{B} &=((D+ i)^{-1}e\mid (D+i)^{-1}e)_{B} \\ &=( (D+ i)^{-1}(S+ i)(S+i)^{-1}e \mid  (D+ i)^{-1}(S+ i)(S+i)^{-1}e)_{B} \\
&\leq C( (S+i)^{-1}e\mid (S+i)^{-1}e)_{B} =C( (1+S^{2})^{-1}e\mid e)_{B}.
\end{align*}
Thus it holds that 
\[(1+D^{2})^{-1}\leq C(1+S^{2})^{-1}.\]
For $0<\delta<\frac{1}{2}$ we have the form estimate
\begin{align*}
( (1+D^{2})^{-\delta}(1+S^{2})^{\delta}e\mid  (1+D^{2})^{-\delta}(1+S^{2})^{\delta}e)_{B} \leq C^{\delta}( e\mid e)_{B} ,
\end{align*}
which proves that the adjoint is bounded. 
\end{proof}
By Proposition \ref{prop:s-cover_frame} every $f\in C_{c}(\mathcal{G})$ can be expressed as a finite sum $f=\sum_{y}\chi_{y}f_{y}$ with $f_{y}\in C(\Omega_0)$.  it follows that $f^{*}=\sum f_{y}^{*}\chi_{y}^{*}$ and since $C_{c}(\mathcal{G})$ is closed under the adjoint operation it thus suffices to show that for all $f\in\textnormal{Lip}(\Omega_0)$ and all $\chi_{k}^{*}$ Lemma \ref{shortcut} is satisfied for certain choices of $\zeta_{|v|}$. From now on we fix the choice $\zeta_{|v|}:=\log(1+|v|)$.

\begin{lemma}
\label{loglemma}
 Let $f\in\textnormal{Lip}(\Omega_0)$. Then
\[\|f\|_{\log}:=\sup_{\alpha \neq \omega}\frac{|f(\alpha)-f(\omega)|}{\log (1-\log(\rho(\alpha,\omega)))}<\infty,\]
and so $\|f(\alpha)-f(\omega)\|\leq \|f\|_{\log}\log (1-\log(\rho(\alpha,\omega)))$.
\end{lemma}
\begin{proof}Since $0\leq \rho(\alpha,\omega)\leq 1$ it follows that 
$$\rho(\alpha,\omega)\leq \log (e-\log(\rho(\alpha,\omega)))\leq \log(1-\log(\rho(\alpha,\omega)))+\log(e-1).$$ 
So for $\rho(\alpha,\omega)$ small there is a uniform constant with
\[\rho(x,y)\leq C \log (1-\log(\rho(\alpha,\omega))),\]
and $\|f\|_{\log}\leq \frac{\|f\|_{\textnormal{Lip}}}{C}$. The statement follows.
\end{proof}
\begin{lemma} 
\label{Ncounting}
Let $\mathcal{V}:=\{V_{y}\}_{y\in Y}$ be an $s$-cover of a groupoid $\mathcal{G}$ with intersection number $N$ and $\chi_{y}$ a subordinate partition of unity. For  $\eta\in\calG$ and $\omega \in\mathcal{G}^{(0)}$ fixed, the set
\[Y_{\eta,\omega}:=\{(\xi,y)\in s^{-1}(\omega)\times Y:\chi_{y}(\xi)\chi_{y}(\eta)\neq 0\},\]
contains at most $N$ elements.
\end{lemma}
\begin{proof} First of all observe that  $(\xi,y)\in Y_{\eta,\omega}$ only if $\xi,\eta\in V_{y}$ and there 
can be at most $N$ distinct indices $y$ for which $\eta\in V_{y}$. Secondly  if $(\xi,y),(\xi',y)\in Y_{\eta,\omega}$ 
then since $V_{y}$ is an $s$-cover it follows that $\xi=\xi'$.  Thus there are at most $N$ distinct pairs 
$(\xi,y)\in Y_{\eta,\omega}$.
\end{proof}

\begin{prop} For $f\in\textnormal{Lip}(\Omega_0)$ and $\delta>0$ the operators $(1+X^{2})^{-\delta}[T,f]$ and $[T,f](1+X^{2})^{-\delta}$ extend to bounded operators.
\end{prop}
\begin{proof}
Because $X^2$ and $T$ act diagonally on the finite dimensional Clifford representation 
space $\bigwedge^*\R^d$, it suffices to prove boundedness on 
$L^2(\calG_s \times_{\tau_+} \calV) \oplus L^2(\calG_s \times_{\tau_-} \calV)$.
For $f\in\textnormal{Lip}(\Omega_0)$ the commutator takes the simple form
\begin{align*}[T_{+},f]&(1+X^{2})^{-\delta}\phi_{+}(\eta,v)=\\ &\sum_{y}\sum_{\xi\in s^{-1}(\tau_{+}(v))} \zeta_{|v|}(1+\|x(\xi)\|^{2})^{-\delta}(f(r(\xi))-f(r(\eta)))\chi_{y}(\eta)\chi_{y}(\xi)\phi_{+}(\xi,v),
\end{align*}
with $s(\eta)=\tau_{-}(v)$. We consider the two cases namely
\[\xi\in J_{v}:=x^{-1} (B(0;|v|R-r)),\quad \xi\notin J_{v}.\]
In the first case we have $\chi_{y}(\xi),\chi_{y}(\eta)\neq 0$ only if
$x(\eta)\in B(0;|v|R),$
and Lemma \ref{rosetta} gives that $x(\xi)=x(\eta)$. Then applying Lemma \ref{loglemma} yields the estimate
\[\|f(r(\xi))-f(r(\eta))\|\leq \|f\|_{\textnormal{log}}\log(1+|v|R-\|x(\xi)\|). \]
Since
\[\sup_{v}\sup_{\xi\in J_{v}}\log(1+|v|R-\|x(\xi)\|)(1+\|x(\xi)\|^{2})^{-\delta}\log(1+|v|)<\infty,\]
and denoting by $J_{v}^{c}$ the complement of $J_{v}$,
\[\sup_{v}\sup_{\xi\in J_{v}^{c}}\zeta_{|v|}(1+\|x(\xi)\|^{2})^{-\delta}<\infty,\]
so we have the following norm estimates (with $C$ denoting a generic constant):
\begin{align*}\
  \big\| [T_{+},f](1+X^{2})^{-\delta}\phi_{+}(\eta,v) \big\|^{2} &=  \\ 
&\hspace{-4cm} \sum_{\substack{v\in\calV, \\ \eta\in s^{-1}(\tau_{-}(v))}} \left\|\sum_{\xi\in s^{-1}(\tau_{+}(v))} \sum_{y}\zeta_{|v|}(1+\|x(\xi)\|^{2})^{-\delta}(f(r(\xi))-f(r(\eta)))\chi_{y}(\eta)\chi_{y}(\xi)\phi_{+}(\xi,v)\right\|^{2}\\
&\hspace{-2cm} \leq  C\sum_{v,\eta}\left(\sum_{\xi\in s^{-1}(\tau_{+}(v))} \left\|\sum_{y}\chi_{y}(\eta)\chi_{y}(\xi)\phi_{+}(\xi,v)\right\|\right)^{2} \\
  &\hspace{-2cm} \leq C\sum_{v,\eta}\left(\sum_{\xi\in s^{-1}(\tau_{+}(v))} \sum_{y}\chi_{y}(\eta)\chi_{y}(\xi)\left\|\phi_{+}(\xi,v)\right\|\right)^{2}\\
  &\hspace{-2cm} \leq  CN\sum_{v,\eta}\sum_{\xi\in s^{-1}(\tau_{+}(v))} \sum_{y}\chi_{y}(\eta)^{2}\chi_{y}(\xi)^{2}\|\phi_{+}(\xi,v)\|^{2},
\end{align*}
where we have used Lemma \ref{Ncounting} and the estimate $(a_{1}+\cdots+a_{N})^{2}\leq N(a_{1}^{2}+\cdots + a_{N}^{2})$. Now we use that for a fixed $\xi\in s^{-1}(\tau_{+}(v))$ and $y\in Y$ with $\chi_{y}(\xi)\neq 0$ there is at most one $\eta\in s^{-1}(\tau_{-}(v))$ with $0\neq \chi_{y}(\eta)\leq 1$ so
\begin{align*}
 \big\| [T_{+},f](1+X^{2})^{-\delta}\phi_{+}(\eta,v) \big\|^{2} \leq &\, CN\sum_{v\in\calV}\sum_{\xi\in s^{-1}(\tau_{+}(v))}\sum_{y}\chi_{y}(\xi)^{2}\|\phi_{+}(\xi,v)\|^{2}\\
\leq & \, CN\sum_{v\in\calV}\sum_{\xi\in s^{-1}(\tau_{+}(v))} \|\phi_{+}(\xi,v)\|^{2}=N\|\phi\|^{2},
\end{align*}
it follows that $[T,f](1+X^{2})^{-\delta}$ defines a bounded operator for all $\delta>0$.
\end{proof}
Next we consider the commutator $[T,\chi_{z}^{*}]$ with $\chi_{z}$ the frame elements. We obtain the same statement for them.
\begin{lemma}For $\phi\in C_{c}(\mathcal{G}_{s}\times_{\tau_{+}}\calV)$ we have
\begin{align}\nonumber &\langle [T_{+},\chi_{z}^{*}]\phi , [T_{+},\chi_{z}^{*}]\phi\rangle=\\
\label{commnorm}&\sum_{v\in\mathcal{V}}\sum_{\eta\in s^{-1}(\tau_{-}(v))}\zeta_{|v|}^{2}\left\|\sum_{\xi\in s^{-1}(\tau_{+}(v))}\sum_{y,\alpha,\beta}\left(\chi_{z}(\alpha)\chi_{y}(\beta\eta)\chi_{y}(\alpha\xi)-\chi_{z}(\beta)\chi_{y}(\eta)\chi_{y}(\xi)\right)\phi(\xi,v) \right\|^{2},
\end{align}
where we used the shorthand notation
$\sum_{y,\alpha,\beta}:=\sum_{\alpha\in s^{-1}(r(\xi))}\sum_{\beta\in s^{-1}(r(\eta))}\sum_{y\in Y}$.
\end{lemma}
\begin{proof} The formula is obtained by direct calculation. First we compute the commutator acting on a function $\phi \in C_{c}(\calG {}_{s}\times_{\tau_+} \calV)$: 
\begin{align*}
[T_{+}&,\chi_{z}]\phi(\eta,v)=\sum_{y}\sum_{\xi\in s^{-1}(\tau_{+}(v))}\sum_{ \alpha\in r^{-1}(r(\xi))}\zeta_{|v|}\chi_{y}(\eta)\chi_{y}(\xi)\chi_{z}(\alpha)\phi(\alpha^{-1}\xi,v)\\
&\quad\quad\quad\quad\quad\quad\quad-\sum_{y}\sum_{ \beta\in r^{-1}(r(\eta))}\sum_{\xi\in s^{-1}(\tau_{+}(v))}\zeta_{|v|}\chi_{z}(\beta)\chi_{y}(\beta^{-1}\eta)\chi_{y}(\xi)\phi(\xi,v)\\
&=\sum_{y}\sum_{\xi\in s^{-1}(\tau_{+}(v))}\sum_{ \alpha\in s^{-1}(r(\xi))}\zeta_{|v|}\chi_{z}(\alpha)\chi_{y}(\eta)\chi_{y}(\alpha\xi)\phi(\xi,v)\\
&\quad-\sum_{y}\sum_{ \beta\in r^{-1}(r(\eta))}\sum_{\xi\in s^{-1}(\tau_{+}(v))}\zeta_{|v|}\chi_{z}(\beta)\chi_{y}(\beta^{-1}\eta)\chi_{y}(\xi)\phi(\xi,v)\\
&=\sum_{y,\xi}\zeta_{|v|}\left(\sum_{ \alpha\in s^{-1}(r(\xi))}\chi_{z}(\alpha)\chi_{y}(\eta)\chi_{y}(\alpha\xi)-\sum_{ \beta\in r^{-1}(r(\eta))}\chi_{z}(\beta)\chi_{y}(\beta^{-1}\eta)\chi_{y}(\xi)\right)\phi(\xi,v).
\end{align*}
The $L^2$-norm of the vector $[T_{+},\chi_{z}^{*}]\phi$ is thus computed as
\begin{align*}
&\langle [T_{+},\chi_{z}^{*}]\phi, [T_{+},\chi_{z}^{*}]\phi\rangle =\sum_{v\in\mathcal{V}}\sum_{\eta\in s^{-1}(\tau_{-}(v))} \|[T_{+},\chi_{k}^{*}]\phi(\eta,v)\|^{2}\\
&=\sum_{v,\eta}\zeta_{|v|}^{2} \left\|\sum_{y,\xi}\left(\sum_{ \alpha\in s^{-1}(r(\xi))}\chi_{z}(\alpha)\chi_{y}(\eta)\chi_{y}(\alpha\xi)-\sum_{ \beta\in r^{-1}(r(\eta))}\chi_{z}(\beta)\chi_{y}(\beta^{-1}\eta)\chi_{y}(\xi)\right)\phi(\xi,v) \right\|^{2}\\
&=\sum_{v,\eta}\zeta_{|v|}^{2}\left\|\sum_{y,\xi}\left(\sum_{ \alpha\in s^{-1}(r(\xi))}\sum_{ \beta\in s^{-1}(r(\eta))}\chi_{z}(\alpha)\chi_{y}(\beta\eta)\chi_{y}(\alpha\xi)-\chi_{z}(\beta)\chi_{y}(\eta)\chi_{y}(\xi)\right)\phi(\xi,v) \right\|^{2}\\
&=\sum_{v,\eta}\zeta_{|v|}^{2}\left\| \sum_{\xi } \sum_{y,\alpha,\beta} \left(\chi_{z}(\alpha)\chi_{y}(\beta\eta)\chi_{y}(\alpha\xi)-\chi_{z}(\beta)\chi_{y}(\eta)\chi_{y}(\xi)\right)\phi(\xi,v) \right\|^{2}.
\end{align*}
This is the desired formula.
\end{proof}
 In the inner product expression \eqref{commnorm} on page \pageref{commnorm}, we split the sum over $\mathcal{V}$ into a sum over 
\[\calV_{z}:=\{v\in\calV: z\in B(0;|v|R-r)\}, \quad\textnormal{and}\quad \calV\setminus\calV_{z}=\{v\in\calV: z\notin B(0;|v|R-r)\}.\]
The sum over $\calV\setminus\calV_{z}$ is easily seen to define a bounded operator $B$. 
We further examine the expression that occurs inside the norm bars in \eqref{commnorm}, namely
\begin{equation}
\label{jsum}
\sum_{\xi\in s^{-1}(\tau_{+}(v))} \sum_{ \alpha\in s^{-1}(r(\xi))}\sum_{ \beta\in s^{-1}(r(\eta))}\sum_{y}\left(\chi_{z}(\alpha)\chi_{y}(\beta\eta)\chi_{y}(\alpha\xi)-\chi_{z}(\beta)\chi_{y}(\eta)\chi_{y}(\xi)\right)\phi(\xi,v), 
\end{equation}
for $v\in \calV_{z}$ and $\eta\in s^{-1}(\tau_{-}(v))$ fixed. We need to further distinguish between two cases for $\xi\in s^{-1}(\tau_{+}(v))$ with $\xi,\eta\in V_{y}$. For the fixed function $\chi_{z}$ we split the sum over $\xi\in s^{-1}(\tau_{+}(v))$ into a sum over
\[J_{(z,v)}:=\{\xi\in s^{-1}(\tau_{+}(v)): x(\xi)\notin B(0; |v|R-\|z\|-r)\},\quad\textnormal{and}\quad s^{-1}(\tau_{+}(v))\setminus J_{(z,v)},\]
for which we obtain the following expression:

\begin{lemma}We have an equality
\begin{align*}
&\sum_{\substack{v\in\calV, \\ \eta\in s^{-1}(\tau_{-}(v))}} \zeta_{|v|}^{2}\left\| \sum_{\xi\in s^{-1}(\tau_{+}(v))} \sum_{y,\alpha,\beta}\left(\chi_{z}(\alpha)\chi_{k}(\beta\eta)\chi_{y}(\alpha\xi)-\chi_{z}(\beta)\chi_{y}(\eta)\chi_{y}(\xi)\right)\phi(\xi,v) \right\|^{2}\\
&=\langle B\phi,B\phi\rangle+\sum_{\substack{v\in\calV_z, \\ \eta\in s^{-1}(\tau_{-}(v))}} \zeta_{|v|}^{2}\left\|\sum_{\xi\in J_{(z,v)}}\sum_{y,\alpha,\beta}\left(\chi_{z}(\alpha)\chi_{y}(\beta\eta)\chi_{y}(\alpha\xi)-\chi_{z}(\beta)\chi_{y}(\eta)\chi_{y}(\xi)\right)\phi(\xi,v) \right\|^{2},
\end{align*}
where $B$ is a bounded operator.
\end{lemma}
\begin{proof} 
We need to show that the sum over
\[\xi\in s^{-1}(\tau_{+}(v))\setminus J_{(z,v)}=\{\xi\in s^{-1}(\tau_{+}(v)): x(\xi)\in B(0; |v|R-\|z\|-r)\}  \]
vanishes. To this end we prove the implication
\begin{equation}
\label{implication}
\xi\in s^{-1}(\tau_{+}(v))\setminus J_{(z,v)}\Rightarrow x(\xi)=x(\eta), \quad x(\alpha)=x(\beta),\quad x(\alpha\xi)=x(\beta\eta).
\end{equation}
Using \eqref{implication} we deduce that the sum over $J_{(z,v)}^{c}:=s^{-1}(\tau_{+}(v))\setminus J_{(z,v)}$ vanishes, because $\chi_{z}(\eta)$ depends only on $x(\eta)$:
\begin{align*}
&\sum_{y}\sum_{\xi\in J_{(z,v)}^{c}}\sum_{ \alpha\in s^{-1}(r(\xi))}\sum_{ \beta\in s^{-1}(r(\eta))}\left(\chi_{z}(\alpha)\chi_{y}(\beta\eta)\chi_{y}(\alpha\xi)-\chi_{z}(\beta)\chi_{y}(\eta)\chi_{y}(\xi)\right)\phi(\xi,v)\\
&=\sum_{\xi\in J_{(z,v)}^{c}}\sum_{ \alpha\in s^{-1}(r(\xi))}\sum_{ \beta\in s^{-1}(r(\eta))}\sum_{y}\chi_{z}(\beta)\left(\chi_{y}(\beta\eta)^{2}-\chi_{y}(\eta)^{2}\right)\phi(\xi,v)=0, 
\end{align*}
and we are left with the sum over the complement $J_{(z,v)}$. It thus remains to show \eqref{implication} holds true.

 Let $\eta\in s^{-1}(\tau_{-}(v))$ and $\xi\in s^{-1}(\tau_{+}(v))$ with $x(\xi),x(\eta)\in B(y;\varepsilon)$ as well as $\xi\in s^{-1}(\tau_{+}(v))\setminus J_{(z,v)}$.
 First observe that we have
\[\mathcal{L}^{(\tau_{+}(v))}\cap B(0; |v|R)=\mathcal{L}^{(\tau_{-}(v))}\cap B(0;|v|R),\]
since  $\rho(\tau_{+}(v),\tau_{-}(v))\leq e^{-|v|R}$. Then since $$x(\xi)\in B(0;|v|R-\|z\|-r)\subset B(0;|v|R-r)$$ 
and $\|x(\xi)-x(\eta)\|<r$ it must hold that $x:=x(\xi)=x(\eta)$. We also conclude that 
$$\rho(T_x\tau_{+}(v),T_x\tau_{-}(v))\leq e^{-|v|R+\|x(\xi)\|},$$ 
by Lemma \ref{rosetta},
and thus
\[ \mathcal{L}^{(T_x\tau_{+}(v))}\cap B(0;|v|R-\|x(\xi)\|)=\mathcal{L}^{(T_x\tau_{-}(v))}\cap B(0;|v|R-\|x(\xi)\|).\]

For any two elements 
\[\alpha=(T_{x(\alpha)}r(\xi),x(\alpha))\in s^{-1}(r(\xi))\cap V_{z},\quad\beta=(T_{x(\beta)}r(\eta),x(\beta))\in s^{-1}(r(\eta))\cap V_{z},\]
we have 
\[-x(\alpha)\in \mathcal{L}^{(T_{x}\tau_{+}(v))}\cap B(z;\varepsilon),\quad -x(\beta)\in \mathcal{L}^{(T_{x}\tau_{-}(v))}\cap B(z;\varepsilon).\]
Now for $x=x(\xi)=x(\eta)\in B(0; |v|R-\|z\|-r)$ we have
$B(z;\varepsilon)\subset B(0;|v|R-\|x(\xi)\|).$
 Using Lemma \ref{rosetta} we find 
\[-x(\alpha),-x(\beta)\in \mathcal{L}^{(T_x\tau_{+}(v))}\cap B(0;|v|R-\|x(\xi)\|)=\mathcal{L}^{(T_x\tau_{-}(v))}\cap B(0;|v|R-\|x(\xi)\|),\]
and since $\|x(\alpha)-x(\beta)\|<r$ it must hold that $z:=x(\alpha)=x(\beta)$. Thus
\[\alpha=(T_{x+w}(\tau_{+}(v)),w),\quad \beta=(T_{x+w}(\tau_{-}(v)),w)\]
where $-w$ is the unique point in $\mathcal{L}^{(T_x\tau_{+}(v))}\cap B(z;\varepsilon)=\mathcal{L}^{(T_x\tau_{-}(v))}\cap B(z;\varepsilon)$. 
We then have
\[\alpha\xi=(T_{w+x}(\tau_{+}(v)),w+x),\quad \beta\eta=(T_{w+x}(\tau_{-}(v)),w+x),\]
that is $x(\alpha\xi)=x(\beta\eta)$. This proves \eqref{implication} on page \pageref{implication}.
\end{proof}
\begin{prop} Let $\chi_{k}$ be the partition of unity elements associated to the $s$-cover $\{V_y\}_{y\in Y}$. The operators
\[[T,\chi_{k}^{*}](1+X^{2})^{-\delta},\quad (1+X^{2})^{-\delta}[T,\chi_{k}^{*}]\]
extend boundedly to all of 
$L^{2}(\mathcal{G}\times_{\tau_{+}}\mathcal{V})\hat\otimes \bigwedge^* \R^d \oplus L^{2}(\mathcal{G}\times_{\tau_{-}}\mathcal{V})\hat\otimes \bigwedge^* \R^d$.
\end{prop}
\begin{proof} 
We can again ignore the finite dimensional space $\bigwedge^*\R^d$, where $X^2$ and $T$ act 
diagonally. Consider
\begin{align*}
&\langle [T,\chi_{k}^{*}](1+X^{2})^{-\delta}\phi, [T,\chi_{k}^{*}](1+X^{2})^{-\delta}\phi\rangle-\langle B\phi,B\phi\rangle\\  
&= \!\! \sum_{\substack{v\in\calV_z, \\ \eta\in s^{-1}(\tau_{-}(v))}} \! \zeta_{|v|}^{2}\left\|\sum_{\xi\in J_{(z,v)}}(1+\|x(\xi)\|^{2})^{-\delta}\sum_{y,\alpha,\beta}\left(\chi_{k}(\alpha)\chi_{y}(\beta\eta)\chi_{y}(\alpha\xi)-\chi_{k}(\beta)\chi_{y}(\eta)\chi_{y}(\xi)\right)\phi(\xi,v) \right\|^{2}\\
&\leq \! \sum_{\substack{v\in\calV_z, \\ \eta\in s^{-1}(\tau_{-}(v))}} \! \left(\sum_{\xi\in J_{(z,v)}}\left\|\sum_{y,\alpha,\beta}\left(\chi_{k}(\alpha)\chi_{y}(\beta\eta)\chi_{y}(\alpha\xi)-\chi_{z}(\beta)\chi_{y}(\eta)\chi_{y}(\xi)\right)\phi(\xi,v) \right\|\right)^{2}\\
&\leq \sum_{\substack{v\in\calV_z, \\ \eta\in s^{-1}(\tau_{-}(v))}}\sum_{\xi\in J_{(z,v)}}N\left\|\sum_{y,\alpha,\beta}\left(\chi_{z}(\alpha)\chi_{y}(\beta\eta)\chi_{y}(\alpha\xi)-\chi_{z}(\beta)\chi_{y}(\eta)\chi_{y}(\xi)\right)\right\|^{2}\|\phi(\xi,v)\|^{2},
\end{align*}
where the last inequality follows from Lemma \ref{Ncounting}.
We proceed
\begin{align*}
&\hspace{-1cm}\langle [T,\chi_{k}^{*}](1+X^{2})^{-\delta}\phi, [T,\chi_{k}^{*}](1+X^{2})^{-\delta}\phi\rangle-\langle B\phi,B\phi\rangle\\  
&\leq 2N\sum_{\substack{v\in\calV_z, \\ \eta\in s^{-1}(\tau_{-}(v))}} \sum_{\xi\in J_{(z,v)}}\sum_{y,\alpha,\beta}\left(\chi_{z}(\alpha)^{2}\chi_{y}(\alpha\xi)^{2}\chi_{y}(\beta\eta)^{2}+\chi_{z}(\beta)^{2}\chi_{y}(\xi)^{2}\chi_{y}(\eta)^{2}\right)\|\phi(\xi,v)\|^{2}\\
&\leq 2N\sum_{v\in\mathcal{V}_{z}}\sum_{\xi\in J_{(z,v)}}\sum_{\alpha,\beta}\left(\chi_{z}(\alpha)^{2}+\chi_{z}(\beta)^{2}\right)\|\phi(\xi,v)\|^{2}\\
&\leq 4N\sum_{v\in\mathcal{V}_{z}}\sum_{\xi\in J_{(z,v)}}\|\phi(\xi,v)\|^{2}\leq 4N\|\phi\|^{2},
\end{align*}
and we conclude that $[T,\chi_{k}^{*}](1+X^{2})^{-\delta}$ is bounded for all $\delta$. The statement for $(1+X^{2})^{-\delta}[T,\chi_{k}^{*}]$ follows in a similar manner.
\end{proof}
\begin{thm} \label{thm:PB_product}
The triple 
$$
\Big( C_c(\calG )\hat\otimes Cl_{0,d}, \, L^{2}(\mathcal{G}_{s}\times_{\tau_{+}}\calV)\hat\otimes 
    \bigwedge\nolimits^{\! *} \R^d\oplus L^{2}(\calG_{s}\times_{\tau_{-}}\calV)\hat\otimes 
    \bigwedge\nolimits^{\! *} \R^d, \, X+T\kappa \Big)
$$ 
is an $\varepsilon$-unbounded Fredholm module for all $0<\varepsilon <1$ in the sense of \cite[Definition A.1]{GM15}. 
It represents the Kasparov product of the class 
$[{}_{d}\lambda_{\Omega_0}] \in KK(C^*_r(\calG )\hat\otimes Cl_{0,d}, C(\Omega_0))$ 
of Equation \eqref{eq:bulk_K-cycle} on page \pageref{eq:bulk_K-cycle} 
and the quasi-homomorphism $[(\pi_{\tau_{+}}, \pi_{\tau_{-}})] \in KK(C(\Omega_0),\C)$.
\end{thm}
\begin{proof} By \cite[Theorem A.6]{GM15} the bounded transform of $X+T\kappa$ is a Fredholm module. By Proposition \ref{prop: prodcandidate} and \cite[Theorem A.3]{longitudinal} this Fredholm module represents the Kasparov product of the indicated classes.
\end{proof}

As previously mentioned, it would be interesting to compare the $K$-homology class of the 
$\varepsilon$-unbounded Fredholm module from Theorem \ref{thm:PB_product} with similar constructions 
in the tiling literature~\cite{MW17, JKS15}. 

We have concretely represented a $K$-homology class containing information of both the 
transversal dynamics and internal structure of the unit space. In Section \ref{sec:PB_pairing} we will briefly consider 
its potential applications to topological phases of lattices or tilings with finite local complexity 
(e.g. quasicrystals) via the index pairing.

\section{Index pairings and topological phases} \label{sec:Applications}
Up to now we have largely been concerned with the $K$-homology and $KK$-theory of 
$C^*_r(\calG,\sigma)$. In this section, we use these constructions and properties 
to consider homomorphisms on $K$-theory. That is, we are interested in product pairings 
in real or complex $K$-theory
\begin{equation} \label{eq:bulk_pairing}
  K_n(C^*_r(\calG, \sigma)) \times KK^d(C^*_r(\calG,\sigma),C(\Omega_0)) 
    \to K_{n-d}(C(\Omega_0)).
\end{equation}

Our motivation for studying such pairings comes 
from applications to topological phases of Hamiltonians on aperiodic lattices, which we briefly 
introduce. A low energy quantum mechanical system with negligible interactions between 
particles is modelled via a self-adjoint Hamiltonian acting on a complex separable 
Hilbert space $\calH$. This Hamiltonian is often an element of or affiliated to 
a $C^*$-algebra of observables.
 One can then consider underlying symmetries of the Hamiltonian, 
where Wigner's theorem implies that such symmetries arise on $\calH$ as projective unitary 
or anti-unitary representations of a symmetry group~\cite{Wigner}. 
In the case of an anti-unitary representation of $\Z_2$ (e.g. a time-reversal symmetry), 
conjugation by the generator of 
this representation often gives a \emph{Real structure} $\mathfrak{r}$ on the 
$C^*$-algebra of observables. That is, an anti-linear order-$2$ automorphism that 
commutes with the $\ast$-operation.

While other symmetry groups can be considered, for free-fermionic 
topological phases, one generally considers a symmetry group 
$G \subset \Z_2 \times \Z_2$. The symmetries that generate this group 
are chiral symmetry (unitary, anti-commutes with Hamiltonian), time-reversal symmetry 
(anti-unitary, commutes with Hamiltonian) and particle-hole symmetry (anti-unitary, 
anti-commutes with Hamiltonian). A key reason for studying such 
symmetries comes from the following result.

\begin{prop}[\cite{Kellendonk15}] \label{prop:tenfoldway}
Let $h$ be a self-adjoint element in the complex $C^*$-algebra $A$ 
with a spectral gap at $0$ (taking a shift $h-\mu$ if necessary).
\begin{enumerate}
  \item The spectral projection $\chi_{(-\infty,0]}(h)$ gives a class in $K_0(A)$.
  \item If $h$ has a chiral symmetry, then $h$ determines an element in $K_1(A)$.
  \item If $h$ has a time-reversal symmetry and/or particle-hole symmetry, then $h$ determines 
  an element in $KO_n(A^{\mathfrak{r}})$.
  The degree $n$ of the $KO$-theory group and the Real structure $\mathfrak{r}$ is determined by the 
  symmetry of the Hamiltonian (cf. \cite[Section 6]{Kellendonk15}).
\end{enumerate}
\end{prop}

\begin{remarks}
\begin{enumerate}
  \item Proposition \ref{prop:tenfoldway} has appeared in 
numerous forms, see~\cite{FM13, Thiang14, BCR15, Kubota15a} for example.  
\item We wish to apply Proposition \ref{prop:tenfoldway} to the case $A = C^*_r(\calG,\sigma)$ 
(complex $C^*$-algebras)
so that we can then apply our $KK$-theoretic machinery to the case of invertible Hamiltonians 
$h \in C^*_r(\calG,\sigma)$. 
For systems with no anti-linear symmetries, this is no problem. 
For systems with  time-reversal symmetry or particle-hole symmetry, 
we require the corresponding real subalgebra $C^*_r(\calG,\sigma)^{\mathfrak{r}}$ to have a 
presentation as a twisted real groupoid $C^*$-algebra, $C^*_r(\calG, \sigma_\R)_\R$ 
with $\sigma_\R$ a $O(1)$-valued twist. 
This places a restriction on the $U(1)$-valued twist $\sigma$, but is immediate if the twist is trivial. Such assumptions are to 
be expected as, for example, the case of magnetic twists from Example \ref{ex:mag_twist} 
should in general not be compatible with a time-reversal symmetry.

Hence, we shall from now on assume that our algebra of observables is given by  
the real or complex transversal groupoid $C^*$-algebra and that we have a class in $K_n(C^*_r(\calG,\sigma))$ 
(real or complex) to take pairings with.
\item Because we have an unbounded Kasparov module 
$$
  {}_{d}\lambda_{\Omega_0} = \bigg( C_c(\calG, \sigma)\hat\otimes Cl_{0,d},\, E_{C(\Omega_0)} \hat\otimes 
   \bigwedge\nolimits^{\! *} \R^d, \, X=\sum_{j=1}^d X_j \hat\otimes \gamma^j \bigg),
$$
the completion $\calA = \ol{C_c(\calG,\sigma)}$ of $C_{c}(\mathcal{G},\sigma)$ in the norm $\|f\|+\|[X,f]\|$ is a Banach $*$-algebra that is stable under the 
holomorphic functional calculus~\cite{BC91}. 
Having fixed such an algebra $\calA$, the spectral gap assumption on $h \in C^*_r(\calG,\sigma)$, 
means that we can improve Proposition \ref{prop:tenfoldway} and obtain an element of the group 
$K_n(\calA)\cong K_n(C^*_r(\calG,\sigma))$.
\end{enumerate}
\end{remarks}

Given a $K$-theory class from Proposition \ref{prop:tenfoldway}, we can thus consider 
pairings such as Equation \eqref{eq:bulk_pairing} on page \pageref{eq:bulk_pairing}. 
In general, the pairing in Equation \eqref{eq:bulk_pairing} can be described using a 
Clifford index similar to Atiyah--Bott--Shapiro~\cite{ABS}. This index then serves as an explicit 
phase label of the $K$-theory class from Proposition \ref{prop:tenfoldway}.

Also of importance  are numerical pairings, which can be defined in a few ways. 
One is by point evaluation $C(\Omega_0) \to \C$ or $\R$, which leads to 
$\Z$ or $\Z_2$-valued invariants. Alternatively, we fix a faithful and invariant 
measure on $\Omega_\calL$, which gives an invariant measure on $\Omega_0$. This defines a trace on $C(\Omega_0)$ 
and a homomorphism $K_0(C(\Omega_0)) \to \R$. 
In particular, for complex algebras, the composition
\begin{equation} \label{eq:semifinite_KK_pairing}
  K_\ast (C^*_r(\calG,\sigma)) \times KK^\ast (C^*_r(\calG,\sigma),C(\Omega_0)) 
  \to K_0(C(\Omega_0)) \xrightarrow{\int} \R
\end{equation}
can be computed using the semifinite local index formula. The cyclic formula that 
we obtain from the local index formula is then more amenable to physical 
interpretation and numerical approximation.

\subsection{Complex pairings} \label{subsec:Complex_pairings}

For complex algebras, 
we use the semifinite local index formula to pair complex $K$-theory classes in 
$K_\ast(C^*_r(\calG,\sigma))$ with the
spin$^c$ semifinite spectral triple 
from Equation \eqref{eq:spin_semifinite_spec_trip} on page \pageref{eq:spin_semifinite_spec_trip}. 
Algebraic manipulation of the Dirac operator means that only the top degree 
term survives as in~\cite[Appendix]{BCPRSW}. Then we can evaluate the 
resolvent cocycle, which uses the residue trace computation from Proposition \ref{prop:semifinite_smooth_spec_trip}.
We will simply state the result as the proof follows the same argument 
as analogous results in~\cite{BRCont,BSBWeak}.

\begin{prop} \label{prop:complex_bulk_pairing}
Let $u$ be a complex unitary in $M_q(\calA)$ 
and ${}_{d}\lambda_{\tau}^{S_\C}$ the complex semifinite
spectral triple from Equation \eqref{eq:spin_semifinite_spec_trip} on page \pageref{eq:spin_semifinite_spec_trip} with $d$ odd. 
Then the semifinite index pairing is given by the formula
$$  
\langle [u], [{}_{d}\lambda_{\tau}^{S_\C}] \rangle = \tilde{C}_d \sum_{\rho \in S_d}(-1)^\rho\, 
 (\Tr_{\C^q}\otimes \Tr_\tau) \bigg(\prod_{j=1}^d u^* \partial_{\rho(j)}u \bigg), 
 \qquad \tilde{C}_{2n+1} = \frac{ 2(2\pi i)^n n!}{(2n+1)!},
 $$
where $\Tr_{\C^q}$ is the matrix trace on $\C^q$ and $S_d$ is the permutation group on $d$ letters.

If $p$ is a  projection in $M_q(\calA)$, then 
the pairing with ${}_{d}\lambda_{\tau}^{S_\C}$ 
 with $d$ even is given by 
\begin{equation*} 
  \langle [p], [{}_{d}\lambda_{\tau}^{S_\C}]\rangle =  C_d \sum_{\rho \in S_d} 
  (-1)^\rho\, (\Tr_{\C^q}\otimes \Tr_\tau) \bigg(p \prod_{j=1}^d \partial_{\rho(j)}p \bigg), 
  \qquad C_{2n} = \frac{(-2\pi i)^{n}}{n!}.
\end{equation*}
\end{prop}

If the measure on the unit space is ergodic, then we can almost surely 
describe the semifinite index pairing 
via the usual $\Z$-valued index pairing with the evaluation spectral triple from
Proposition \ref{prop:evaluation_spec_trip}.
Namely, setting $F_X = X(1+X^2)^{-1/2}$ and 
$\Pi_q =  \frac{1}{2}(1+ F_X) \otimes 1_q$, we have for almost all $\omega\in\Omega_0$, 
\begin{align*}
\Index\big( \Pi_q \pi_\omega(u) \Pi_q + (1-\Pi_q) \big) &= \tilde{C}_d \sum_{\rho \in S_d}(-1)^\rho \, 
 (\Tr_{\C^q}\otimes \Tr_\mathrm{Vol}) \bigg(\prod_{j=1}^d \pi_\omega(u)^* [X_{\rho(j)},\pi_\omega(u)]  \bigg),   \\
   \Index\big( \pi_\omega(p) (F_X \otimes 1_q)_+ \pi_\omega(p) \big)  &=  C_d \sum_{\rho \in S_d} 
  (-1)^\rho\, (\Tr_{\C^q}\otimes \Tr_\mathrm{Vol}) \bigg(\pi_\omega(p) \prod_{j=1}^d [X_{\rho(j)},\pi_\omega(p)] \bigg), 
\end{align*}
which was proved by slightly different means in~\cite{BP17}.

\subsubsection{Weak Chern numbers}

Analogous to the construction in Section \ref{sec:semifinite_construction}, 
we can construct a semifinite spectral triple ${}_{d}\lambda_{\tau_k}^{S_\C}$ from 
the Kasparov module ${}_{d}\lambda_{k}$ via the dual trace $\Tr_{\tau_k}$ 
constructed from the trace $\tau_k$ on $C^*_r(\Upsilon_k,\sigma)$. 
This semifinite spectral triple is $QC^\infty$ and $(d-k)$-summable with 
a residue trace evaluation analogous to Proposition \ref{prop:semifinite_smooth_spec_trip}.
The semifinite pairing with ${}_{d}\lambda_{\tau_k}^{S_\C}$ represents 
the composition
\begin{equation} \label{eq:weak_chern}
  K_{d-k}(C^*_r(\calG,\sigma)) \times KK^{d-k}(C^*_r(\calG,\sigma),C^*_r(\Upsilon_k,\sigma)) 
   \to K_0(C^*_r(\Upsilon_k,\sigma)) \xrightarrow{\tau_k} \R.
\end{equation}
We again use the local index formula to compute this pairing; the 
interested reader can consult~\cite{BSBWeak}, where the proof transfers 
to this setting without issue.
\begin{prop}
The composition from Equation \eqref{eq:weak_chern} is computed by, 
for $d-k$ even and $p\in M_q(\calA)$ a projection,
$$
  \langle [p], [{}_{d}\lambda_{\tau_k}^{S_\C}] \rangle = 
      C_{d-k} \sum_{\rho \in S_{d-k}}(-1)^\rho (\Tr_\tau \otimes \Tr_{\C^q}) \bigg( p \prod_{j=k+1}^d \partial_{\rho(j)}(p) \bigg), 
       \qquad C_{2n} = \frac{(-2\pi i)^{n}}{n!}.
$$
If $d-k$ is odd and $u\in M_q(\calA)$ is unitary, then
$$
 \langle [u], [{}_{d}\lambda_{\tau_k}^{S_\C}] \rangle = 
    \tilde{C}_{d-k} \sum_{\rho\in S_{d-k}}(-1)^\rho (\Tr_\tau\otimes \Tr_{\C^q}) \bigg( \prod_{j=k+1}^d u^* \partial_{\rho(j)}(u) \bigg), 
    \qquad \tilde{C}_{2n+1} = \frac{ 2(2\pi i)^n n!}{(2n+1)!}.
$$
\end{prop}
Hence we recover and extend results from~\cite{PSBKK,BSBWeak}.

\subsection{Real pairings and analytic indices} \label{sec:real_pairings}
Our aim for this section is to define an analytic index representing the map
$$
  KO_{n}(C^*_r(\calG,\sigma)) \times KKO^d(C^*_r(\calG,\sigma), C(\Omega_0)) \to KO_{n-d}(C(\Omega_0))
$$
Suppose we are given a 
gapped Hamiltonian $h=h^*$ in a $C^*$-algebra $A$ such that $h$ is compatible with the $CT$-symmetry group 
$G \subset \{1,T,C,CT\}$. Then, following the construction in~\cite[Section 3.3]{BCR15}, one is able to 
construct, up to Morita equivalence, a finitely generated and projective module $pA^{\oplus N}_A$ with a 
representation $Cl_{n,0}\to \End^*(pA^{\oplus N}_A)$ constructed from the symmetry group $G$. 
Note that if the Hamiltonian is particle hole symmetric, then the projection $p \in M_N(A)$ is closely related, but not equal, to 
the Fermi projection $\chi_{(-\infty,E_F]}(h)$.

When we apply this construction to the transversal groupoid, we obtain the projective 
module $p C^*_r(\calG,\sigma)^{\oplus N}_{C^*_r(\calG,\sigma)}$ which, with its left $Cl_{n,0}$ action, is 
a representative of the class $[h]\in KO_n(C^*_r(\calG,\sigma))$ from Proposition \ref{prop:tenfoldway}. 
The fact that the Hamiltonian is gapped implies that this class can be represented by a projective 
module over a smooth dense subalgebra $\calA \subset C^*_r(\calG,\sigma)$.

The perspective outlined in~\cite{BCR15, GSB16} is that topological phases are measured via 
a pairing of this $K$-theory class $[h]\in KO_n(C_r^*(\calG,\sigma))$ with a dual 
element. In our case, this element is precisely the bulk $KK$-cycle from Equation \eqref{eq:bulk_K-cycle} 
on page \pageref{eq:bulk_K-cycle}.
Hence we compute the product
\begin{align*}
  &\left( Cl_{n,0}, p C^*_r(\calG,\sigma)^{\oplus N}_{C^*_r(\calG,\sigma)}, 0 \right) \hat\otimes_{C^*_r(\calG,\sigma)} 
   \,\, {}_{d}\lambda_{\Omega_0}   = 
   \left( Cl_{n,d},  p E^{\oplus N}_{C(\Omega_0)} \hat\otimes \R^{2^d}, \, p(X\otimes 1_N)p \right)
\end{align*}

Making small adjustments (that do not change the $KK$-class) if necessary, 
 we can ensure that the product $pXp$ graded-commutes with the left 
 $Cl_{n,d}$-action.  We denote by  $F_{pXp}$  the bounded 
transform of $pXp$. 
If the operator $F_{pXp}$ is a regular Fredholm operator (as characterised in~\cite[Section 4.3]{ElementsNCG}), 
then $\Ker(F_{pXp})_{C(\Omega_0)}$ is a complemented $C^*$-submodule 
of $p E^{\oplus N}_{C(\Omega_0)} \hat\otimes \bigwedge\nolimits^{\! *} \R^d$ with a 
graded left-action of $Cl_{n,d}$. Furthermore, all index-theoretic information 
of the Kasparov product 
is contained in the Clifford module $\Ker(F_{pXp})_{C(\Omega_0)}$, see~\cite[Appendix B]{BCR15}.
If $F_{pXp}$ is not regular, then we can amplify $F_{pXp}$ to a regular Fredholm operator 
at the expense that this changes the supporting model $p E^{\oplus N}\oplus C(\Omega_0)^{K}$ for some $K$. 
The physical significance of this amplification is not always clear 
and, as such, needs to be considered on a case by case basis.

We briefly summarise our argument.
\begin{prop} \label{prop:Product_as_Clifford_index}
If $F_{pXp}$ is regular, then the $C^*$-module $\Ker(F_{pXp})_{C(\Omega_0)}$ with left $Cl_{n,d}$-action 
represents the Kasparov product of the class $[h]\in KO_n(C^*_r(\calG,\sigma))$ 
with the bulk $KK$-cycle from Equation \eqref{eq:bulk_K-cycle} on page \pageref{eq:bulk_K-cycle}.
\end{prop}

Let us now associate an analytic index to the Kasparov product.

\begin{defn}
We let ${}_{r,s}{\mathfrak{M}}_{C(\Omega_0)}$ be the Grothendieck group of equivalence 
classes of real $\Z_2$-graded right-$C(\Omega_0)$ $C^*$-modules carrying a 
graded left-representation of $C\ell_{r,s}$.
\end{defn}

Provided $F_{pXp}$ is regular, the product  $\Ker(F_{pXp})$ determines a class 
in the quotient group ${}_{n,d}{\mathfrak{M}}_{C(\Omega_0)}/i^*({}_{n+1,d}{\mathfrak{M}}_{C(\Omega_0)})$, 
where $i^*$ comes from restricting a Clifford action of 
$C\ell_{n+1,d}$ to $C\ell_{n,d}$. Next, we use an 
extension of the 
Atiyah--Bott--Shapiro isomorphism, see \cite[{\S}2.3]{SchroderKTheory}, 
to make the identification
$$  
{}_{n,d}{\mathfrak{M}}_{C(\Omega_0)}/i^*{}_{n+1,d}{\mathfrak{M}}_{C(\Omega_0)} 
  \cong  KO_{n-d}(C(\Omega_0)). 
$$

\begin{defn}
If $F_{pXp}$ is regular, 
the Clifford index of $F_{pXp}$ is given by 
the class
$$ 
\Index_{n-d}(F_{pXp})= [\Ker(F_{pXp})] \in 
{}_{n,d}{\mathfrak{M}}_{C(\Omega_0)}/i^*{}_{n+1,d}{\mathfrak{M}}_{C(\Omega_0)} 
  \cong  KO_{n-d}(C(\Omega_0)). $$
\end{defn}

\begin{remark}[Range of the pairing]
In general it is a difficult task to compute $KO_{n-d}(C(\Omega_0))$ for a transversal 
set $\Omega_0$ that comes from a generic Delone set. However, if our original Delone 
lattice has finite local complexity, then $\Omega_0$ is totally disconnected (Proposition \ref{prop:tranversal_properties}),
so by the continuity of the $K$-functor,
\begin{align*}
   KO_j(C(\Omega_0)) \cong C(\Omega_0, KO_j(\R)) 
   = \begin{cases} C(\Omega_0,\Z), & j=0\,\,\mathrm{mod}\,4, \\  C(\Omega_0,\Z_2), & j=1,2\,\,\mathrm{mod}\,8, \\ 0, & \text{otherwise}. \end{cases}
\end{align*}
\end{remark}

\begin{example}[Spectral triple pairings] \label{ex:spec_trip_indexpairing}
By the evaluation map $\mathrm{ev}_\omega:C(\Omega_0)\to \R$, 
we can also pair our $K$-theory classes with the evaluation spectral triple 
${}_{d}\lambda_{\omega}$ from Proposition \ref{prop:evaluation_spec_trip},
$$
  KO_n(C^*_r(\calG)) \times KO^d(C^*_r(\calG)) \to KO_{n-d}(\R).
$$
The $\Z$ or $\Z_2$-valued indices can be measured using results 
from~\cite{ASskewadj, GSB16, KK16}.
Writing these pairings explicitly, 
$$
  [h]\hat\otimes [{}_{d}\lambda_{\omega}] = \begin{cases}  \mathrm{dim}_\R  \Ker\big( (F_{p_\omega X p_\omega})_+ \big)  - \mathrm{dim}_\R \Ker\big( (F_{p_\omega X p_\omega})_- \big),  & n-d = 0 \,\mathrm{mod}\, 8 \\
     \mathrm{dim}_\R \Ker\big( (F_{p_\omega X p_\omega})_+ \big) \,\mathrm{mod}\, 2, &  n-d = 1 \,\mathrm{mod}\, 8 \\
     \mathrm{dim}_\C \Ker\big( (F_{p_\omega X p_\omega})_+ \big) \, \,\mathrm{mod}\, 2, & n-d = 2 \,\mathrm{mod}\, 8 \\
      \mathrm{dim}_{\mathbb{H}} \Ker\big( (F_{p_\omega X p_\omega})_+ \big) - \mathrm{dim}_{\mathbb{H}} \Ker\big( (F_{p_\omega X p_\omega})_- \big),  & n-d = 4 \,\mathrm{mod}\, 8 \\
      0, & \text{otherwise}   \end{cases}
$$
under the decomposition of $F = \begin{pmatrix} 0 & F_- \\ F_+ & 0 \end{pmatrix}$ by the grading. 
By considering $\mathbb{H}$ as 
an even-dimensional complex space, the quaternionic index naturally takes values in $2\Z$. 
\end{example}

Let us also briefly remark that complex topological phase labels can also be defined via a 
Clifford index, though generally indices defined via cyclic cocycles can be more easily 
related to measurable physical phenomena.

\subsubsection{Extending the pairings} \label{subsec:localisation}
In~\cite{BP17}, complex bulk indices are extended to a larger algebra constructed 
from the noncommutative Sobolev spaces $\calW_{r,p}$, obtained as the closure of $C_c(\calG,\sigma)$ in the norms 
$$
  \| f\|_{r,p} = \sum_{|\alpha| \leq r} \Tr_\tau \Big( |\partial^\alpha f|^p \Big)^{1/p}, \quad r\in \N, \,\, p\in [1,\infty), 
  \,\, \alpha\in \N^d, \,\, \partial^\alpha = \partial_1^{\alpha_1} \cdots \partial_d^{\alpha_d}, \,\, 
  |\alpha| = \sum_{j=1}^d \alpha_j.
$$
We also consider the von Neumann algebra generated by the GNS representation of 
$C^*_r(\calG,\sigma)$ with respect to the dual trace $\Tr_\tau$, denoted by $L^\infty(\calG,\Tr_\tau)$. 
Following~\cite{BP17}, we define $\calA_\mathrm{Sob}$ as the intersection 
of $\calW_{r,p}$ for $r,p\in\N$ with $L^\infty(\calG,\Tr_\tau)$, but emphasise that 
the topology of $\calA_\mathrm{Sob}$ comes only from the Sobolev norms 
$\| \cdot \|_{r,p}$ and not the von Neumann norm (see also~\cite[Section 5]{BRCont}).

If the  measure on the continuous hull $\Omega_\calL$ is 
ergodic under the translation action, 
then $\Z$ and $\Z_2$-valued bulk topological phases can be defined over $\calA_\mathrm{Sob}$. 
For complex pairings, the Hochschild cocycle from 
the semifinite spectral triple is also well-defined for the Sobolev algebra and, as this 
cocycle represents the Chern character (because the lower-order terms vanish), 
the cyclic formulas for the index also extend to the Sobolev algebra. 
For real pairings with an ergodic measure, the analytic indices considered in 
Example \ref{ex:spec_trip_indexpairing} 
are almost surely well defined and constant over $\Omega_0$ in the 
 Sobolev setting. See~\cite{GSB16,KK16,BRCont} for a more comprehensive treatment.

For Hamiltonians acting 
on $\ell^2(\Z^d)\otimes\C^n$ without a spectral gap but instead a region 
$\Delta \subset \sigma(h)$ of dynamical localisation (which implies, amongst other properties, 
that $\Delta \cap \sigma_\text{pp}(h)$ is dense in $\Delta$), the pairings with $\calA_\mathrm{Sob}$ are  
connected to these localised regions via the Aizenman--Molchanov bound~\cite{AM93,PSBbook}. 
For the case of a general Delone set, the Hamiltonian $h\in \calA_\mathrm{Sob}$ acts on the 
family $\{\ell^2(\calL^{(\omega)})\}_{\omega\in\Omega_0}$. In the general Delone setting, 
spectral properties of the Hamiltonian are more difficult to determine. 
See~\cite{LPV07, GMRM15, RMreview} for more information.

\subsubsection{Weak indices}

Our $KK$-theoretic pairings can also be used to define analytic indices for the pairing 
with the higher codimension Kasparov modules constructed in Section \ref{sec:weak}. 
Namely, using the $KK$-cycles ${}_{d}\lambda_{k}$ from Section \ref{sec:weak} 
 we have a well-defined map, 
$$
  KO_n(C^*_r(\calG,\sigma)) \times KKO^{d-k}(C^*_r(\calG,\sigma),C^*_r(\Upsilon_k,\sigma)) 
   \to KO_{n-(d-k)}(C^*_r(\Upsilon_k,\sigma)).
$$
Once again this index can be described using Clifford modules.

\subsection{Pairings for lattices with finite local complexity} \label{sec:PB_pairing}
The complex and real pairings from the previous section can be defined for general Delone sets. 
When the underlying lattice $\calL$ used to construct the continuous hull $\Omega_\calL$ has 
finite local complexity, we can define new numerical phase labels via the $\varepsilon$-unbounded 
Fredholm module from Theorem \ref{thm:PB_product}.

Recall the unbounded operator $X+T\kappa$ on 
$\big( L^2(\calG {}_{s}\times_{\tau_+} \calV) \oplus L^2(\calG {}_{s}\times_{\tau_-} \calV) \big) \hat\otimes 
\bigwedge^*\R^d$ whose bounded transform $b(X+T\kappa)$ is Fredholm and has compact 
commutators with the representation of $C_c(\calG)$ (cf. Lemma \ref{shortcut}). 

There are well-defined index pairings for the $K$-theoretic phase of the Hamiltonian 
$[h]\in K_n(C^*_r(\calG))$ with the $K$-homology class $[b(X+T\kappa)] \in K^d(C^*_r(\calG))$ 
via a Fredholm index for complex phases and a skew-adjoint Fredholm index for real 
phases,
$$
  K_n(C^*_r(\calG)) \times K^d( C^*_r(\calG) ) \to K_{n-d}(\mathbb{F}), \qquad 
  \mathbb{F} = \R,\,\C.
$$

We emphasise that unlike the cases of $\Z$ or $\Z_2$-valued indices that can be 
defined by the evaluation map $\mathrm{ev}_\omega : \Omega_0\to \mathbb{F}$, these 
indices depend on the ultra-metric structure of the transversal. To more explicitly 
show this, we note the following result, which 
is an immediate consequence of the associativity of the Kasparov product.
\begin{prop}
The index pairing of the $K$-theoretic Hamiltonian phase $[h] \in K_n(C^*_r(\calG))$ with 
the class $[b(X+T\kappa)] \in K^d(C^*_r(\calG))$ is the same as the pairing of the 
class of the Clifford module $[\Ker(F_{pXp})] \in K_{n-d}(C(\Omega_0))$ 
from Proposition \ref{prop:Product_as_Clifford_index} with the 
Pearson--Bellissard spectral triple $[(\pi_{\tau_{+}}, \pi_{\tau_{-}})] \in K^0(C(\Omega_0))$ 
from Proposition \ref{prop: BPlog}.
\end{prop}
It is worth noting that the index paring of any fixed class $\alpha\in K_{n-d}(C(\Omega_0))$ with $[(\pi_{\tau_{+}},\pi_{\tau_{-}})]$ 
depends on only finitely many of the values $\tau_{\pm}(v)$, viewed as a pair of point evaluations. 
This follows from the fact that $K^0(C(\Omega_0))$ is generated by the classes of indicator functions 
$\chi_{p}$ of the cylinder sets $\calC_{p}$. For $|v|>|p|$, it holds that $\tau_{+}(v)\in \calC_{p}$ if and only 
if $\tau_{-}(v)\in\calC_p$, and thus 
\[[\chi_{p}]\otimes [(\pi_{\tau_{+}},\pi_{\tau_{-}})]=\sum_{|v|\leq |p|}[\chi_{p}]\otimes[(\pi_{\tau_{+}}(v),\pi_{\tau_{-}}(v))].\]

This generic observation was used in \cite[Theorem 6.3.1]{GM15} to determine 
the rational $K$-homology class of an analogous operator. This does not seem to be possible in the present context.

The physical distinction 
between the indices defined via $b(X+T\kappa)$ and the more standard bulk index pairings 
in Section \ref{subsec:Complex_pairings} and \ref{sec:real_pairings} is currently unclear
to us as well.
Another question is whether the class of the $\varepsilon$-unbounded 
Fredholm module has a finitely summable representative and, if so, whether the corresponding Chern character gives 
 additional physical information.

\subsection{The bulk-boundary correspondence}

Because our topological phases arise as pairings with 
the bulk $KK$-cycle, the 
results from Section \ref{sec:factorisation} can be used to relate pairings of differing 
dimension. Recall that we have the extension,
$$
  0 \to C^*_r(\calG\ltimes \calG/\Upsilon, \sigma) \to \calT \to C^*_r(\calG,\sigma) \to 0,
$$
with $\calG\ltimes \calG/\Upsilon$ the edge groupoid and  
$\calT$ acting on a half-infinite space. Suppose that $[h]\in K_n(C^*_r(\calG,\sigma))$ 
(real or complex) and consider the product with ${}_{d}\lambda_{\Omega_0}$. Then by 
Theorem \ref{thm:bulk-edge_main},
\begin{align*}
  [h] \hat\otimes_{C^*_r(\calG,\sigma)} [{}_{d}\lambda_{\Omega_0} ] 
   &= (-1)^{d-1} [h]\hat\otimes_{C^*_r(\calG,\sigma)} \big( [{}_{d}\lambda_{d-1}] \hat\otimes_{C^*_r(\Upsilon,\sigma)} [{}_{d-1}\lambda_{\Omega_0}] \big) \\  
    & = (-1)^{d-1} \big([h]\hat\otimes_{C^*_r(\calG,\sigma)} [{}_{d}\lambda_{d-1}] \big) \hat\otimes_{C^*_r(\Upsilon,\sigma)} [{}_{d-1}\lambda_{\Omega_0}] \\
   &= (-1)^{d-1} \partial[h]   \hat\otimes_{C^*_r(\Upsilon,\sigma)} [{}_{d-1}\lambda_{\Omega_0}]
\end{align*}
with $\partial[h] \in KO_{n-1}(C^*_r(\calG\ltimes \calG/\Upsilon, \sigma))$ the image under 
the boundary map in $K$-theory.
That is,  the pairing with respect to the bulk algebra $C^*_r(\calG,\sigma)$ is non-trivial 
if and only if the pairing $\partial[h]   \hat\otimes_{C^*_r(\Upsilon,\sigma)} [{}_{d-1}\lambda_{\Omega_0}]$ 
over the edge algebra $C^*_r(\Upsilon,\sigma)$ (or $C^*_r(\calG\ltimes \calG/\Upsilon,\sigma)$) 
is non-trivial.

Furthermore, because the semifinite pairings involve the Kasparov product and then the trace, 
the bulk-edge correspondence also holds for the Chern number formulas (using the 
Morita equivalence between spin$^c$ and oriented structures). Using the 
notation from Proposition \ref{prop:complex_bulk_pairing},
\begin{align*}
   &\langle [p], [{}_{d}\lambda_{\tau}^{S_\C}] \rangle = - \langle \partial[p], [{}_{d-1}\lambda_{\tau}^{S_\C}] \rangle,  
   &&\langle [u], [{}_{d}\lambda_{\tau}^{S_\C}] \rangle = \langle \partial[u], [{}_{d-1}\lambda_{\tau}^{S_\C}] \rangle.
\end{align*}

Similarly our weak or higher codimension pairings also factorise by 
Theorem \ref{thm:weak_factorisation2}. Namely, via the short exact sequence, 
\begin{equation} \label{eq:weak_extension}
  0 \to C^*_r(\Upsilon_k \ltimes \Upsilon_k/\Upsilon_{k-1}, \sigma) \to 
   \calT_k
    \to C^*_r(\Upsilon_k,\sigma) \to 0,
\end{equation}
we have the equality of pairings,
\begin{align*}
  [h] \hat\otimes_{C^*_r(\calG,\sigma)}[{}_{d}\lambda_{k-1}] &= 
    (-1)^{d-k} [h] \hat\otimes_{C^*_r(\calG,\sigma)} \big([{}_{d}\lambda_{k}] \hat\otimes_{C^*_r(\Upsilon_k,\sigma)} [{}_k\lambda_{k-1}] \big) \\
    &= (-1)^{d-k} \big( [h] \hat\otimes_{C^*_r(\calG,\sigma)}[{}_{d}\lambda_{k}]\big) \hat\otimes_{C^*_r(\Upsilon_k,\sigma)} [{}_k\lambda_{k-1}]  \\
    &= (-1)^{d-k} \partial \big( [h] \hat\otimes_{C^*_r(\calG,\sigma)}[{}_{d}\lambda_{k}]\big).
\end{align*}
That is, our weak pairing $[h] \hat\otimes_{C^*_r(\calG,\sigma)}[{}_{d}\lambda_{k}]$ takes values in 
$K_{n-d+k}(C^*_r(\Upsilon_k,\sigma))$ and if we apply the boundary map coming from the short 
exact sequence in Equation \eqref{eq:weak_extension} on page \pageref{eq:weak_extension}, then up to a sign we obtain the 
weak pairing $[h] \hat\otimes_{C^*_r(\calG,\sigma)}[{}_{d}\lambda_{k-1}] \in K_{n-d+k-1}(C^*_r(\Upsilon_{k-1},\sigma))$. 
Of course this equality is not necessarily related to the presence of a boundary and is more 
a property of the Kasparov modules that we use to define the weak topological phases.

For lattices with finite local complexity, we also obtain a factorisation 
of index pairings of the class 
$[h] \in K_n(C^*_r(\calG,\sigma))$ with the $\varepsilon$-unbounded Fredholm module 
from Theorem \ref{thm:PB_product}, 
\begin{align*}
   [h] \hat\otimes_{C^*_r(\calG,\sigma)} \big( [{}_{d} \lambda_{\Omega_0} ] \hat\otimes_{C(\Omega_0)} [(\pi_{\tau_{+}}, \pi_{\tau_{-}})] \big)
    &= (-1)^{d-1}  \partial[h] \hat\otimes_{C^*_r(\Upsilon,\sigma)} 
       \big( [{}_{d-1} \lambda_{\Omega_0} ] \hat\otimes_{C(\Omega_0)} [(\pi_{\tau_{+}}, \pi_{\tau_{-}})] \big),
\end{align*}
where the right-hand side is a pairing 
$K_{n-1}(C^*_r(\Upsilon))\times K^{d-1}(C^*_r(\Upsilon))\to K_{n-d}(\R)$ (or complex).

\subsection{Examples from materials science and meta-materials}

Constructing model Hamiltonians for generic Delone sets is in general a difficult 
task, particularly if the underlying lattice is amorphous. However, given 
$\omega\in\Omega_0$,
we can write down a basic Hamiltonian by coupling lattice sites with exponential 
decay  and twisting by a magnetic flux,
$$
  (H_\omega \psi)(x) = \sum_{y\in\calL^{(\omega)}} e^{-i \Gamma_{\calL^{(\omega)}}\langle 0, x,y \rangle} 
     e^{-\beta|x-y|} \psi(y), \qquad \beta>0,\,\, \psi \in \ell^2(\calL^{(\omega)}).
$$
There is some element $h \in C^*_r(\calG,\sigma)$ such that $\pi_\omega(h) = H_\omega$ 
using the point-wise representation. 
If $\Delta\subset\R$ is a spectral gap of $h$, then the spectral projection 
$P_E= \chi_{(-\infty,E]}(h)$ is in the smooth $\ast$-subalgebra 
$\calA\subset C^*_r(\calG,\sigma)$  for any $E\in \Delta$. 
One of the advantages of introducing a magnetic flux into the Hamiltonian is that it 
can potentially open gaps in the spectrum of $h$, as is required by our assumptions. Let us also 
remark that our choice of Hamiltonian can also be used to model mechanical or 
gyroscopic meta-materials provided the energies are low, see~\cite{GyroInsulator} 
for example. 

In the setting of a spectral gap, our results give that for $d$ even,
\begin{align} \label{eq:cyclicbb_even}
   C_{d} \sum_{\mu\in S_d}(-1)^{\mu} \Tr_\tau\Big( P_E \prod_{j=1}^d \partial_{\mu(j)}(P_E) \Big) 
     &= -C_{d-1} \sum_{\nu \in S_{d-1}}(-1)^\nu (\Tr_\tau^\Upsilon \otimes \Tr_\calH)
          \Big( \prod_{j=1}^{d-1} \hat{u}_h^* \partial_{\nu(j)}(\hat{u}_h) \Big)
\end{align}
where $\hat{u}_h = e^{2\pi i f(\Pi_d h\Pi_d)}$ 
with $\Pi_d$ the projection onto a half-space and 
$f$ a function that smoothly goes from $0$ to $1$ inside the spectral gap $\Delta$. We also 
use that $C^*_r(\calG \ltimes \calG/\Upsilon,\sigma) \cong C^*_r(\Upsilon,\sigma)\otimes \mathbb{K}(\calH)$ and 
so our boundary semifinite pairing can be written using the semifinite spectral triple 
over $C^*_r(\Upsilon,\sigma)\otimes \mathbb{K}(\calH)$ relative to $\Tr_\tau^\Upsilon \otimes \Tr_\calH$. 

If $0\notin\sigma(h)$ (shifting by a constant term if necessary) and 
$R_ChR_C = -h$ for some self-adjoint unitary $R_C\in \calA$, then we can define 
the Fermi unitary $U_F = \frac{1}{2}(1-R_C)(1-2P_F)\frac{1}{2}(1+R_C)$ with $[U_F] \in K_1(\calA)$. 
Then, for $d$ odd
\begin{align} \label{eq:cyclicbb_odd}
  C_{d} \sum_{\mu \in S_{d}}(-1)^\mu \Tr_\tau\Big( \prod_{j=1}^d U_F^* \partial_{\mu(j)}U_F \Big) 
    &= C_{d-1} \sum_{\nu\in S_{d-1}} (-1)^\nu (\Tr_\tau^\Upsilon \otimes \Tr_\calH) 
          \Big( \mathrm{Ind}(U_F) \prod_{j=1}^{d-1} \partial_{\nu(j)}\mathrm{Ind}(U_F) \Big)
\end{align}
with $\mathrm{Ind}(U_F)$ the index map in complex $K$-theory.

If the measure on $\Omega_\calL$ is ergodic under the translation action, we can 
replace the dual trace in the left-hand side Equation \eqref{eq:cyclicbb_even} and \eqref{eq:cyclicbb_odd} 
with the trace per unit volume on $\ell^2(\calL^{(\omega)})$ for almost all $\omega\in\Omega_0$. 
In this setting the bulk cyclic formulas continue to be well-defined and integer-valued if the 
assumption on $\Delta$ is relaxed to a mobility gap (as characterised in~\cite[Section 6.2]{BP17}).

We can implement other $CT$-symmetries on $h$ by a choice of Real structure on 
$C^*_r(\calG,\sigma)$. 
Because the equation for $h$ is quite generic, such symmetries and invariants can 
be implemented by passing to matrices over $C^*_r(\calG,\sigma)$. 
The corresponding bulk and boundary pairings are described 
in Section \ref{sec:real_pairings}, though let us note that if there is no magnetic flux 
(such as in time-reversal 
symmetric Hamiltonians), then a model Hamiltonian with spectral gap may be difficult 
to construct for a generic Delone lattice. However, gaps in the spectrum without 
a magnetic field may be possible by considering more ordered (but still aperiodic) 
lattices coming from quasicrystals or substitution tilings.


\begin{thebibliography}{9}

\bibitem{AM93}
M. Aizenman and S. Molchanov. \emph{Localization at large disorder and at extreme energies: An elementary derivation}. 
Comm. Math. Phys. \textbf{157} (1993), no. 2, 245--278.

\bibitem{AndersonPutnam}
J. E. Anderson and I. F. Putnam. \emph{Topological invariants for substitution tilings and their
  associated {$C^*$}-algebras}. Ergodic Theory Dynam. Systems, \textbf{18} (1998), no. 3, 509--537.

\bibitem{ADL16}
F. Arici, F. D'Andrea and G. Landi. \emph{Pimsner algebras and circle bundles}. In 
D. Alpay, F. Cipriani, F. Colombo, D. Guido, I. Sabadini and J.-L. Sauvageot, editors, 
Noncommutative Analysis, Operator Theory and Applications, pages 1--25, 
Springer International Publishing (2016).

\bibitem{AriciKaadLandi}
F. Arici, J. Kaad and G. Landi. \emph{Pimsner algebras and Gysin sequences from principal 
circle actions} J. Noncommut. Geom., \textbf{10} (2016), no. 1, 29--64.

\bibitem{ABS} 
  M. F. Atiyah, R. Bott, and A. Shapiro. \emph{Clifford modules}. {Topology}, 
  \textbf{3} (1964), suppl. 1, 3--38.

\bibitem{ASskewadj}
  M. F. Atiyah and I. M. Singer. \emph{Index theory for skew-adjoint Fredholm operators}. {Inst. Hautes \'{E}tudes Sci. Publ. Math.}, 
  \textbf{37} (1969), 5--26.
  
\bibitem{BJ83}
S. Baaj and P. Julg. \emph{Th\'eorie bivariante de {K}asparov et op\'erateurs non born\'es dans les {$C^{\ast} $}-modules hilbertiens}.
C. R. Acad. Sci. Paris S\'er. I Math., \textbf{296} (1983), no. 21, 875--878.

\bibitem{quasiphase}
M. A. Bandres, M. C. Rechtsman and M. Segev. 
\emph{Topological photonic quasicrystals: fractal topological spectrum and protected transport}. 
Phys. Rev. X, \textbf{6} (2016), 011016.

\bibitem{BBD18}
S. Beckus, J. Bellissard and G. De Nittis. \emph{Spectral continuity for aperiodic quantum 
systems I. General theory}. J. Funct. Anal., \textbf{275} (2018), no. 11, 2917--2977.

\bibitem{Bel86}
J. Bellissard. \emph{{$K$}-theory of {$C^\ast$}-algebras in solid state physics}. In 
Statistical mechanics and field theory: mathematical aspects ({G}roningen, 1985), volume 257 
of Lecture Notes in Phys., Springer, Berlin (1986), 99--156.

\bibitem{BellissardDelone}
J. Bellissard. \emph{Delone sets and materials science: a program}. In {Mathematics of aperiodic order}, volume 309 
of Progr. Math., Birkh\"auser/Springer, Basel (2015), 405--428.

\bibitem{BBG06}
J. Bellissard, R. Benedetti and J.-M. Gambaudo. \emph{Spaces of tilings, finite telescopic approximations and gap-labeling.} 
Comm. Math. Phys., \textbf{261} (2006), no. 1, 1--41.

\bibitem{Bellissard94}
J.~Bellissard, A.~van~Elst and H.~Schulz-Baldes. 
\emph{The non-commutative geometry of the quantum Hall-effect}.
J. Math. Phys. {\bf 35} (1994), 5373--5451.

\bibitem{BHZ00} 
J. Bellissard, D. J. L. Herrmann and M. Zarrouati. \emph{Hulls of 
aperiodic solids and gap labelling theorems}. Directions in Mathematical Quasicrystals.
 Volume 13 of CIRM Monograph Series (2000), pp207--259.


\bibitem{BLM13}
F. Belmonte, M. Lein, and M. M{\u{a}}ntoiu. \emph{Magnetic twisted actions on general 
abelian $C^*$-algebras}. {J. Operator Theory}, \textbf{69} (2013), no. 1, 33--58.

\bibitem{BCPRSW}
M. Benameur, A. L. Carey, J. Phillips, A. Rennie, F. A. Sukochev, and K. P. Wojciechowski. 
\emph{An analytic approach to spectral flow in von Neumann algebras.}
In B. Boo{\ss}-Bavnbek, S. Klimek, M. Lesch, and W. Zhang, editors, 
{Analysis, Geometry and Topology of Elliptic Operators}, pages 297--352. World Scientific Publishing (2006).

\bibitem{Blackadar}
B. Blackadar. \emph{$K$-Theory for Operator Algebras}. Volume 5 of {Mathematical Sciences Research Institute Publications}, 
Cambridge Univ. Press, Cambridge (1998). 

\bibitem{BC91}
B. Blackadar and J. Cuntz. \emph{Differential Banach algebra norms and smooth 
subalgebras of $C^*$-algebras}. J. Operator Theory, \textbf{26} (1991), 255--282.

\bibitem{BCR15}
C. Bourne, A. L. Carey and A. Rennie. \emph{A non-commutative framework for topological insulators}. 
Rev. Math. Phys., \textbf{28} (2016), no. 2, 1650004.

\bibitem{BKR1}
C. Bourne, J. Kellendonk and A. Rennie. \emph{The $K$-theoretic bulk-edge correspondence for topological insulators}. 
Ann. Henri Poincar\'{e}, \textbf{18} (2017), no. 5, 1833--1866.

\bibitem{BP17}
C. Bourne and E. Prodan. \emph{Non-commutative Chern numbers for generic aperiodic discrete systems}. 
J. Phys. A, \textbf{51} (2018), no. 23, 235202.

\bibitem{BRCont}
C. Bourne and A. Rennie. \emph{Chern numbers, localisation and the bulk-edge correspondence
for continuous models of topological phases}. Math. Phys. Anal. Geom., \textbf{21} (2018), no. 3, 16.  

\bibitem{BSBWeak}
C. Bourne and H. Schulz-Baldes. \emph{Applications of semifinite index theory to weak
topological phases}. In D. Wood, J. de Gier, C. Praeger, T. Tao, editors, 
{2016 Matrix Annals}, Springer (2018).


\bibitem{CNNR}
A. L. Carey, S. Neshveyev, R. Nest and A. Rennie. \emph{Twisted cyclic theory, equivariant {$KK$}-theory and {KMS} states}. 
J. Reine Angew. Math., \textbf{650} (2011), 161--191.

\bibitem{CPRS2}
A. L. Carey, J. Phillips, A. Rennie and F. A. Sukochev. \emph{The local index formula in semifinite von 
Neumann algebras i: Spectral flow}. {Adv. Math.}, \textbf{202} (2006), no. 2, 451--516.

\bibitem{CPRS3}
A. L. Carey, J. Phillips, A. Rennie and F. A. Sukochev. \emph{The local index formula in semifinite von 
Neumann algebras ii: The even case}. {Adv. Math.}, \textbf{202} (2006), no. 2, 517--554.

\bibitem{ConnesThom}
A. Connes. \emph{An analogue of the {T}hom isomorphism for crossed products of a {$C^{\ast} $}-algebra by an action of {${\bf R}$}}.
Adv. in Math. \textbf{39} (1981), no. 1, 31--55.

\bibitem{longitudinal} 
A. Connes and G. Skandalis, \emph{The longitudinal index theorem for foliations}. 
Publ. RIMS, Kyoto Univ., \textbf{20} (1984), 1139--1183.

\bibitem{Cuntznewlook} J. Cuntz, \emph{A new look at $KK$-theory}, 
$K$-theory, \textbf{1} (1987), 31--51.

\bibitem{Daenzer09}
C. Daenzer. \emph{A groupoid approach to noncommutative {$T$}-duality}. Comm. Math. Phys. 
\textbf{288} (2009), no. 1, 55--96.

\bibitem{EM18}
E. E. Ewert and R. Meyer. \emph{Coarse geometry and topological phases}. arXiv:1802.05579 (2018).

\bibitem{FK}
T. Fack, H. Kosaki. \emph{Generalised $s$-numbers of 
$\tau$-measurable operators}. {Pac. J. Math.},
{\bf 123} (1986), no. 2, 269--300.

\bibitem{FHK}
A. Forrest, J. Hunton and J. Kellendonk. \emph{Topological invariants for projection method patterns}. 
Mem. Amer. Math. Soc., \textbf{159} (2002), no. 758, x+120.


\bibitem{FM13}
D. S. Freed and G. W. Moore. \emph{Twisted equivariant matter}. 
{Ann. Henri Poincar\'{e}}, \textbf{14} (2013), no. 8, 1927--2023.

\bibitem{GMRM15}
F. Germinet, P. M\"{u}ller and C. Rojas-Molina. \emph{Ergodicity and dynamical localization for {D}elone-{A}nderson operators}. 
Rev. Math Phys. \textbf{27} (2015), no. 9, 1550020.

\bibitem{GM15}
M. Goffeng and B. Mesland. \emph{Spectral triples and finite summability on {C}untz-{K}rieger algebras}. 
Doc. Math., \textbf{20} (2015), 89--170.

\bibitem{GMR}
M. Goffeng, B. Mesland and A. Rennie. \emph{Shift-tail equivalence and an unbounded representative of the 
{C}untz-{P}imsner extension}. Ergodic Theory Dynam. Systems, \textbf{38} (2018), no. 4, 1389--1421.

\bibitem{GT18}
K. Gomi and G. C. Thiang. \emph{Crystallographic bulk-edge correspondence: glide reflections and twisted mod 2 indices}. 
Lett. Math. Phys., online first (2018), \url{https://doi.org/10.1007/s11005-018-1129-1}.

\bibitem{TilingKTheory}
D. Gon\c{c}alves and M. Ramirez-Solano. \emph{On the $K$-theory of $C^*$-algebras for substitution tilings (a pedestrian version)}.
arXiv:1712.09551 (2017).

\bibitem{GSB16}
J. Grossmann and H. Schulz-Baldes. \emph{Index pairings in presence of 
symmetries with applications to topological insulators}. 
{Comm. Math. Phys.}, \textbf{343} (2016), no. 2, 477--513.

\bibitem{ElementsNCG}
J. M. Gracia-Bond\'{i}a, J. C. V\'{a}rilly and H. Figueroa. \emph{Elements of Noncommutative Geometry}. 
Birkh\"{a}user Advanced Texts Basler Lehrb\"{u}cher. Birkh\"{a}user, Boston, (2001).

\bibitem{HMT} 
K. Hannabuss, V. Mathai and G. C. Thiang. \emph{T-duality simplifies bulk-boundary correspondence: the noncommutative case.} 
Lett. Math. Phys., \textbf{108} (2018), no. 5, 1163--1201.

\bibitem{HuntonTiling} 
J. Hunton. \emph{Spaces of projection method patterns and their cohomology}. 
In {Mathematics of aperiodic order}, volume 309 
of Progr. Math., Birkh\"auser/Springer, Basel (2015),  105--135.

\bibitem{Kasparov80}
  G. G. Kasparov. \emph{The operator $K$-functor and extensions of $C^*$-algebras}. 
  {Math. USSR Izv.}, \textbf{16} (1981), 513--572.
  
  
\bibitem{KasparovNovikov}
G. G. Kasparov. \emph{Equivariant $KK$-theory and the Novikov conjecture}. 
{Invent. Math.}, \textbf{91} (1988), no. 1, 147--201.

\bibitem{KK16}
H. Katsura and T. Koma. \emph{The noncommutative index theorem and the periodic table for disordered 
topological insulators and superconductors}. J. Math. Phys., \textbf{58} (2018), no. 3, 031903.

\bibitem{Kellendonk95}
J. Kellendonk \emph{Noncommutative geometry of tilings and gap labelling}.
Rev. Math. Phys., \textbf{7} (1995), 1133--1180.

\bibitem{Kellendonk97}
J. Kellendonk. \emph{The local structure of tilings and their integer group of coinvariants}.
Comm. Math. Phys. \textbf{187} (1997), 115--157.

\bibitem{Kellendonk15}
J. Kellendonk. \emph{On the $C^*$-algebraic approach to topological 
phases for insulators}. {Ann. Henri Poincar\'e}, 
\textbf{18} (2017), no. 7, 2251--2300.

\bibitem{KellendonkPutnam}
J. Kellendonk and I. Putnam. \emph{Tilings, {$C^*$}-algebras, and {$K$}-theory}. 
In {Directions in mathematical quasicrystals}, Volume 13 of CRM Monogr. Ser., 
pages 177--206. Amer. Math. Soc., Providence, RI (2000).

\bibitem{KR08}
 J. Kellendonk and S. Richard. \emph{Topological boundary maps in physics}. 
 In F. Boca, R. Purice and \c{S}. Str\u{a}til\u{a}, editors, {Perspectives in operator algebras and mathematical physics}. 
 Theta Ser. Adv. Math., volume 8, pages 105--121 Theta, Bucharest (2008). arXiv:math-ph/0605048.

\bibitem{KS02}
M. Khoshkam and G. Skandalis. \emph{Regular representation of groupoid $C^*$-algebras
and applications to inverse semigroups}. {J. Reine Angew. Math.}, \textbf{546} (2002), 47--72.

\bibitem{Kubota15a}
Y. Kubota. \emph{Notes on twisted equivariant $K$-theory 
for $C^*$-algebras}. {Int. J. Math.}, \textbf{27} (2016), 1650058.

 \bibitem{Kubota15b}
Y. Kubota. \emph{Controlled topological phases and bulk-edge correspondence}. 
{Comm. Math. Phys.}, \textbf{349} (2017), no. 2, 493--525.

\bibitem{Kucerovsky97}
  D. Kucerovsky. \emph{The $KK$-product of unbounded modules}. 
  {$K$-Theory}, \textbf{11} (1997), 17--34.

\bibitem{KumjianDiagonal}
A. Kumjian, \emph{On $C^*$-diagonals}. Canad. J. Math. \textbf{38} (1986), 969--1008.

\bibitem{JKS15}
A. Julien, J. Kellendonk and J. Savinien. \emph{On the noncommutative geometry of tilings}. 
In {Mathematics of aperiodic order}, volume 309 
of Progr. Math., Birkh\"auser/Springer, Basel (2015), 259--306.

\bibitem{LN04}
M. Laca and S. Neshveyev. \emph{KMS states of quasi-free dynamics on Pimsner algebras.} 
{J. Funct. Anal.}, \textbf{211} (2004), no. 2, 457--482.


\bibitem{Lagragias03}
J. C. Lagragias and P. A. B. Pleasants. \emph{Repetitive Delone sets and quasicrystals}, 
Ergodic Theory Dynam. Systems, \textbf{23} (2003), no. 3, 831--867.

\bibitem{Lance}
E. C. Lance. \emph{Hilbert $C^*$-modules: A toolkit for operator algebraists}. Volume 210 of 
London Mathematical Society Lecgture Note Series, Cambridge Univ. Press, Cambridge (1995).

\bibitem{Spingeo}
H. B. Lawson and M.-L. Michelsohn. \emph{Spin geometry}. Volume 38 of 
Princeton Mathematical Series, Princeton University Press, Princeton, NJ (1989).

\bibitem{LPV07}
D. Lenz, N. Peyerimhoff and I. Veseli\'{c}. \emph{Groupoids, von {N}eumann algebras and the integrated density of states}. 
Math. Phys. Anal. Geom. \textbf{10} (2007), no. 1, 1--41.

\bibitem{MesLesch}
M. Lesch and B. Mesland. \emph{Sums of regular selfadjoint operators in Hilbert $C^*$-modules}. 
J. Math. Anal. Appl., to appear. arXiv:1803.08295 (2018).

\bibitem{MW17}
M. Mampusti and M. Whittaker. \emph{Fractal spectral triples on {K}ellendonk's {$C^*$}-algebra of a substitution tiling}.
J. Geom. Phys., \textbf{112} (2017), 224--239.

\bibitem{MeslandGpoid}
B. Mesland. \emph{Groupoid cocycles and $K$-theory}. M\"{u}nster J. of Math. 
\textbf{4} (2011), 227--250.

\bibitem{MR}
B. Mesland and A. Rennie, \emph{Nonunital spectral triples and metric completeness in unbounded $KK$-theory}, 
J. Funct. Anal., \textbf{271} (2016), no. 9, 2460--2538.

\bibitem{SOU} M. Macho Stadler and M. O'Uchi, \emph{Correspondence of groupoid $C^{*}$-algebras}, J. Operator Theory \textbf{42} (1999), 103--119.

\bibitem{GyroInsulator}
N. P. Mitchell, L. M. Nash, D. Hexner, A. M. Turner and W. T. M. Irvine. 
\emph{Amorphous topological insulators constructed from random point sets}. 
Nature Physics, \textbf{14} (2018), 380--385.

\bibitem{MoutuouThesis}
E. M. Moutuou. \emph{Twisted groupoid $KR$-theory}. PhD thesis, Universit\'{e} de Lorraine, Universit\"{a}t Paderborn (2012).

\bibitem{MT11}
E. M. Moutuou and J.-L. Tu. \emph{Equivalence of fell systems and their reduced groupoid $C^*$-algebras}. 
arXiv:1101.1235 (2011).

\bibitem{MRW}
P. S. Muhly, J. Renault and D. P. Williams. \emph{Equivalence and isomorphism for groupoid 
 $C^{\ast}$-algebras}. J. Operator Theory, \textbf{17} (1987), 3--22.
 
\bibitem{MuhlyWilliams}
P. S. Muhly and D. P. Williams. \emph{Renault's equivalence theorem for groupoid crossed products}. 
Volume 3 of NYJM Monographs, State University of New York, University at Albany, Albany (2008).

\bibitem{PR89}
J. A. Packer and I. Raeburn. \emph{Twisted crossed products of $C^*$-algebras}. 
{Math. Proc. Cambridge Philos. Soc.}, \textbf{106} (1989), 293--311.

\bibitem{PearsonBellissard}
J. Pearson and J. Bellissard. \emph{Noncommutative {R}iemannian geometry and diffusion on ultrametric {C}antor sets}. 
J. Noncommut. Geom., \textbf{3} (2009), no. 3, 447--480.

\bibitem{PutnamSpielberg}
I. F.  Putnam and J. Spielberg. \emph{The structure of {$C^\ast$}-algebras associated with hyperbolic dynamical systems}. 
J. Funct. Anal. \textbf{163} (1999), no. 2, 279--299.

\bibitem{Prodanbook}
E. Prodan. \emph{A computational non-commutative geometry program for disordered topological insulators}. 
Springer, Cham (2017).

\bibitem{PSBbook}
E. Prodan and H. Schulz-Baldes. \emph{Bulk and Boundary Invariants for Complex Topological Insulators: From $K$-Theory to Physics}. 
Springer, Berlin (2016).

\bibitem{PSBKK}
E. Prodan and H. Schulz-Baldes. \emph{Generalized Connes-Chern characters in $KK$-theory with an 
application to weak invariants of topological insulators.} {Rev. Math. Phys.}, \textbf{28} (2016), 1650024.

\bibitem{Renault80}
J. Renault. \emph{A groupoid approach to $C^*$-algebras}. 
Lecture Notes in Mathematics, vol. 793, Springer-Verlag (1980).


\bibitem{RRS}
A. Rennie, D. Robertson and A. Sims. \emph{The extension class and {KMS} states for {C}untz-{P}imsner algebras 
of some bi-{H}ilbertian bimodules}. J. Topol. Anal., \textbf{9} (2017), no. 2, 297--327.

\bibitem{RRSgpoid}
A. Rennie, D. Robertson and A. Sims. \emph{Groupoid algebras as {C}untz-{P}imsner algebras}. 
Math. Scand., \textbf{120} (2017), no. 1, 115--123.

\bibitem{Rieffel82}
M. Rieffel. \emph{Connes' analogue for  crossed products of the Thom isomorphism}. 
{Contemp. Math.}, \textbf{10} (1982), 143--154.

\bibitem{RMreview}
C. Rojas-Molina. \emph{Random Schr\"{o}dinger Operators on discrete structures}. 
arXiv:1710.02293 (2017).

\bibitem{SadunBook}
L. Sadun. \emph{Topology of tiling spaces}. Volume 46 of University Lecture Series, 
American Mathematical Society, Providence (2008).

\bibitem{SW03}
L. Sadun and R. W. Williams. \emph{Tiling spaces are {C}antor set fiber bundles}. 
Ergodic Theory Dynam. Systems., \textbf{23} (2003), no. 1, 307--316.

\bibitem{SavThesis}
J. Savinien. \emph{Cohomology and $K$-theory of aperiodic tilings}. PhD thesis, Georgia Institute of Technology (2008).

\bibitem{BelSav}
J. Savinien and J. Bellissard. \emph{A spectral sequence for the {$K$}-theory of tiling spaces}. 
{Ergodic Theory Dynam. Systems}, \textbf{29} (2009), no. 3, 997--1031.

\bibitem{SchroderKTheory}
H. Schr\"{o}der. \emph{$K$-theory for real $C^*$-algebras and applications}. Longman Scientific \& Technical, Harlow; 
copublished in the United States with John Wiley \& Sons Inc., New York, (1993).

\bibitem{SimsNotes}
A. Sims. \emph{\'{E}tale groupoids and their $C^*$-algebras}. 
To appear in G. Szabo, D. Williams and A. Sims, editors, 
{Operator algebras and dynamics: groupoids, crossed products and Rokhlin dimension}, 
Birkh\"{a}user. arXiv:1710.10897 (2017).





\bibitem{SimsWilliams}
A. Sims and D. P. Williams. \emph{Renault equivalence Theorem for reduced groupoid 
$C^*$-algebras}. J. Operator Theory \textbf{68} (2012), no. 1, 223--239.

\bibitem{SimsWilliamsFell}
A. Sims and D. P. Williams. \emph{An equivalence theorem for reduced Fell bundle $C^*$-algebras}. 
New York J. Math., \textbf{19} (2013), 159--178.

\bibitem{SimsYeend}
A. Sims and T. Yeend. 
\emph{$C^*$-algebras associated to product systems of Hilbert bimodules}. 
J. Operator Theory \textbf{64} (2010), no. 2, 349--376.


\bibitem{Thiang14}
G. C. Thiang. \emph{On the K-theoretic classification of topological phases of matter}. {Ann. Henri Poincar\'{e}}, 
\textbf{17} (2016), no. 4, 757--794.

\bibitem{Wigner}
E. P. Wigner. \emph{Group theory: and its application to the quantum mechanics of atomic spectra}. 
Expanded and improved ed. Translated from the German by J. J. Griffin. 
Volume 5 of Pure and Applied Physics, Academic Press, New York (1959).


\bibitem{WilliamsonThesis}
P. Williamson. \emph{Cuntz--Pimsner algebras associated with substitutuion tilings}. PhD thesis, Unviersity of Victoria (2016).

\end{thebibliography}
\end{document}